\documentclass[12pt,a4paper]{amsart}

\usepackage{tikz}
\tikzset{
op/.style = {shape=rectangle, rounded corners, draw=black, align=center,
fill=blue!30},base align tikzpicture/.style={execute at end picture={
\path (current bounding box.north) -- (current bounding box.south)
node[midway](Xphantom){\phantom{X}};},baseline={(Xphantom.base)}}
        }
\usepackage{tikz-cd}

\usepackage{graphicx} 
\usepackage{hyperref}
\usepackage[notref,notcite]{}
\hypersetup{
    colorlinks=true,
    linkcolor=blue,
    filecolor=magenta,     
    urlcolor=cyan,
    pdftitle={Codimbraid},
    }
\usepackage{amssymb,amsmath,amsthm,amsfonts,mathrsfs,bbm}
\usepackage{array,amssymb,amsxtra} 
\usepackage{geometry}
\usepackage[shortlabels]{enumitem} 
\usepackage[normalem,normalbf]{ulem}
\usepackage[all]{xy} 
\usepackage[american]{babel}
\usepackage{fancybox}
\usepackage{diagbox}
\usepackage{thmtools}
\usepackage{cite}
\usepackage{scalerel}
\usepackage{stackengine,wasysym}

\usepackage{stmaryrd}
\usepackage{MnSymbol}

\newcommand{\Z}{\mathbb{Z}} 
\newcommand{\PP}{\mathcal{P}}
\newcommand{\N}{\mathbb{N}}

\newcommand{\CC}{\mathcal{C}}

\newcommand{\DD}{\mathcal{D}}
\newcommand{\MM}{\mathcal{M}}
\newcommand{\m}{\mathscr{M}}

\newcommand{\OO}{\mathscr{O}_{p, q}(\normalfont\text{SL}_{2}(k))}

\DeclareMathOperator{\Hom}{Hom}
\DeclareMathOperator{\id}{id}

\DeclareMathOperator{\e}{Ext}
\usepackage{thmtools}
\DeclareMathOperator{\pd}{pd}
\DeclareMathOperator{\lgd}{l.gldim}
\DeclareMathOperator{\rgd}{r.gldim}
\DeclareMathOperator{\gd}{gldim}
\DeclareMathOperator{\cd}{cd}
\DeclareMathOperator{\ext}{Ext}

\newtheorem{theorem}{Theorem }[section]
\newcommand{\bt}{\begin{theorem}}
\newcommand{\et}{\end{theorem}}

\definecolor{navajowhite}         {rgb}{1.00,0.87,0.68}
\newtheorem{prop}[theorem]{Proposition}
\newcommand{\bp}{\begin{prop}}
\newcommand{\ep}{\end{prop}}

\newtheorem{dep}[theorem]{Proposition-Definition}
\newcommand{\bdep}{\begin{dep}}
\newcommand{\edep}{\end{dep}}

\newtheorem{cor}[theorem]{Corollary}
\newcommand{\bc}{\begin{cor}}
\newcommand{\ec}{\end{cor}}
\definecolor{babyblueeyes}{rgb}    {0.63, 0.79, 0.95}
\newtheorem{lemma}[theorem]{Lemma}
\newcommand{\bl}{\begin{lemma}}
\newcommand{\el}{\end{lemma}}

\theoremstyle{definition}

\newtheorem{dfn}[theorem]{Definition}
\newcommand{\bd}{\begin{dfn}}
\newcommand{\ed}{\end{dfn}}

\theoremstyle{remark}

\newtheorem{rmq}[theorem]{Remark}
\newcommand{\brq}{\begin{rmq}}
\newcommand{\erq}{\end{rmq}}

\theoremstyle{definition}

\newtheorem{ex}[theorem]{Example}
\newcommand{\bex}{\begin{ex}}
\newcommand{\eex}{\end{ex}}

\newcommand{\bpf}{\begin{proof}}
\newcommand{\epf}{\end{proof}}

\newcommand{\smallbraid}[2]{\draw (#1,#2) ..controls +(0,-0.5) and
+(0,0.5).. (#1-1,#2-1);
\draw[white][line width = 5pt] (#1-1,#2) ..controls +(0,-0.5) and
+(0,0.5).. (#1,#2-1);
\draw (#1-1,#2) ..controls +(0,-0.5) and +(0,0.5).. (#1,#2-1);}

\newcommand{\produit}[3]{\draw (#1,#2) arc (180:360:#3); \draw(#1 + #3, #2 -#3)--(#1 + #3, #2 -#3 -0.5);}
 \newcommand{\coproduit}[3]{\draw (#1,#2) arc (0:180:#3); \draw(#1 - #3, #2 +#3 + 0.5)--(#1 - #3, #2 + #3);}

 \newcommand{\gauche}[2]{\draw (#1,#2)..controls+(0,-1)and+(1,0)..(#1-1,#2-1);
     
         \draw (#1-2,#2)..controls+(0,-1)and+(-1,0)..(#1-1,#2-1); \draw(#1-1.5, #2 - 1)--(#1-1.5, #2 - 1.5);}
         
   \newcommand{\droite}[2]{\draw (#1,#2)..controls+(0,-1)and+(1,0)..(#1-1,#2-1);
       \draw (#1-2,#2)..controls+(0,-1)and+(-1,0)..(#1-1,#2-1); \draw(#1-0.5, #2 - 1)--(#1-0.5, #2 - 1.5);}

\title{Cohomological dimension of braided Hopf algebras}
\author{Julien Bichon}
\address{Universit\'e Clermont Auvergne, CNRS, LMBP, F-63000 CLERMONT-FERRAND, FRANCE}
\email{julien.bichon@uca.fr}
\author{Thi Hoa Emilie Nguyen}
\email{thi\_hoa.nguyen@uca.fr}

\begin{document}

\begin{abstract}
We show that for a braided Hopf algebra in the category of comodules
over a cosemisimple coquasitriangular Hopf algebra, the Hochschild cohomological
dimension, the left and right global dimensions and the projective
dimensions of the trivial left and right module all coincide. We also provide convenient criteria for smoothness and the twisted Calabi-Yau property for such braided Hopf algebras (without the cosemisimplicity assumption on $H$), in terms of properties of the trivial module.
These generalize a well-known result in the case of ordinary Hopf algebras. 
As an illustration, we study the case of the coordinate algebra on the
two-parameter braided quantum group $\textrm{SL}_{2}$.
\end{abstract}

\maketitle

\section{Introduction}
The global dimension is an important homological invariant of an algebra, which most often serves as a good analogue of the dimension of a smooth affine algebraic variety. However, there are some examples where the global dimension does not match with geometric intuition. Consider for example the first Weyl algebra $A_1(k) =k\langle x, y \ | \ xy-yx=1 \rangle$. If the base field $k$ has characteristic zero, then $\gd(A_1(k))= 1$ \cite{rin62}, while $A_1(k)$, being a filtered deformation of the polynomial algebra $k[x,y]$, should be an  object of dimension $2$. This often leads us to consider the Hochschild cohomological dimension rather than the global dimension.  Recall that for an algebra $A$, the Hochschild  cohomological dimension $\cd(A)$ is defined to be the projective dimension of $A$ in the category of $A$-bimodules. The (left or right) global dimension of $A$ is always smaller than $\cd(A)$, while they coincide in the important case of the coordinate algebra on a smooth affine variety, and for the Weyl algebra one has $\cd(A_1(k)) =2$, as expected \cite{Sri}.

It is thus a natural and important question to determine classes of algebras for which the global dimension and the Hochschild cohomological dimension coincide. Among such classes that are known, let us mention two important ones.
\begin{enumerate}
 \item If $A$ is a graded connected algebra, we have $\cd(A)=\gd(A)$, and these coincide with the projective dimensions of the trivial left and right $A$-modules. See \cite{Ber05}.
\item If $A$ is a Hopf algebra, we have $\cd(A)=\gd(A)$, and these coincide with the projective dimensions of the trivial left and right $A$-modules. This follows from \cite[Proposition 5.6]{GK93}, see the appendix in \cite{WYZ17}.
\end{enumerate}

In this paper we enlarge this list by generalizing the Hopf algebra case to a class of braided Hopf algebras.
Recall \cite{Maj93} that a braided Hopf algebra is a Hopf algebra in a braided category. Braided Hopf algebras  generalize ordinary Hopf algebras, providing a wider theory of quantum symmetries. They are also very useful, through the bosonization construction \cite{Rad85,Maj94}, in studying certain classes of usual Hopf algebras themselves, see \cite{heckenberger2020hopf}.  

The primary objective in  this paper is to extend a range of homological properties observed in ordinary Hopf algebras to the case of braided Hopf algebras. 
In particular, our main result (Theorem \ref{thm:equaldimen}) is that if $A$ is  a Hopf algebra in the braided  category of comodules over a cosemisimple coquasitriangular Hopf algebra $H$, then we have $\cd(A)=\lgd(A)=\rgd(A)$, and these coincide with the projective dimensions of the trivial left and right $A$-modules. Our strategy is to extend \cite[Proposition 5.6]{GK93} to a general braided context (Corollary \ref{cor:equalpd}) and then to use comparison results for various projective dimensions in the setting of separable functors \cite{na1989separable}.  

We then study some more subtle homological properties for braided Hopf algebras, such as smoothness (an adequate analogue of regularity for noncommutative algebras) and the twisted Calabi-Yau property (an analogue of Poincaré duality in Hochschild cohomology). For a Hopf algebra $A$ in the braided  category of comodules over a coquasitriangular Hopf algebra $H$ (no cosemisimplicity assumption on $H$ is needed here), we provide convenient criteria for smoothness and the twisted Calabi-Yau property, in terms of properties of the trivial module, see Theorem \ref{thm:smoothness} and Theorem \ref{thm:tcy} respectively. Again this generalizes known results \cite{BZ08} for ordinary Hopf algebras.

We wish to emphasize that while there exist appropriate (co)homology theories for braided Hopf algebras \cite{Baez94, AMS07, HK09} to which some of our considerations apply, our main aim in this paper is not to study braided Hopf algebras from this internal prespective, but rather to use the additional structure to study the homological properties of the underlying algebras.

We illustrate our results by studying an interesting example of a braided Hopf algebra, the coordinate algebra on the two-parameter braided quantum group ${\rm SL}_2$. 
The corresponding algebra $\OO$, depending on parameters $p,q \in k^*$, coincides when $p=q$ with the usual coordinate Hopf algebra on quantum ${\rm SL}_2$, and in general is a Hopf algebra in the category of $\mathbb Z$-graded vector space endowed with an appropriate braiding. We show that $\cd(\OO)= 3$ and that $\OO$ is a twisted Calabi-Yau algebra.

A summary of this paper is as follows. Section \ref{sec:prelim} consists of preliminaries. In Section \ref{sec:modbimod}, we study the relations between categories of modules and bimodules over a braided Hopf algebra, and then provide the proof of Theorem \ref{thm:equaldimen}. In Section \ref{sec:finiteness} we discuss finiteness conditions for modules in a tensor category and prove our smoothness criterion (Theorem \ref{thm:smoothness}). In Section \ref{sec:homodual} we study the twisted Calabi-Yau property for braided Hopf algebras, and prove Theorem \ref{thm:tcy}.  
Section \ref{sec:twop} provides  illustrations of our results on the  example of the coordinate algebra on braided quantum ${\rm SL}_2$.

\medskip

\noindent
\textbf{Acknowledgements.} The second author thanks Rachel Taillefer for her comments on the proof of Proposition \ref{prop:reso}, and for continuous help and support. She also thanks Julian Le Clainche for his useful suggestions during discussions.

\section{Preliminaries}\label{sec:prelim}

This section, which also aims at fixing some notation, consists of reminders on monoidal categories and braided Hopf algebras, together with some preliminary material to be used in the proof of our comparison of cohomological dimensions for a braided Hopf algebra. Standard references we use are \cite{KSBook97} for ordinary Hopf algebra theory,  \cite{EGNO15, bulacu2019quasi, heckenberger2020hopf}  for monoidal categories and braided Hopf algebras, and \cite{weibel1994introduction} for homological algebra. We work over a fixed base field $k$.

\subsection{Monoidal and braided monoidal categories}
Recall that a \textsl{monoidal category} $(\CC, \otimes, I, a, l, r)$ consists of  a category $\CC$ endowed with the following components:
\begin{itemize}
\item [(1)] a bifunctor $\otimes: \CC \times \CC \to \CC$, called the tensor product;
\item [(2)] an object $I$, called the unit of the monoidal category;
\item[(3)] three natural isomorphisms expressing properties of the tensor product operation:
\begin{itemize}
\item [$\bullet$] a natural isomorphism 
\[a_{X, Y, Z}: X \otimes (Y \otimes Z) \simeq (X \otimes Y) \otimes Z\]
for all objects $X, Y, Z$ in $\CC$, called the associativity constraint; 
\item[$\bullet$]  two natural isomorphisms 
\[l_{X} : I \otimes X \simeq X \quad \text{and} \quad r_{X}: X \otimes I \simeq X.\]
for any object $X$ of $\mathcal C$, called
the left and right unit constraints;
\end{itemize}
\end{itemize}
satisfying the familiar pentagon and triangle axioms, see \cite[Definition 1.1]{bulacu2019quasi}.

The monoidal category  $\mathcal C = (\CC, \otimes, I, a, l, r)$  is said to be \textsl{strict} if the associativity and unit constraints $a, l, r $  all are identities of the category.

In this paper the monoidal categories of interest are all categories of vector spaces endowed with additional structures (most notably categories of comodules over a Hopf algebra) and with the associativity and unit constraints of vector spaces. In this case there is no danger in suppressing the associativity and unit constraints, and we follow this convention, hence considering our monoidal categories as strict monoidal categories. More generally  Mac Lane's coherence theorem (see e.g.  \cite[Section 1.5]{bulacu2019quasi}) states that every monoidal category is monoidally equivalent to a strict monoidal category, and this justifies further that we only consider strict monoidal categories.

Working in strict monoidal categories allows us to use the familiar graphical calculus in monoidal categories: for objects $X, Y$ in $\mathcal C$, the identity morphism $\id_{X}: X \to X$ and a morphism $f: X \to Y$ are denoted by

\[\begin{tikzpicture}[scale=0.5, base align tikzpicture]
\draw[above](0.5,0) node{$X$};
\draw(0.5,0) -- (0.5,-1.2);
\draw (0,0) -- ( 1,0);
\draw(0,-1.2) -- (1,-1.2);
\draw[below](0.5,-1.5) node{$X$};
\end{tikzpicture}
\quad \textrm{and} \quad \begin{tikzpicture}[scale=0.6, base align tikzpicture]
\draw[above](0.5,0) node{$X$};
\draw(0.5,0) -- (0.5,-1.2);
\draw (0,0) -- ( 1,0);
\draw(0,-1.2) -- (1,-1.2);
\draw[below](0.5,-1.5) node{$Y$};
 \draw [fill=white] (0.5,-0.5) circle (0.4);
 \draw (0.5,-0.5) node {\footnotesize{$f$}};
\end{tikzpicture}\]
and morphisms in $\mathcal C$ are equal precisely when the corresponding  graphical diagrams are the same up to isotopy.

A \textsl{braided monoidal category} is a monoidal category endowed with a \textsl{braiding}, i.e a  family of natural isomorphisms
\[c_{X, Y} : X \otimes Y \to Y \otimes X\]
such that for all objects $X, Y, Z$ in $\mathcal C$, we have
\[c_{X, Y\otimes Z} = (\id_{Y} \otimes c_{X,Z}) \circ (c_{X,Y} \otimes \id_{Z}), \quad
c_{X\otimes Y, Z} = (c_{X, Z} \otimes  \id_{Y}) \circ (\id_{X} \otimes  c_{Y, Z}), \quad  
 c_{X,I}=\id_X = c_{I,X}
\] 
A braided monoidal category is said to be \textsl{symmetric} when we have $c_{X,Y}= c^{-1}_{Y, X}$ for any objects $X, Y$ in $\CC$.

For objects $X$ and $Y$ of a braided monoidal category $\CC$ , we denote the braiding isomorphism $c_{X,Y}$ and its inverse $c^{-1}_{X,Y}$ respectively by
\[ \begin{tikzpicture}[scale=0.5, base align tikzpicture]
\draw[above](1,0) node{$X$};
\draw[above](2,0) node{$Y$};
\draw (0.5,0) -- (2.5,0);
\draw (1,0) -- (1, -0.25);
\draw (2,0) -- (2, -0.25);
\smallbraid{2}{-0.25}
\draw (1,-1.25) -- (1, -1.5);
\draw (2,-1.25) -- (2, -1.5);
\draw(0.5,-1.5) -- (2.5,-1.5);
\draw[below](1,-1.5) node{$Y$};
\draw[below](2,-1.5) node{$X$};
\end{tikzpicture} \quad \text{and}\quad
 \begin{tikzpicture}[scale=0.5, base align tikzpicture]
\draw[above](1,0) node{$Y$};
\draw[above](2,0) node{$X$};
\draw (0.5,0) -- (2.5,0);
\draw (1,0) -- (1, -0.25);
\draw (2,0) -- (2, -0.25);
\draw (1,-0.25) ..controls +(0,-0.5) and +(0,0.5).. (2,-1.25);
\draw[white][line width = 5pt] (1.45, -0.5) ..controls +(0,-0.5) and
+(0,0.5).. (1.55,-1);
\draw (2, -0.25) ..controls +(0,-0.5) and +(0,0.5).. (1,-1.25);
\draw (1,-1.25) -- (1, -1.5);
\draw (2,-1.25) -- (2, -1.5);
\draw(0.5,-1.5) -- (2.5,-1.5);
\draw(0.5,-1.5) -- (2.5,-1.5);
\draw[below](1,-1.5) node{$X$};
\draw[below](2,-1.5) node{$Y$};
\end{tikzpicture}\] 
The braiding axioms then are \begin{align}
\begin{tikzpicture}[scale=0.5, base align tikzpicture]
\draw[above](1,0) node{$X$};
\draw[above](2,0) node{$Y$};
\draw[above](3,0) node{$Z$};
\smallbraid{2}{0};
\smallbraid{3}{-1};
\draw (0.5,0) -- ( 3.5,0);
\draw(3, 0) -- (3,-1);
\draw(1, -1) -- (1, -2);
\draw(0.5,-2) -- (3.5,-2);
\draw[below](1,-2.5) node{$Y$};
\draw[below](2,-2.5) node{$Z$};
\draw[below](3,-2.5) node{$X$};
\draw(-1.5,-1.5) node{$c_{X, Y \otimes Z} =$};
\end{tikzpicture} \quad
\text{and} \quad
\begin{tikzpicture}[scale=0.5, base align tikzpicture]
\draw[above](1,0) node{$X$};
\draw[above](2,0) node{$Y$};
\draw[above](3,0) node{$Z$};
\smallbraid{3}{0};
\smallbraid{2}{-1};
\draw (0.5,0) -- ( 3.5,0);
\draw(3, -1) -- (3,-2);
\draw(1, 0) -- (1, -1);
\draw(0.5,-2) -- (3.5,-2);
\draw[below](1,-2.5) node{$Z$};
\draw[below](2,-2.5) node{$X$};
\draw[below](3,-2.5) node{$Y$};
\draw(-1.5,-1.5) node{$c_{X \otimes Y, Z} =$};
\end{tikzpicture} \label{(braid)}
\end{align}

We will use as well the \textsl{reverse category} of a monoidal category: if $\CC$ is a monoidal category, then $\CC^{\rm rev}$ is the monoidal category endowed with tensor product $\otimes^{\rm rev}$ defined by $X\otimes^{\rm rev}Y = Y\otimes X$. If $\CC$ is braided, then so is $\CC^{\rm rev}$, with the braiding defined by $c'_{X,Y}=c_{Y,X}$.

\subsection{Coquasitriangular Hopf algebras} Basic examples of symmetric monoidal categories are the category $_{k}\MM$ of $k$-vector spaces over our base field $k$, with the (symmetric) braiding given by the flip operators, and more generally the category of comodules over a commutative bialgebra. In this subsection, we recall the structure that produces a braiding on the category of comodules over an arbitrary bialgebra.

A \textsl{coquasitriangular bialgebra} is a bialgebra $H$ equipped with a convolution-invertible linear form $\textbf{r}: H \otimes H \to k$ (called a universal $r$-form) such that, for any $x,y,z \in H$,
\begin{equation}
yx = \textbf{r}(x_{(1)},y_{(1)}) x_{(2)}y_{(2)} \textbf{r}^{-1}(x_{(3)},y_{(3)}) \label{r1}
\end{equation}
\begin{equation}
\textbf{r}(xy,z)= \textbf{r}(x,z_{(1)}) \textbf{r}(y,z_{(2)}), \quad  \textbf{r}(x,yz)= \textbf{r}(x_{(1)}, z) \textbf{r}(x_{(2)},y) \label{r2}
\end{equation}
A \textsl{coquasitriangular Hopf algebra} is a Hopf algebra which is a coquasitriangular bialgebra.

Let $H$ be a coquasitriangular bialgebra with universal $r$-form $\textbf{r}$. For right $H$-comodules $V$ and $W$, the  linear map $\textbf{r}_{V, W} : V \otimes W \to W \otimes V$ defined by
\begin{align}
\textbf{r}_{V, W}(v \otimes w) = \textbf{r}(v_{(1)}, w_{(1)})w_{(0)} \otimes v_{(0)} \label{rVW}
\end{align}
is an $H$-colinear isomorphism, and it is an immediate verification that the axioms of an $r$-form ensure that this procedure defines a braiding on $\mathcal M^H$, which thus becomes a braided category, which might be denoted by $\mathcal M^{H,\textbf{r}}$ if we want to remember the braided structure.

\bex\label{ex:gradedbraided}
Let $\Gamma$ be an abelian group. Then the universal $r$-forms on the group algebra $k\Gamma$ correspond to the bicharacters $\Gamma\times \Gamma \to k^*$, i.e the maps $\psi$ such that
\[ \psi(xy, z) = \psi(x, z)\psi(y,z); \,\psi(x, yz) = \psi(x, y)\psi(x, z) \quad \text{for} \quad x, y, z \in \Gamma.\]
Let us explicitly describe  the braiding associated with a such a bicharacter $\psi$. 
For this, recall first that $\mathcal M^{k\Gamma}$ identifies with the category of $\Gamma$-graded vector spaces as follows: if $V = (V, \alpha)$ is a right $k\Gamma$-comodule, put, for $g \in \Gamma$, $V_{g} = \{v \in V \mid \alpha(v) = v \otimes g\}$. Then $V = \bigoplus_{g \in \Gamma} V_{g}$ defines a $\Gamma$-grading on $V$. Conversely, if $V = \bigoplus_{g \in \Gamma} V_{g}$ is $\Gamma$-graded, putting $\alpha(v) = v\otimes g$ for $v \in V_{g}$,  defines a structure of $k\Gamma$-comodule on $V$.

Given a bicharacter $\psi$, the category $\mathcal M^{k\Gamma}$ is braided with braiding: 
\begin{align*}
c_{V, W}:\quad V \otimes W &\longrightarrow W \otimes V\\
v \otimes w \in V_{g} \otimes W_{h} &\mapsto \psi(g, h)w \otimes v
\end{align*}
When $\Gamma = \mathbb Z  =\langle z \rangle$ is the infinite cyclic group with a fixed generator $z$, a bicharacter is uniquely determined by $\xi= \psi(z,z)$. We denote by $\mathcal M^{k\mathbb Z,\xi}$ the resulting braided category. 
\eex

\subsection{Algebras, modules, coalgebras and comodules in monoidal categories}
The familiar notions of algebras,  modules, coalgebras and comodules  in vector spaces categories have direct generalizations in monoidal categories.

Let $\CC$ be a monoidal category. Recall that 
\textsl{an algebra in $\CC$} is a triple $(A, m_{A}, \eta_{A})$, where $A \in \text{ob}(\CC)$, and $m_{A}: A \otimes A \to A$ and $\eta_A : I \to A$ are morphisms such that
\[m_{A} \circ (m_{A}\otimes \id_{A}) =  m_{A} \circ (\id_A \otimes m_{A}), \ m_{A} \circ (\eta_{A}\otimes \id_A)= \id_A = m_{A} \circ (\id_A \otimes \eta_{A}).\]
Denoting the multiplication and the unit by
\[m_A=\begin{tikzpicture}[scale=0.4, base align tikzpicture]
	\draw(0.5,0)--(3.5,0);
	\draw(0.5,-1.5)--(3.5,-1.5);
	\draw[above](1,0) node{$A$};
	\draw[above](3,0) node{$A$};
	\produit{1}{0}{1}
	\draw[below](2,-1.5) node{$A$};
\end{tikzpicture} \quad \textrm{and} \quad 
\eta_A = \begin{tikzpicture}[scale=0.4, base align tikzpicture]
	\draw(8,0)--(9,0);
	\draw(8,-1.5)--(9,-1.5);
	\draw[above](8.5,0) node{$I$};
	\draw(8.5,-0.5) node{$\bullet$};
	\draw(8.5,-0.5) -- (8.5,-1.5);
	\draw[below](8.5,-1.5) node{$A$};
\end{tikzpicture}\]
the above associativity and unit  axioms read
\begin{equation}
	\begin{tikzpicture}[scale=0.35]
		\draw(0,0)--(5,0);
		\draw(0,-4.5)--(5,-4.5);
		\draw(0.5,0)--(0.5,-3);
		\draw[above](4.5, 0) node{$A$};
		\draw[above](0.5,0) node{$A$};
		\draw[above](2.5,0) node{$A$};
		\produit{2.5}{0}{1}
		\produit{0.5}{-3}{1}
		\draw[below](1.5,-4.5) node{$A$};
		\draw(3.5,-1.5) ..controls +(0,-0.7) and +(0,0.7).. (2.5,-3);
		\draw(6.5,-2) node{$=$};
		\draw(8,0)--(13,0);
		\draw(8,-4.5)--(13,-4.5);
		\draw(12.5,0)--(12.5,-3);
		\draw[above](12.5, 0) node{$A$};
		\draw[above](8.5,-0) node{$A$};
		\draw[above](10.5,-0) node{$A$};
		\produit{10.5}{-3}{1}
		\produit{8.5}{0}{1}
		\draw[below](11.5, -4.5) node{$A$};
		\draw(9.5,-1.5) ..controls +(0,-0.7) and +(0,0.7).. (10.5,-3);
		\draw(15,-2) node{and};
		\draw(17,-0.5)--(20,-0.5);
		\draw(17,-3)--(20,-3);
		\draw(19.5,-0.5)--(19.5,-1.5);
		\draw(17.5,-1.25)--(17.5,-1.5);
		\draw[above](19.5, -0.5) node{$A$};
		\draw(17.5,-1.25) node{$\bullet$};
		\produit{17.5}{-1.5}{1}
		\draw[below](18.5, -3) node{$A$};
		\draw(21,-2) node{$=$};
		\draw(22,-1)--(23,-1);
		\draw(22,-2.5)--(23,-2.5);
		\draw[above](22.5, -1) node{$A$};
		\draw[below](22.5, -2.5) node{$A$};
		\draw(22.5,-1)--(22.5,-2.5);
		\draw(24,-2) node{$=$};
		\draw(25,-0.5)--(28,-0.5);
		\draw(25,-3)--(28,-3);
		\draw(25.5,-0.5)--(25.5,-1.5);
		\draw(25.5,-1.25)--(25.5,-1.5);
		\draw[above](27.5, -0.5) node{$A$};
		\draw(27.5,-1.25) node{$\bullet$};
		\draw(27.5,-1.25)--(27.5,-1.5);
		\produit{25.5}{-1.5}{1}
		\draw[below](26.5, -3) node{$A$};
		\draw(29,-2) node{.};
	\end{tikzpicture}\label{associatif}
\end{equation}
If  $A$, $B$ are  algebras in the monoidal category $\CC$, an algebra morphism $f : A \to B$ is a morphism in $\CC$ such that
\[f \circ m_A = m_B\circ (f\otimes f) \quad \textrm{and} \quad f\circ\eta_A = \eta_B\]
Graphically, this means 
\begin{equation*}
	\begin{tikzpicture}[scale=0.55]
		\draw(0, 0)--(2,0);
		\draw(0, -3)--(2,-3);
		\produit{0.5}{0}{0.5}
		\draw(0.5, -2)--(0.5,-3);
		\draw(1,-1) ..controls +(0,-0.7) and +(0,0.7).. (0.5,-2);
		\draw [fill=white] (0.5,-2.25) circle (0.4);
		\draw (0.5,-2.25) node {\footnotesize{$f$}};
		\draw[above](0.5,0) node{$A$};
		\draw[above](1.5,0) node{$A$};
		\draw[below](0.5,-3) node{$B$};
		\draw (2.5,-1.5) node{$=$};
		\draw[above](4.5,-.5) node{$A$};
		\draw[above](5.5,-.5) node{$A$};
		\draw (4,-.5) -- ( 6,-0.5);
		\draw (5.5,-0.5) -- (5.5,-1.2);
		\draw (4.5,-.5) -- (4.5,-1.2);
		\draw (5.5,-1.1) -- (5.5,-2.1);
		\draw (4.5,-1.1) -- (4.5,-2.1);
		\draw [fill=white] (4.5,-1.73) circle (0.4);
		\draw (4.5,-1.73) node {\footnotesize{$f$}};
		\draw [fill=white] (5.5,-1.73) circle (0.4);
		\draw (5.5,-1.73) node {\footnotesize{$f$}};
		\draw(4, -3.1) -- (6,-3.1);
		\produit{4.5}{-2.1}{0.5}
		\draw[below](5,-3.1) node{$B$};
		\draw (7,-1.5) node{and};
		\draw(8, -0.5)--(9.5,-0.5);
		\draw(8, -2.5)--(9.5,-2.5);
		\draw(8.75, -1) node {$\bullet$};
		\draw(8.75, -1)--(8.75,-2.5);
		\draw [fill=white] (8.75,-1.75) circle (0.4);
		\draw (8.75,-1.75) node {\footnotesize{$f$}};
		\draw[above](8.75,-0.5) node{$I$};
		\draw[below](8.75,-2.5) node{$A$};
		\draw (10.5,-1.5) node{$=$};
		\draw(11.5, -0.75)--(13,-0.75);
		\draw(11.5, -2.25)--(13,-2.25);
		\draw(12.25, -1.35) node {$\bullet$};
		\draw(12.25, -1.35)--(12.25,-2.25);
		\draw[above](12.25,-0.75) node{$I$};
		\draw[below](12.25,-2.25) node{$A$};
		\draw (15,-1.5) node{$.$};
	\end{tikzpicture}\label{morphism}\\
\end{equation*}

Let $A$ be an algebra in $\CC$. 
A \textsl{left $A$-module} $M$ (in $\CC$) is an object $M$ in $\CC$ together with a morphism $\mu_{M}^{l}: A \otimes M \to M$, denoted by 
\[\begin{tikzpicture}[scale=0.35, base align tikzpicture]
\draw(0.5,0)--(3.5,0);
\draw(0.5,-1.5)--(3.5,-1.5);
\draw[above](3, 0) node{$M$};
\draw[above](1,-0) node{$A$};
\gauche{3}{0}{1}
\draw[below](1.5,-1.5) node{$M$};
\end{tikzpicture}\]
 such that
\begin{align}
\begin{tikzpicture}[scale=0.35]
\draw(0,0)--(5,0);
\draw(0,-4)--(5,-4);
\draw(0.5,0)--(0.5,-2.5);
\draw[above](4.5, 0) node{$M$};
\draw[above](0.5,0) node{$A$};
\draw[above](2.5,0) node{$A$};
\gauche{4.5}{0}{1}
\gauche{2.5}{-2.5}{1}
\draw[below](1,-4) node{$M$};
\draw(3,-1.5) ..controls +(0,-0.7) and +(0,0.7).. (2.5,-2.5);
\draw(6.5,-2) node{$=$};
\draw(8,0)--(13,0);
\draw(8,-4)--(13,-4);
\draw(12.5,0)--(12.5,-2.5);
\draw[above](12.5, 0) node{$M$};
\draw[above](8.5,-0) node{$A$};
\draw[above](10.5,-0) node{$A$};
\gauche{12.5}{-2.5}{1}
\produit{8.5}{0}{1}
\draw[below](11, -4) node{$M$};
\draw(9.5,-1.5) ..controls +(0,-0.7) and +(0,0.7).. (10.5,-2.5);
\draw(15,-2) node{and};
\draw(17,-0.5)--(20,-0.5);
\draw(17,-3)--(20,-3);
\draw(19.5,-0.5)--(19.5,-1.5);
\draw(17.5,-1.25)--(17.5,-1.5);
\draw[above](19.5, -0.5) node{$M$};
\draw(17.5,-1.25) node{$\bullet$};
\gauche{19.5}{-1.5}{1}
\draw[below](18, -3) node{$M$};
\draw(21,-2) node{$=$};
\draw(22,-1)--(23,-1);
\draw(22,-2.5)--(23,-2.5);
\draw[above](22.5, -1) node{$M$};
\draw[below](22.5, -2.5) node{$M$};
\draw(24,-2) node{.};
\draw(22.5,-1)--(22.5,-2.5);
\end{tikzpicture}\label{l}
\end{align}
 The category of left $A$-modules (in $\CC$) is denoted  $_{A}\CC$, with morphisms the left $A$-linear morphisms, defined just as in the classical case. 
The  category  $\CC_{A}$  of right $A$-modules is defined similarly.

An \textsl{$A$-bimodule in} $\CC$ is an object 
$M$ in $\CC$ which is simultaneously a left and right $A$-module and such that
\begin{align}
\begin{tikzpicture}[scale=0.35]
\draw(0.5,0)--(5.5,0);
\draw(0.5,-4)--(5,-4);
\draw(5,0)--(5,-2.5);
\draw[above](5, 0) node{$A$};
\draw[above](1,-0) node{$A$};
\draw[above](3,-0) node{$M$};
\gauche{3}{0}{1}
\droite{5}{-2.5}{1}
\draw[below](4,-4) node{$M$};
\draw(1.5,-1.5) ..controls +(0,-0.7) and +(0,0.7).. (3,-2.5);
\draw(6.5,-2) node{$=$};
\draw(8,0)--(13,0);
\draw(8,-4)--(13,-4);
\draw(8.5,0)--(8.5,-2.5);
\draw[above](12.5, 0) node{$A$};
\draw[above](8.5,-0) node{$A$};
\draw[above](10.5,-0) node{$M$};
\gauche{10.5}{-2.5}{1}
\droite{12.5}{0}
\draw[below](9.5, -4) node{$M$};
\draw(12,-1.5) ..controls +(0,-0.7) and +(0,0.7).. (10.5,-2.5);
\end{tikzpicture}\label{bimodule}.
\end{align}
 The category of $A$-bimodules in $\CC$ is denoted $_{A}\CC_{A}$. 

\bex \label{ex:aMH} Let $H$ be a bialgebra. An algebra in the category of right $H$-comodules  $\MM^H$ is an $H$-comodule algebra, that is, an ordinary $k$-algebra $A$ endowed with an $H$-comodule structure such that the coaction map $A \to A\otimes H$ is an algebra map in the usual sense. The category $_A(\mathcal M^H)$ of left $A$-modules in $\mathcal M^H$ is the usual category of relative Hopf modules $_A\mathcal M^H$, whose objects are vector spaces $V$ endowed simultaneously with a right $H$-comodule and a left $A$-module structure and such that for any $a\in A$ and $v\in V$, we have
\[ (a.v)_{(0)} \otimes (a.v)_{(1)} = a_{(0)}.v_{(0)} \otimes a_{(1)}v_{(1)}\]
where we have used Sweedler's notation in the standard way. Similarly, the categories $(\mathcal M^H)_A$ and $_A(\mathcal M^H)_A$ are the familiar categories $\mathcal M^H_A$ and $_A\mathcal M^H_A$ respectively.
\eex

The following result is the  straightforward adaptation to monoidal categories of the familiar free module construction, see \cite[Proposition 1.6]{AMS07} for example.

\bp \label{prop:adjmod}
Let $\CC$ be a monoidal category, let $A$ be an algebra in $\CC$, and let $V$ be an object in $\CC$.
\begin{enumerate}
 \item Left multiplication endows $A\otimes V$ with a structure of left $A$-module. This construction defines a functor $\CC \to {_A\CC}$ which is left adjoint to the forgetful functor $_A\CC \to \CC$. A left $A$-module isomorphic to $A\otimes V$ as above is said to be free.
\item Right multiplication endows $V\otimes A$ with a structure of right $A$-module. This construction defines a functor $\CC \to \CC_A$ which is left adjoint to the forgetful functor $\CC_A \to \CC$. A right $A$-module isomorphic to $V\otimes A$ as above is said to be free.
\item Left and right multiplications endow $A\otimes V\otimes A$ with a structure of $A$-bimodule. This construction defines a functor $\CC \to \,_A\CC_A$ which is left adjoint to the forgetful functor $_A\CC_A \to \CC$. An $A$-bimodule isomorphic to $A\otimes V\otimes A$ as above is said to be free.
\end{enumerate}
\ep

\bpf
The proof is similar to the usual one in vector spaces categories. For example, if $X$ is a left $A$-module, the map
\begin{align*}
\Phi: \Hom_{_{A}\CC}(A \otimes V, X) &\longrightarrow \Hom_{\CC}(V, X)\\
f &\longmapsto f \circ (\eta_{A} \otimes \id_{V})
\end{align*}
is an isomorphism  with inverse
\begin{align*} 
\Psi: \Hom_{\CC}(V, X)&\longrightarrow \Hom_{_{A}\CC}(A \otimes V, X)\\
g &\longmapsto \mu^{l}_{X} \circ (\id_{A} \otimes g).\qedhere
\end{align*}
\epf

As in the ordinary case of vector spaces, the definition of a coalgebra in a monoidal category is dual that of an algebra. More precisely, a \textsl{coalgebra in the monoidal category} $\CC$ is a triple $(C, \Delta_{C}, \varepsilon_{C})$, where $\Delta_{C}: C \to C \otimes C$ and  $\varepsilon_{C}: C \to I$ are morphisms, denoted by
\begin{equation*}
 \begin{tikzpicture}[scale=0.35, base align tikzpicture]
\draw(0.5,0)--(3.5,0);
\draw(0.5,-1.5)--(3.5,-1.5);
\draw[below](3,-1.5) node{$C$};
\draw[below](1,-1.5) node{$C$};
\coproduit{3}{-1.5}{1}
\draw[above](2,-0) node{$C$};
\end{tikzpicture}
\quad  \textrm{and} \quad  \begin{tikzpicture}[scale=0.35, base align tikzpicture] 
\draw(8,0)--(9,0);
\draw(8,-1.5)--(9,-1.5);
\draw[above](8.5,0) node{$C$};
\draw(8.5,-1) node{$\bullet$};
\draw(8.5,0) -- (8.5,-1);
\draw[below](8.5,-1.5) node{$I$};
\end{tikzpicture}, 
\end{equation*}
satisfying the coassociativity and counit conditions:
\begin{equation}
\begin{tikzpicture}[scale=0.35]
\draw(0,0)--(5,0);
\draw(0,-4.5)--(5,-4.5);
\draw(4.5,-1.5)--(4.5,-4.5);
\draw[below](4.5, -4.5) node{$C$};
\draw[below](0.5,-4.5) node{$C$};
\draw[below](2.5,-4.5) node{$C$};
\coproduit{4.5}{-1.5}{1}
\coproduit{2.5}{-4.5}{1}
\draw[above](3.5,0) node{$C$};
\draw(2.5,-1.5) ..controls +(0,-0.7) and +(0,0.7).. (1.5,-3);
\draw(6.5,-2) node{$=$};
\draw(8,0)--(13,0);
\draw(8,-4.5)--(13,-4.5);
\draw(8.5,-1.5)--(8.5,-4.5);
\draw[below](12.5, -4.5) node{$C$};
\draw[below](8.5,-4.5) node{$C$};
\draw[below](10.5,-4.5) node{$C$};
\coproduit{10.5}{-1.5}{1}
\coproduit{12.5}{-4.5}{1}
\draw[above](9.5,0) node{$C$};
\draw(10.5,-1.5) ..controls +(0,-0.7) and +(0,0.7).. (11.5,-3);
\draw(15,-2) node{and};
\draw(17,-0.5)--(20,-0.5);
\draw(17,-3)--(20,-3);
\draw(19.5,-2)--(19.5,-3);
\draw(17.5,-2)--(17.5,-2.5);
\draw[below](19.5, -3) node{$C$};
\draw(17.5,-2.5) node{$\bullet$};
\coproduit{19.5}{-2}{1}
\draw[above](18.5, -0.5) node{$C$};
\draw(21,-2) node{$=$};
\draw(22,-1)--(23,-1);
\draw(22,-2.5)--(23,-2.5);
\draw[above](22.5, -1) node{$C$};
\draw[below](22.5, -2.5) node{$C$};
\draw(22.5,-1)--(22.5,-2.5);
\draw(24,-2) node{$=$};
\draw(25,-0.5)--(28,-0.5);
\draw(25,-3)--(28,-3);
\draw(27.5,-2)--(27.5,-3);
\draw(25.5,-2)--(25.5,-2.5);
\draw[below](25.5, -3) node{$C$};
\draw(25.5,-2.5) node{$\bullet$};
\coproduit{27.5}{-2}{1}
\draw[above](26.5, -0.5) node{$C$};
\draw(29,-2) node{.};
\end{tikzpicture}\label{cogebra}
\end{equation}

The definition of a coalgebra morphism and of the categories of right or  left comodules over a coalgebra in $\CC$ (denoted $\mathcal C^C$ and $^C\mathcal C$ respectively) are straightforward adaptations of the ordinary ones, and we omit them.

Algebras or coalgebras in a monoidal category are algebras or coalgebras in the reverse category $\mathcal C^{\rm rev}$ as well, and there are obvious category isomorphisms
\[\CC_A \simeq {_A\CC^{\rm rev}}, \quad _A\CC \simeq \CC^{\rm rev}_A, \quad _A\CC_A \simeq {_A\CC^{\rm rev}_A}.\]

\subsection{Hopf algebras in braided monoidal categories (braided Hopf algebras)}\label{sub:braidedHopf}
There is no natural way to formulate the definition of a bialgebra in an arbitrary monoidal category, but this becomes possible in the presence of a braiding, thanks to the following construction.

Let $\CC=(\mathcal C,c)$ be a braided monoidal category, and let $A$, $B$ be algebras in $\CC$. 
The braiding  of $\CC$ gives rise to an algebra structure on the object $A \otimes B$ with multiplication given by
 \begin{align}
\begin{tikzpicture}[scale=0.5]
\draw(3.5,0)--(3.5,-1.5);
\draw[above](1.5, 0) node{$B$};
\draw[above](0.5, 0) node{$A$};
\draw[above](2.5, 0) node{$A$};
\draw[above](3.5, 0) node{$B$};
\smallbraid{2.5}{0}
\draw(0,0)--(4,0);
\draw(0,-2.5)--(3.5,-2.5);
\produit{.5}{-1.5}{0.5}
\produit{2.5}{-1.5}{0.5}
\draw(0.5,0)--(0.5,-1.5);
\draw(1.5,-1)--(1.5,-1.5);
\draw(2.5,-1)--(2.5,-1.5);
\draw(-1.5,-1.5) node{$m_{A \otimes_{c} B} = $};
\draw[below](1, -2.5) node{$A$};
\draw[below](3, -2.5) node{$B$};
\end{tikzpicture}.\label{cross}
 \end{align}
and unit $\eta_A\otimes \eta_B$.
The resulting algebra in $\mathcal C$ is denoted by $A \otimes_{c} B$ and is called the \textsl{braided tensor product algebra of $A$ and $B$}.

The notion of bialgebra in a braided category  is then defined as follows: a \textsl{bialgebra $H = (H, m_{H}, \eta_{H}, \Delta_{H}, \varepsilon_{H})$ in a braided category $\CC$} is an algebra $(H, m_{H}, \eta_{H})$ and a coalgebra $(H, \Delta_{H}, \varepsilon_{H})$ in $\CC$ such that $\Delta_{H}: H \to H \otimes_{c} H$ and $\varepsilon_{H}: H \to I$ are algebra morphisms; that is
 \begin{align}
 \begin{tikzpicture}[scale=0.5]
 \draw(0, 0)--(2,0);
 \draw(0, -1.5)--(2,-1.5);
 \coproduit{1.5}{-1.5}{0.5}
 \produit{0.5}{0}{0.5}
 \draw[above](0.5,0) node{$H$};
 \draw[above](1.5,0) node{$H$};
 \draw[below](0.5,-1.5) node{$H$};
 \draw[below](1.5,-1.5) node{$H$};
 \draw (2.5,-0.75) node{$=$};
 \draw(3, 0.5)--(7,0.5);
 \draw(3, -2.5)--(7,-2.5);
 \draw(3.5, -0.5)--(3.5,-1.5);
 \draw(6.5, -0.5)--(6.5,-1.5);
 \coproduit{4.5}{-0.5}{0.5}
 \coproduit{6.5}{-0.5}{0.5}
 \produit{5.5}{-1.5}{0.5}
 \produit{3.5}{-1.5}{0.5}
 \smallbraid{5.5}{-0.5}
 \draw[above](4,0.5) node{$H$};
 \draw[above](6,0.5) node{$H$};
 \draw[below](4,-2.5) node{$H$};
 \draw[below](6,-2.5) node{$H$};
 \draw (7.5,-0.75) node{$,$};
 \draw(8, 0.5)--(10,0.5);
 \draw(8, -2.5)--(10,-2.5);
 \coproduit{9.5}{-2.5}{0.5}
 \draw[above](8.5,0.5) node{$I$};
 \draw(8.5, -0.25) node{$\bullet$};
 \draw(8.5,-0.25) ..controls +(0,-0.7) and +(0,0.7).. (9,-1.5);
 \draw[below](8.5,-2.5) node{$H$};
 \draw[below](9.5,-2.5) node{$H$};
 \draw (10.5,-0.75) node{$=$};
 \draw(11, 0)--(13, 0);
 \draw(11, -1.5)--(13,-1.5);
 \draw(11.5, -0.75) node{$\bullet$};
 \draw(12.5, -0.75) node{$\bullet$};
 \draw(11.5, -0.75)--(11.5, -1.5);
 \draw(12.5, -0.75)--(12.5, -1.5);
 \draw[above](12, 0) node{$I$};
 \draw[below](11.5, -1.5) node{$H$};
 \draw[below](12.5, -1.5) node{$H$};
 \draw (14,-0.75) node{ and};
 \draw(15, 0.5)--(17,0.5);
 \draw(15, -2.5)--(17,-2.5);
 \produit{15.5}{0.5}{0.5}
 \draw[above](15.5,0.5) node{$H$};
 \draw[above](16.5,0.5) node{$H$};
 \draw(16.5, -1.75) node{$\bullet$};
 \draw(16,-0.5) ..controls +(0,-0.7) and +(0,0.7).. (16.5,-1.6);
 \draw[below](16.5,-2.5) node{$I$};
 \draw (17.5,-0.75) node{$=$};
 \draw(18, 0)--(20, 0);
 \draw(18, -1.5)--(20,-1.5);
 \draw(18.5, -0.75) node{$\bullet$};
 \draw(19.5, -0.75) node{$\bullet$};
 \draw(18.5, 0)--(18.5, -0.75);
 \draw(19.5, 0)--(19.5, -0.75);
 \draw[below](19, -1.5) node{$I$};
 \draw[above](18.5, 0) node{$H$};
 \draw[above](19.5, 0) node{$H$};
 \draw (21,-0.75) node{,};
 \draw(25, 0)--(26, 0);
 \draw(25,-1.5)--(26,-1.5);
 \draw(25.5, 0)--(25.5, -1.5);
 \draw[below](25.5, -1.5) node{$I$};
 \draw[above](25.5, 0) node{$I$};
 \draw (27,-0.75) node{.};
 \draw (24,-0.75) node{=};
 \draw(22, 0.25)--(23, 0.25);
 \draw(22,-1.75)--(23,-1.75);
 \draw(22.5, -0.25) node{$\bullet$};
 \draw(22.5, -1.25) node{$\bullet$};
 \draw(22.5, -0.25)--(22.5, -1.25);
 \draw[below](22.5, -1.75) node{$I$};
 \draw[above](22.5, 0.25) node{$I$};
 \end{tikzpicture}\label{Hopf}
 \end{align}
A \textsl{Hopf algebra in a braided category $\CC$}  is a braided biagebra $H$ in $\CC$ such that there exists a morphism $S : H \to H$ in $\CC$ (called the \textsl{antipode} of $H$) with $S * \id_{H} = \eta_{H} \circ \varepsilon_{H} = \id_{H} * S$, where $*$ is the convolution product (see e.g. \cite[Lemma $2.57$]{bulacu2019quasi}), which, in diagrammatic notation, means that $S$ satisfies 
\begin{align}
\begin{tikzpicture}[scale=0.5]
\draw(0, 0)--(2,0);
\draw(0, -3)--(2,-3);
\coproduit{1.5}{-1}{0.5}
\draw(0.5, -1)--(0.5,-2);
\draw(1.5, -1)--(1.5,-2);
 \draw [fill=white] (0.5,-1.5) circle (0.4);
 \draw (0.5,-1.5) node {$S$};
\produit{0.5}{-2}{0.5}
\draw[above](1,0) node{$H$};
\draw[below](1,-3) node{$H$};
\draw (2.5,-1.5) node{$=$};
\draw(4, -0.5)--(5, -0.5);
\draw(4,-2.5)--(5,-2.5);
\draw(4.5, -1) node{$\bullet$};
\draw(4.5, -2) node{$\bullet$};
\draw(4.5, -0.5)--(4.5, -1);
\draw(4.5, -2)--(4.5, -2.5);
\draw[below](4.5, -2.5) node{$H$};
\draw[above](4.5, -0.5) node{$H$};
\draw (6,-1.5) node{$=$};
\draw(7, 0)--(9,0);
\draw(7, -3)--(9,-3);
\coproduit{8.5}{-1}{0.5}
\draw(7.5, -1)--(7.5,-2);
\draw(8.5, -1)--(8.5,-2);
 \draw [fill=white] (8.5,-1.5) circle (0.4);
 \draw (8.5,-1.5) node {$S$};
\produit{7.5}{-2}{0.5}
\draw[above](8,0) node{$H$};
\draw[below](8,-3) node{$H$};
\draw (10.5,-1.5) node{$.$};
\end{tikzpicture}\label{S}
\end{align}

A \textsl{braided Hopf algebra} is a Hopf algebra in an appropriate braided category.

Given an algebra $A$ in a braided category $\CC$, the opposite algebra $A^{\rm op}$ is the algebra having $A$ as underlying object, multiplication defined by $m_{A^{\rm op}} = m_A \circ c_{A,A}$ and the same unit as $A$. In case $\CC$ is a category of ordinary vector spaces, the opposite algebra $A^{\rm op}$ above should not be confused with the usual opposite algebra, and in that case $A^{\rm op}$ might be denoted $A^{{\rm op},c}$ to highlight the dependency on  the braiding $c$. One defines similarly the co-opposite coalgebra $C^{\rm cop}$ of a coalgebra $C$ in $\mathcal C$. The antipode of a Hopf algebra $H$ in $\CC$ is then an algebra map $H\to H^{\rm op}$ and a coalgebra map $H \to H^{\rm cop}$ (see e.g. \cite[Proposition 2.65]{bulacu2019quasi}), which, in diagrammatic notation, means
\begin{align}
\begin{tikzpicture}[scale=0.5]
\draw(-6.5, -1.5) node{\text{i.}};
\draw(0, 0)--(2,0);
\draw(0, -3)--(2,-3);
\produit{0.5}{0}{0.5}
\draw(0.5, -2)--(0.5,-3);
\draw(1,-1) ..controls +(0,-0.7) and +(0,0.7).. (0.5,-2);
 \draw [fill=white] (0.5,-2.25) circle (0.4);
 \draw (0.5,-2.25) node {$S$};
\draw[above](0.5,0) node{$H$};
\draw[above](1.5,0) node{$H$};
\draw[below](0.5,-3) node{$H$};
\draw (2.5,-1.5) node{$=$};
\draw[above](4.5,0) node{$H$};
\draw[above](5.5,0) node{$H$};
\draw (4,0) -- ( 6,0);
\draw (5.5,0) -- (5.5,-0.1);
\draw (4.5,0) -- (4.5,-0.1);
\draw (5.5,-1.1) -- (5.5,-2.1);
\draw (4.5,-1.1) -- (4.5,-2.1);
\draw [fill=white] (4.5,-1.73) circle (0.4);
 \draw (4.5,-1.73) node {$S$};
 \draw [fill=white] (5.5,-1.73) circle (0.4);
  \draw (5.5,-1.73) node {$S$};
\smallbraid{5.5}{-0.1}
\draw(4, -3.1) -- (6,-3.1);
\produit{4.5}{-2.1}{0.5}
\draw[below](5,-3.1) node{$H$};
\draw (6.5,-1.5) node{$,$};
\draw(8, -0.5)--(9.5,-0.5);
\draw(8, -2.5)--(9.5,-2.5);
\draw(8.75, -1) node {$\bullet$};
\draw(8.75, -1)--(8.75,-2.5);
 \draw [fill=white] (8.75,-1.75) circle (0.4);
 \draw (8.75,-1.75) node {$S$};
\draw[above](8.75,-0.5) node{$I$};
\draw[below](8.75,-2.5) node{$H$};
\draw (10.5,-1.5) node{$=$};
\draw(11.5, -0.75)--(13,-0.75);
\draw(11.5, -2.25)--(13,-2.25);
\draw(12.25, -1.35) node {$\bullet$};
\draw(12.25, -1.35)--(12.25,-2.25);
\draw[above](12.25,-0.75) node{$I$};
\draw[below](12.25,-2.25) node{$H$};
\draw (15,-1.5) node{$,$};
\end{tikzpicture}\label{i}\\
\begin{tikzpicture}[scale=0.5]
\draw(-6.5, -1.5) node{\text{ii.}};
\draw(0, 0)--(2,0);
\draw(0, -3)--(2,-3);
\coproduit{1.5}{-3}{0.5}
\draw(0.5, 0)--(0.5,-1.25);
\draw(0.5,-1.25) ..controls +(0,-0.5) and +(0,0.5).. (1,-2);
 \draw [fill=white] (0.5,-0.75) circle (0.4);
 \draw (0.5,-0.75) node {$S$};
\draw[above](0.5,0) node{$H$};
\draw[below](1.5,-3) node{$H$};
\draw[below](0.5,-3) node{$H$};
\draw (2.5,-1.5) node{$=$};
\draw (4,0) -- (6,0);
\draw (4,-3.1) -- (6,-3.1);
\coproduit{5.5}{-1}{0.5}
 \draw (4.5,-1) -- (4.5,-2);
  \draw (5.5,-1) -- (5.5,-2);
\draw [fill=white] (4.5,-1.5) circle (0.4);
 \draw (4.5,-1.5) node {$S$};
 \draw [fill=white] (5.5,-1.5) circle (0.4);
  \draw (5.5,-1.5) node {$S$};
  \smallbraid{5.5}{-2}
\draw[above](5,0) node{$H$};
\draw[below](4.5,-3.1) node{$H$};
\draw[below](5.5,-3.1) node{$H$};
 \draw (4.5,-3) -- (4.5,-3.1);
  \draw (5.5,-3) -- (5.5,-3.1);
\draw (6.5,-3.1) node{$,$};
\draw(8, -0.5)--(9.5,-0.5);
\draw(8, -2.5)--(9.5,-2.5);
\draw(8.75, -2) node {$\bullet$};
\draw(8.75, -0.5)--(8.75,-2);
 \draw [fill=white] (8.75,-1.15) circle (0.4);
 \draw (8.75,-1.15) node {$S$};
\draw[above](8.75,-0.5) node{$H$};
\draw[below](8.75,-2.5) node{$I$};
\draw (10.5,-1.5) node{$=$};
\draw(11.5, -0.75)--(13,-0.75);
\draw(11.5, -2.25)--(13,-2.25);
\draw(12.25, -1.75) node {$\bullet$};
\draw(12.25, -0.75)--(12.25,-1.75);
\draw[above](12.25,-0.75) node{$H$};
\draw[below](12.25,-2.25) node{$I$};
\draw (15,-1.5) node{$.$};
\end{tikzpicture}\label{ii}
\end{align}

If $A$ is a Hopf algebra in a braided category $\CC$ and $M$ is an object in $\CC$, then $ \varepsilon\otimes {\rm id}_M : A\otimes M \to M$ defines a left $A$-module structure on $M$, and we denote by $_\varepsilon M$ the resulting left $A$-module. If $M$ is a right $A$-module, then $_\varepsilon M$, with the left $A$-module structure above, becomes an $A$-bimodule. Similarly, we construct the right $A$-module $M_\varepsilon$ and the $A$-bimodule $M_\varepsilon$ if $M$ is a left $A$ module. For the unit object $I$ of $\CC$, we obtain in this way the trivial left and right $A$-modules $_\varepsilon I$ and $I_\varepsilon$.

\subsection{Abelian categories and projective dimensions}
In this paper, the abelian categories we consider are \textsl{abelian $k$-linear categories}, which means that our categories are abelian categories in the usual sense and moreover each $\Hom$ set is endowed with a structure of vector space over $k$, and the composition operation is $k$-bilinear. The basic examples of course are module or bimodule categories $_A\mathcal M$, $\mathcal M_A$ or $_A\mathcal M_A$  over a $k$-algebra $A$.

Let $\CC$ be an abelian $k$-linear category. If $\mathcal C$ has enough projectives, which as usual  means that for every object $X$ of $\CC$ there is an epimorphism $P \twoheadrightarrow X$ with $P$ projective, then every object $X$ in $\CC$ has a projective resolution and the \textsl{projective dimension of $X$}, denoted by $\pd_{\CC}(X)$, is defined to be the smallest possible length of a projective resolution of $X$. An alternative description of $\pd_{\CC}(X)$ is given by the formula
\begin{align*}\text{pd}_{\CC}(X) &= \text{sup} \{n \in \mathbb{N} \mid \exists Y \in \textrm{ob}(\CC), \e_{\CC}^{n}(X, Y) \neq \{0\}\} \in \mathbb N \cup \{\infty\} \\
& = \textrm{inf}\{n \in \mathbb{N} \mid \e_{\CC}^{n+1}(X, Y) ={0}, \ \forall  Y \in \textrm{ob}(\CC)  \}
\end{align*}
where $\e^*_{\CC}(-,-)$ are the usual Yoneda $\e$-spaces of the abelian category $\CC$, which can be computed using projective resolutions of the first factor when $\CC$ has enough projectives, and using injective resolutions of the second factor when $\CC$ has enough injectives.

The \textsl{projective dimension of $\mathcal C$} (still assuming that $\mathcal C$ has enough projectives) is then defined by
\[\pd(\CC)= \text{sup} \{\pd_{\CC}(M), \ M \in {\rm ob}(\CC)\}.\]
For a $k$-algebra $A$, the projective dimension of a left (resp. right) $A$-module $M$ is $\pd_A(M) = \pd_{_A\mathcal M}(M)$ (resp. $\pd_{A^{\rm op}}(M)=\pd_{\mathcal M_A}(M)$), and the \textsl{left and right global dimensions} of $A$ are respectively defined by
\[\lgd(A)= \pd(_A\mathcal M), \ \rgd(A)= \pd(\mathcal M_A).\]

When $\lgd(A)$ and $\rgd(A)$ coincide, the common quantity is denoted $\gd(A)$, and is called the \textsl{global dimension of $A$}. Finally, the \textsl{Hochschild cohomological dimension} of $A$ is defined by
\[\cd(A) = \pd_{_A\mathcal M_A}(A).\]
It is well known (see e.g. \cite[Proposition IX.7.6]{CE56}) that 
\[\lgd(A)\leq \cd(A), \ \rgd(A) \leq \cd(A).\]

We record, for future use, a well-known useful result (see e.g. \cite{adams1967adjoint}).

\bp\label{prop:abelian}
Let $\CC$ and $\DD$ be $k$-linear abelian categories, and let $F: \CC \to \DD$ and $G: \DD \to \CC$ be some $k$-linear exact functors  with $ G$  right adjoint to $F$. 
\begin{enumerate}
\item  If $P \in \normalfont\text{ob}(\CC)$ is projective, then $F(P) \in \normalfont \text{ob}(\DD)$ is also projective. If $\CC$ has enough projectives and furthermore $G$ is faithful, then 
$\DD$ also has enough projectives.
\item  If $P \in \normalfont \text{ob}(\DD)$ is injective, then $G(P) \in  \normalfont\text{ob}(\CC)$ is also injective. If $\DD$ has enough injectives and furthermore $F$ is faithful, then $\CC$ also has enough injectives.
\item Suppose that $\CC$ and $\DD$ have enough projectives or injectives. Then we have natural isomorphisms
\[\normalfont\e^{*}_{\DD}(F(X), V) \simeq \e^{*}_{\CC}(X, G(V))\]
for any $X \in \normalfont\text{ob}(\CC)$ and $V \in {\rm ob}(\DD)$. In particular, if $\CC$ and $\DD$ have enough projectives, then we have $\normalfont\text{pd}_{\DD}(F(X)) \leq \pd_{\CC}(X)$.
\end{enumerate}
\ep

In order that the inequality of projective dimensions in the above result becomes an equality, we need one more assumption on the functor $F$.  Let $\CC, \DD$ be categories and let $F: \CC \to \DD$ be a functor. Then $F$ induces a natural transformation
\[\PP_{-, -}: \Hom_{\CC}(-, -) \longrightarrow \Hom_{\DD}\big(F(-), F(-)\big).\]
We say that $F$ is a \textsl{separable functor} \cite{na1989separable} if there is natural transformation \[\textbf{M}_{-, -}: \Hom_{\DD}\big(F(-), F(-)\big) \longrightarrow \Hom_{\CC}(-, -)\]
such that $\textbf{M}_{-, -} \circ \PP_{-, -} = \textbf{1}_{\Hom_{\CC}(-, -)}$.

The following  result is certainly well known, for a proof we refer the reader to the obvious adaptation of \cite[Proposition 14]{bichon2022monoidal}.

\bp\label{prop:pdseparable}
Let $\CC$ and $\DD$ be $k$-linear abelian categories having enough projective objects, and let $F: \CC \to \DD$ be a $k$-linear functor. Assume that $F$ is exact, preserves projective objects and  is separable. Then for any object $X$ in $\CC$, we have $\normalfont \text{pd}_{\CC}(X) = \text{pd}_{\DD}(F(X)).$
\ep

The main examples of separable functors we consider in this paper are provided by the following result from \cite{CMIZ99}.

\bp\label{prop:sepHopfModule}
Let $H$ be a cosemisimple Hopf algebra and let $A$ be a right $H$-comodule algebra. The forgetful functors
$ _A\mathcal M^H \to {_A \mathcal M}$ and $\mathcal M^H_A\to \mathcal M_A$
are separable.
\ep

\bpf
These are left-right variations on \cite[Corollary 3.5]{CMIZ99} or \cite[Corollary 24]{caenepeel2004frobenius}, based on Rafael's separability criterion for adjoint functors, see also the direct approach using the Haar integral in \cite[Lemma 20]{bichon2022monoidal}.
\epf

\subsection{Abelian monoidal categories} \label{sub:abelianmonoidal}

An \textsl{abelian $k$-linear monoidal category} is a $k$-linear abelian category $\CC$ endowed with a monoidal category structure such that the bifunctor $-\otimes -: \CC \times \CC \to \CC$ is $k$-bilinear and such that for any object $X$ in $\CC$, the functors $X\otimes -: \CC \to \CC$ and $- \otimes X : \CC \to \CC$ are exact.  

An  \textsl{abelian $k$-linear braided category} is an abelian $k$-linear monoidal category endowed with a braiding (and hence is in particular a braided monoidal category). 

Notice that exactness of the above tensor product functors in the definition of an abelian $k$-linear monoidal category is not always assumed in the literature, but it is convenient, in order to simplify the terminology, to include these conditions as part of our axioms.

\bp \label{prop:modabelian}
Let $\CC$ be an abelian $k$-linear monoidal category, and let $A$ be an algebra in $\CC$. The categories $_A\CC$, $\CC_A$ and $_A\CC_A$ are all abelian $k$-linear, and have enough projective objects if $\CC$ has.
\ep

\bpf
That $_A\CC$, $\CC_A$ and $_A\CC_A$ are all abelian is proved for example in \cite{ardizzoni2004category}, and that these categories have enough projective objects if $\CC$ has follows from the combination of Proposition \ref{prop:adjmod} and of Proposition \ref{prop:abelian}.
\epf

We now specialize to the abelian $k$-linear monoidal category $\mathcal M^H$ with $H$ a bialgebra.
Let $A$ be a right $H$-comodule algebra. Then (see \cite{caenepeel2004frobenius} or \cite{bichon2022monoidal})
the forgetful functor $\Omega_{H}: \,_{A}\mathcal M^{H}_{A} \longrightarrow \,_{A}\mathcal M_{A}$ has a right adjoint
\begin{align*}
R: \,_{A}\mathcal M_{A} &\longrightarrow \,_{A}\mathcal M^{H}_{A}\\
V &\longmapsto V \odot H
\end{align*}
where $V \odot H$ is $V \otimes H$ as vector space, its $H$-bimodule structure is given by
\begin{align*}
x \cdot (v \otimes a) = x_{(0)} \cdot v \otimes x_{(1)}h, \hspace*{2cm}(v \otimes a)\cdot x = v \cdot x_{(0)} \otimes hx_{(1)}
\end{align*}
and its $H$-comodule structure is induced by the comultiplication of $H$. Similarly, if $V$ is a left (resp. right) $A$-module, when endowing $V \otimes H$ with only the above left (resp. right) $A$-module structure, we denote it by $V \boxdot H$ (resp. $V \diamonddot$ H), and obtain an object in $\,_{A}\MM^{H}$ (resp. in $\mathcal M_A^H$), and this defines a functor $\,_{A}\mathcal M \longrightarrow \,_{A}\mathcal M^{H}$ (resp.  $\mathcal M_A \longrightarrow \mathcal M_A^{H}$) which is right adjoint to the forgetful functor $\Omega_{H}: \,_{A}\mathcal M^{H} \longrightarrow \,_{A}\mathcal M$ (resp. $\Omega_{H}: \mathcal M_A^{H} \longrightarrow \mathcal M_A$).

\bp\label{prop:pdHopfmod}
Let $H$ be a bialgebra and let $A$ be a right $H$-comodule algebra. Then the categories  $_A\mathcal M^H$, $\mathcal M^H_A$ and $_A\mathcal M^H_A$ are all abelian $k$-linear categories having enough injectives, and have enough projectives if $\mathcal M^H$ has. We have, for any object $V$ in $_{A}\MM^{H}_A$ (resp. in $_A\MM^H$, resp. in $\mathcal M_A^H$) and any $A$-bimodule (resp. any left $A$-module, resp. any right $A$-module) $W$, natural isomorphisms
\begin{align}
\e^{*}_{_{A}\MM_{A}}(\Omega_{H}(V), W) &\simeq \e^{*}_{_{A}\MM^{H}_{A}}(V, W \odot  H) \label{forget-H}\\
 (\text{resp.} \ \e^{*}_{_A\MM}(\Omega_{H}(V), W) &\simeq \e^{*}_{_{A}\MM^{H}}(V, W \boxdot  H) \label{forget-H2} ) \\
(\text{resp.} \ \e^{*}_{\MM_{A}}(\Omega_{H}(V), W) &\simeq \e^{*}_{\MM^{H}_{A}}(V, W \diamonddot  H)).
\end{align}
In particular, if $H$ is a cosemisimple Hopf algebra,  the categories  $_A\mathcal M^H$, $\mathcal M^H_A$ and $_A\mathcal M^H_A$ are abelian $k$-linear categories having enough projectives, and we have for any object $V$ in  $_A\mathcal M^H$ (resp. in $\mathcal M^H_A$), 
\[\pd_{_A\mathcal M^H}(V) = \pd_{_A\mathcal M}(V), \ (\textrm{resp.} \ \pd_{\mathcal M^H_A}(V) = \pd_{\mathcal M_A}(V) ).\] 
\ep

\bpf
Our categories are abelian by Proposition \ref{prop:modabelian}, and the remaining statements follow,  by Proposition \ref{prop:abelian}, from the existence of previous adjoint functors and the fact that the categories $_{A}\MM_{A}^{H}$, $_{A}\MM^{H}$ and $\mathcal M_A^H$ have enough injectives, and from  Proposition \ref{prop:pdseparable}.
\epf

\section{Modules and bimodules over a braided Hopf algebra and projective dimensions}\label{sec:modbimod}

In this section we prove our result on the comparison of the global dimension and the Hochschild cohomological dimension for some braided Hopf algebras.
We begin by examining the relations between modules and bimodules over a braided Hopf algebra. 

\bp\label{prop:lefttobim}
Let $\CC$ be a braided category and let $A$ be a bialgebra in $\CC$. Let $V$ be a left $A$-module in $\CC$. Endow $V \otimes A$ with the right $A$-module structure defined by right multiplication. Then the morphism
\[\begin{tikzpicture}[scale=0.45]
\draw(-1,0)--(4,0);
\draw(-1,-3.5)--(4,-3.5);
\draw(3,0)--(3,-2.5);
\draw[above](0.5,0) node{$A$};
\draw[above](2,0) node{$V$};
\draw (2,0)--(2,-1);
\draw (2,-2)--(2,-2.5);
\draw[above](3,0) node{$A$};
\coproduit{1}{-1}{0.5}
\smallbraid{2}{-1}
\produit{2}{-2.5}{0.5}
\gauche{1}{-2}
\draw(0,-1) ..controls +(0,-0.7) and +(0,0.7).. (-1,-2);
\draw[below](-0.5,-3.5) node{$V$};
\draw[below](2.5,-3.5) node{$A$};
\draw(-3.5, -2) node{$\mu^{l}_{V \otimes A} =$};
\end{tikzpicture}\]
provides $V \otimes A$ with a left $A$-module structure, hence with an $A$-bimodule structure in $\CC$. Denoting the resulting $A$-bimodule by $V \boxtimes A$, this construction yields a functor
\begin{align*}
L = - \boxtimes A: \, _{A}\CC &\longrightarrow \,_{A}\CC_{A}\\
V& \longmapsto V \boxtimes A.
\end{align*}
\ep

\bpf
We verify that $\mu^{l}_{V \otimes A}$ is indeed a left $A$-module structure on $V\otimes A$, which means, \[\mu^{l}_{V \otimes A} \circ (m_{A} \otimes \id_{V \otimes A}) = \mu^{l}_{V \otimes A} \circ (\id_{A} \otimes \mu^{l}_{V \otimes A}):\]
\begin{center}
\begin{tikzpicture}[scale=0.45]
\draw(0,0)--(5,0);
\draw[above](1,0) node{$A$};
\draw[above](2,0) node{$A$};
\draw[above](3,0) node{$V$};
\draw[above](4,0) node{$A$};
\produit{1}{0}{0.5}
\coproduit{2}{-1.5}{0.5}
\draw(3,0)--(3,-1.5);
\draw(1,-1.5) ..controls +(0,-0.7) and +(0,0.7).. (0,-2.5);
\draw(4,0)--(4,-2.5);
\smallbraid{3}{-1.5}
\produit{3}{-2.5}{0.5}
\gauche{2}{-2.5}
\draw(3.5,-3.5)--(3.5,-4);
\draw(0,-4)--(5,-4);
\draw[below](0.5,-4) node{$V$};
\draw[below](3.5,-4) node{$A$};
\draw(6,-2) node{$=$};
\draw[above](6,-2) node{$\eqref{Hopf}$};
\draw(7,0)--(13,0);
\draw[above](8,0) node{$A$};
\draw[above](10,0) node{$A$};
\draw[above](11.5,0) node{$V$};
\draw[above](12.5,0) node{$A$};
\produit{1}{0}{0.5}
\coproduit{8.5}{-1}{0.5}
\coproduit{10.5}{-1}{0.5}
\draw(11.5,0)--(11.5,-2);
\draw(11.5,-2) ..controls +(0,-0.7) and +(0,0.7).. (11,-3);
\draw(12.5,0)--(12.5,-3);
\smallbraid{9.5}{-1}
\produit{9.5}{-2}{0.5}
\produit{7.5}{-2}{0.5}
\draw(7.5,-1)--(7.5,-2);
\draw(10.5,-1)--(10.5,-2);
\smallbraid{11}{-3}
\produit{11}{-4}{0.5}
\draw(12.5,-3) ..controls +(0,-0.7) and +(0,0.7).. (12,-4);
\gauche{10}{-4}
\draw(8,-3)--(8,-4);
\draw(7,-5.5)--(13,-5.5);
\draw(11.5,-5)--(11.5,-5.5);
\draw[below](8.5,-5.5) node{$V$};
\draw[below](11.5,-5.5) node{$A$};
\draw(14,-3) node{$=$};
\draw(15,0)--(21,0);
\draw[above](16,0) node{$A$};
\draw[above](18,0) node{$A$};
\draw[above](19.5,0) node{$V$};
\draw[above](20.5,0) node{$A$};
\coproduit{16.5}{-1}{0.5}
\coproduit{18.5}{-1}{0.5}
\draw(19.5,0)--(19.5,-1);
\draw(17.5,-3) ..controls +(0,-0.7) and +(0,0.7).. (18,-4);
\draw(20.5,0)--(20.5,-3);
\smallbraid{17.5}{-1}
\produit{18.5}{-3}{0.5}
\draw(19.5,-2)--(19.5,-3);
\produit{15.5}{-2}{0.5}
\draw(15.5,-1)--(15.5,-2);
\smallbraid{19.5}{-1}
\smallbraid{18.5}{-2}
\produit{19}{-4}{0.5}
\draw(20.5,-3) ..controls +(0,-0.7) and +(0,0.7).. (20,-4);
\gauche{18}{-4}
\draw(16,-3)--(16,-4);
\draw(15,-5.5)--(21,-5.5);
\draw(19.5,-5)--(19.5,-5.5);
\draw[below](16.5,-5.5) node{$V$};
\draw[below](19.5,-5.5) node{$A$};
%
\draw(-2,-12) node{$=$};
\draw[above](-2,-12) node{$\eqref{associatif}$};
\draw(-0.5,-8)--(6,-8);
\draw[above](1,-8) node{$A$} ;
\draw[above](3,-8) node{$A$};
\draw[above](4.5,-8) node{$V$};
\draw[above](5.5,-8) node{$A$};
\coproduit{1.5}{-9}{0.5}
\coproduit{3.5}{-9}{0.5}
\draw(4.5,-8)--(4.5,-9);
\draw(3.5,-11) ..controls +(0,-0.3) and +(0,0.3).. (4.,-11.5);
\draw(2.5,-11) ..controls +(0,-0.2) and +(0,0.2).. (3.,-11.5);
\draw(1.5,-10) ..controls +(0,-1) and +(0,1).. (1.,-11.5);
\draw(5.5,-8)--(5.5,-10);
\smallbraid{2.5}{-9}
\produit{4.5}{-10}{0.5}
\draw(5,-11)--(5,-11.5);
\produit{4}{-11.5}{0.5}
\draw(0.5,-9)--(0.5,-11);
\smallbraid{4.5}{-9}
\smallbraid{3.5}{-10}
\draw(0.5,-11) ..controls +(0,-1) and +(0,1).. (-0.5,-12.5);
\gauche{3}{-11.5}
\gauche{1.5}{-12.5}
\draw(-0.5,-14)--(6,-14);
\draw(4.5,-12.5)--(4.5,-14);
\draw[below](0,-14) node{$V$};
\draw[below](4.5,-14) node{$A$};
\draw(7,-12) node{$=$};
\draw(8,-8)--(15,-8);
\draw[above](9.5,-8) node{$A$} ;
\draw[above](12,-8) node{$A$};
\draw[above](13.5,-8) node{$V$};
\draw[above](14.5,-8) node{$A$};
\coproduit{12.5}{-9}{0.5}
\draw(13.5,-8)--(13.5,-9);
\smallbraid{13.5}{-9}
\produit{13.5}{-10}{0.5}
\draw(14.5,-8)--(14.5,-10);
\gauche{12.5}{-10}
\draw(11.5,-9) ..controls +(0,-1) and +(0,1).. (10.5,-10);
\draw(9.5,-8)--(9.5,-10.5);
\coproduit{10}{-11.5}{0.5}
\smallbraid{11}{-11.5}
\gauche{10}{-12.5}
\draw(9,-11.5) ..controls +(0,-1) and +(0,1).. (8,-12.5);
\draw(11,-12.5) ..controls +(0,-0.5) and +(0,0.5).. (12,-13);
\draw(14,-11) ..controls +(0,-1) and +(0,1).. (13,-13);
\produit{12}{-13}{0.5}
\draw(8,-14)--(15,-14);
\draw[below](8.5,-14)node {$V$};
\draw[below](12.5,-14)node {$A$};
\draw(16,-12) node{$.$};
\end{tikzpicture}
\end{center}

and $\mu^{l}_{V \otimes A} \circ (\eta_{A} \otimes \id_{V \otimes A}) = \id_{V \otimes A}:$

\begin{center}
\begin{tikzpicture}[scale=0.45]
\draw(0.5,0)--(5,0);
\draw(2,0)[above] node{$A$};
\draw(3.5,0)[above] node{$V$};
\draw(4.5,0)[above] node{$A$};
\draw(2,-0.5) node{$\bullet$};
\coproduit{2.5}{-1.5}{0.5}
\draw(3.5, 0)--(3.5,-1.5);
\smallbraid{3.5}{-1.5}
\gauche{2.5}{-2.5}
\produit{3.5}{-3}{0.5}
\draw(3.5, -2.5)--(3.5,-3);
\draw(4.5, 0)--(4.5,-3);
\draw(0.5,-4)--(5,-4);
\draw(1.5,-1.5) ..controls +(0,-0.7) and +(0,0.7).. (0.5,-2.5);
\draw(1,-4)[below] node{$V$};
\draw(4,-4)[below] node{$A$};
\draw(6,-2.5) node{$=$};
\draw(6,-2.5)[above] node{$\eqref{Hopf}$};
\draw(7.5,0)--(12,0);
\draw(9,0)[above] node{$A$};
\draw(10.5,0)[above] node{$V$};
\draw(11.5,0)[above] node{$A$};
\draw(8.5,-0.5) node{$\bullet$};
\draw(9.5,-0.5) node{$\bullet$};
\draw(10.5, 0)--(10.5,-1.5);
\draw(8.5, -0.5)--(8.5,-1.5);
\draw(9.5, -0.5)--(9.5,-1.5);
\smallbraid{10.5}{-1.5}
\gauche{9.5}{-2.5}
\produit{10.5}{-3}{0.5}
\draw(10.5, -2.5)--(10.5,-3);
\draw(11.5, 0)--(11.5,-3);
\draw(7.5,-4)--(12,-4);
\draw(8.5,-1.5) ..controls +(0,-0.7) and +(0,0.7).. (7.5,-2.5);
\draw(8,-4)[below] node{$V$};
\draw(11,-4)[below] node{$A$};
\draw(13,-2.5) node{$=$};
%
\draw(14.5,-1)--(16.5,-1);
\draw(15,-1)[above] node{$V$};
\draw(16,-1)[above] node{$A$};
\draw(15,-2.5)[below] node{$V$};
\draw(16,-2.5)[below] node{$A$};
\draw(14.5,-2.5)--(16.5,-2.5);
\draw(15,-1)--(15,-2.5);
\draw(16,-1)--(16,-2.5);
\draw(17,-2.5) node{$.$};
\end{tikzpicture}
\end{center}
Thus, we need only check the compatility of the two structures in order to conclude that $V\boxtimes A$ is a well-defined $A$-bimodule and we leave this verification to the reader. Now, let $f \in \Hom_{_{A}\CC}(V, W)$, we see that $f \otimes \id_{A} \in \Hom_{_{A}\CC}(V \boxtimes A, W \boxtimes A)$:
\begin{center}
\begin{tikzpicture}[scale=0.5]
\draw(-1,0)--(4,0);
\draw(-1,-4.5)--(4,-4.5);
\draw(3,0)--(3,-2.5);
\draw[above](0.5,0) node{$A$};
\draw[above](2,0) node{$V$};
\draw (2,0)--(2,-1);
\draw (2,-2)--(2,-2.5);
\draw[above](3,0) node{$A$};
\coproduit{1}{-1}{0.5}
\smallbraid{2}{-1}
\produit{2}{-2.5}{0.5}
\gauche{1}{-2}
\draw (-.5,-3)-- (-.5,-4.5);
\draw (2.5,-3)-- (2.5,-4.5);
\draw(0,-1) ..controls +(0,-0.7) and +(0,0.7).. (-1,-2);
\draw [fill=white] (-.5,-3.7) circle (0.4);
\draw (-0.5,-3.7) node {\footnotesize{$f$}};
\draw[below](-0.5,-4.5) node{$V$};
\draw[below](2.5,-4.5) node{$A$};
\draw[below](5,-2) node{$=$};
\draw(6,0)--(11,0);
\draw(6,-4.5)--(11,-4.5);
\draw(10,0)--(10,-2.5);
\draw[above](7.5,0) node{$A$};
\draw[above](9,0) node{$V$};
\draw (9,0)--(9,-1);
\draw (9,-2)--(9,-2.5);
\draw[above](10,0) node{$A$};
\coproduit{8}{-1}{0.5}
\smallbraid{9}{-1}
\produit{9}{-2.5}{0.5}
\gauche{8}{-3}
\draw (6,-2.5)-- (6,-3);
\draw (8,-2)-- (8,-3);
\draw (9.5,-3)-- (9.5,-4.5);
\draw(7,-1) ..controls +(0,-0.7) and +(0,0.7).. (6,-2.5);
\draw [fill=white] (8,-3) circle (0.4);
\draw[below](6.5,-4.5) node{$V$};
\draw[below](9.5,-4.5) node{$A$};
\draw[below](12,-2) node{$=$};
\draw (8,-3) node {\footnotesize{$f$}};
\draw(13,0)--(18,0);
\draw(13,-4.5)--(18,-4.5);
\draw(14.5,0)--(14.5,-1);
\produit{16}{-3}{0.5}
\draw(17,0)--(17,-3);
\draw[above](14.5,0) node{$A$};
\draw[above](16,0) node{$V$};
\draw[above](17,0) node{$A$};
\draw(16,0)--(16,-2);
\smallbraid{16}{-2}
\coproduit{15}{-2}{0.5}
\draw [fill=white] (16,-1) circle (0.4);
\draw (16,-1) node {\footnotesize{$f$}};
\gauche{15}{-3}
\draw(14,-2) ..controls +(0,-0.7) and +(0,0.7).. (13,-3);
\draw(16.5,-4)--(16.5,-4.5);
\draw[below](16.5,-4.5) node{$A$};
\draw[below](13.5,-4.5) node{$V$};
\draw[below](19,-2) node{$.$};
\end{tikzpicture}
\end{center}
Similarly, we also observe that $f \otimes \text{id}_{A}$ is a morphism in $\mathcal{C}_{A}$. Consequently, it is a morphism in $_{A}\CC_{A}$, implying that $L$ defines a functor.
\epf

The following is \cite[Proposition 3.7.1]{heckenberger2020hopf}, we include the proof for the sake of completeness.

\bp \label{prop:bimtoleft}
Let $\CC$ be a braided category and $A$ be a Hopf algebra in $\CC$. Let $M$ be an $A$-bimodule in $\CC$, the morphism
\[
\begin{tikzpicture}[scale=0.45, base align tikzpicture]
\draw(-1.5,-3) node{$\mu^{l}_{\widetilde{M}} =$};
\draw(0,0)--(3.5,0);
\draw(0,-6.2)--(3.5,-6.2);
\coproduit{1}{-1}{0.5}
\draw[above](0.5, 0) node{$A$};
\draw[above](3,0) node{$M$};
\draw[below](2,-6.2) node{$M$};
\draw(1, -1)--(1, -2);
\draw(3, 0)--(3, -2.7);
\draw(0, -1)--(0, -3.7);
 \draw [fill=white] (1,-1.5) circle (0.4);
 \draw (1,-1.5) node {\footnotesize{$S$}};
 \smallbraid{3}{-2.7}
 \gauche{2}{-3.7}
\draw(1,-2) ..controls +(0,-0.4) and +(0,0.4).. (2,-2.7);
\droite{2.5}{-4.7}
\draw(3,-3.7) ..controls +(0,-0.7) and +(0,0.7).. (2.5,-4.7);
\end{tikzpicture}
\]
endows $M$ with a left $A$-module structure in $\CC$. We then denote by $\widetilde{M}$ the resulting left $A$-module. This construction defines a functor
\begin{align*}
R:{}_{A}\CC_{A} &\longrightarrow {}_{A}\CC\\
M &\mapsto \widetilde{M}
\end{align*}
\ep

\bpf
We begin by showing that $\mu^{l}_{\widetilde{M}} \circ (m_{A} \otimes \id_{\widetilde{M}}) = \mu^{l}_{\widetilde{M}} \circ (\id_{A} \otimes \mu^{l}_{\widetilde{M}}):$
\[
\begin{tikzpicture}[scale=0.45]
\draw(-0.5,0)--(3.5,0);
\draw(-0.5,-5.5)--(3.5,-5.5);
\produit{.5}{0}{0.5}
\draw[above](0.5, 0) node{$A$};
\draw[above](2.7,0) node{$M$};
\draw[above](1.5,0) node{$A$};
\coproduit{1.5}{-1.5}{.5}
\draw(2.5,0)--(2.5,-2.1);
\smallbraid{2.5}{-2.1}
\draw [fill=white] (1.5,-1.8) circle (0.35);
\draw (1.5,-1.8) node {\footnotesize{$S$}};
\draw(0.5,-1.5) ..controls +(0,-1) and +(0,1).. (-0.5,-3);
\draw(2.5,-3) ..controls +(0,-0.7) and +(0,0.7).. (2,-4);
\gauche{1.5}{-3}
\droite{2}{-4}
\draw[below](1.5,-5.5) node{$M$};
\draw[below](5,-3) node{\footnotesize{$=$}};
\draw[above](5,-3) node{{$\eqref{Hopf}$}};
%
\draw(7,0)--(12,0);
\draw(7,-7.5)--(12,-7.5);
\draw[above](8, 0) node{$A$};
\draw[above](11.2,0) node{$M$};
\draw[above](10,0) node{$A$};
\coproduit{8.5}{-1}{0.5}
\coproduit{10.5}{-1}{0.5}
\smallbraid{9.5}{-1}
\produit{9.5}{-2}{0.5}
\produit{7.5}{-2}{0.5}
\draw(7.5, -1)--(7.5,-2);
\draw(10.5, -1)--(10.5,-2);
\smallbraid{11}{-3.5}
\draw [fill=white] (10,-3.2) circle (0.35);
\draw (10,-3.2) node {\footnotesize{$S$}};
\draw(11,0)--(11, -3.5);
\gauche{10}{-4.5}
\droite{10.5}{-6}
\draw(8, -2.5)--(8,-4.5);
\draw(11,-4.5) ..controls +(0,-0.7) and +(0,0.7).. (10.5,-6);
\draw[below](10,-7.5) node{$M$};
\draw[below](14,-3) node{\footnotesize{$=$}};
\draw[above](14,-3) node{{$\eqref{i}$}};
\draw(16,0)--(21,0);
\draw(16,-8)--(21,-8);
\draw[above](17, 0) node{$A$};
\draw[above](20.5,0) node{$M$};
\draw[above](19,0) node{$A$};
\coproduit{17.5}{-1}{0.5}
\coproduit{19.5}{-1}{0.5}
\smallbraid{18.5}{-1}
\produit{16.5}{-2}{0.5}
\draw(16.5, -1)--(16.5,-2);
\draw(19.5, -1)--(19.5,-2);
\smallbraid{19.5}{-2.7}
\draw [fill=white] (18.5,-2.4) circle (0.35);
\draw (18.5,-2.4) node {\footnotesize{$S$}};
\draw [fill=white] (19.5,-2.4) circle (0.35);
\draw (19.5,-2.4) node {\footnotesize{$S$}};
\produit{18.5}{-3.7}{0.5}
\smallbraid{20}{-4.5}
\draw(20.5, 0)--(20.5,-3.5);
\gauche{19}{-5.5}
\droite{19.5}{-6.5}
\draw(17, -2.5)--(17,-5.5);
\draw(20,-5.5) ..controls +(0,-0.7) and +(0,0.7).. (19.5,-6.5);
\draw(20.5,-3.5) ..controls +(0,-0.7) and +(0,0.7).. (20,-4.5);
\draw[below](19,-8) node{$M$};
\draw(-3,-13) node{$=$};
\draw[above](-3,-13) node{\eqref{l}};
\draw(-2,-10)--(4,-10);
\draw(-2,-20)--(4,-20);
\draw[above](0, -10) node{$A$};
\draw[above](3.5,-10) node{$M$};
\draw[above](2,-10) node{$A$};
\coproduit{0.5}{-11}{0.5}
\coproduit{2.5}{-11}{0.5}
\smallbraid{1.5}{-11}
\draw(-0.5, -11)--(-0.5,-12);
\draw(2.5, -11)--(2.5,-12);
\smallbraid{3.5}{-13.7}
\smallbraid{2.5}{-12.7}
\draw [fill=white] (1.5,-12.4) circle (0.35);
\draw (1.5,-12.4) node {\footnotesize{$S$}};
\draw [fill=white] (2.5,-12.4) circle (0.35);
\draw (2.5,-12.4) node {\footnotesize{$S$}};
\draw(3.5, -10)--(3.5,-13.7);
\produit{2.5}{-15.7}{0.5}
\smallbraid{2.5}{-14.7}
\draw(3.5, -14.7)--(3.5,-15.7);
\draw(1.5, -13.7)--(1.5,-14.7);
\draw(0.5, -12)--(0.5,-14.7);
\gauche{1.5}{-15.7}
\draw(0.5,-14.7) ..controls +(0,-0.7) and +(0,0.7).. (-0.5,-15.7);
\gauche{0}{-16.7}
\draw(-0.5, -12)--(-0.5,-13.5);
\draw(-0.5,-13.5) ..controls +(0,-1) and +(0,1).. (-2,-15);
\draw(-2, -15)--(-2,-16.7);
\draw(3,-16.5) ..controls +(0,-1) and +(0,1).. (2,-18.5);
\draw(-1.5,-18.2) ..controls +(0,-0.3) and +(0,0.3).. (0,-18.5);
\droite{2}{-18.5}
\draw[below](1.5,-20) node{$M$};
\draw[below](5.5,-13) node{\footnotesize{$=$}};
\draw(7,-10)--(13,-10);
\draw(7,-20)--(13,-20);
\draw[above](9, -10) node{$A$};
\draw[above](12.5,-10) node{$M$};
\draw[above](11,-10) node{$A$};
\coproduit{9.5}{-11}{0.5}
\coproduit{11.5}{-11}{0.5}
\smallbraid{10.5}{-11}
\draw(8.5, -11)--(8.5,-12);
\draw(11.5, -11)--(11.5,-12);
\smallbraid{12.5}{-12.7}
\smallbraid{11.5}{-13.7}
\draw [fill=white] (10.5,-12.4) circle (0.35);
\draw (10.5,-12.4) node {\footnotesize{$S$}};
\draw [fill=white] (11.5,-12.4) circle (0.35);
\draw (11.5,-12.4) node {\footnotesize{$S$}};
\draw(10.5, -12.7)--(10.5,-13.7);
\draw(12.5, -10)--(12.5,-12.7);
\smallbraid{12.5}{-14.7}
\produit{11.5}{-15.7}{0.5}
\draw(12.5, -13.7)--(12.5,-14.7);
\gauche{10.5}{-14.7}
\draw(9.5, -12)--(9.5,-13.7);
\draw(9.5,-13.7) ..controls +(0,-0.7) and +(0,0.7).. (8.5,-14.7);
\gauche{9}{-15.7}
\draw(8.5, -11)--(8.5,-12.5);
\draw(8.5,-13.5) ..controls +(0,-1) and +(0,1).. (7,-15);
\draw(7, -15)--(7,-15.7);
\draw(8.5, -12.5)--(8.5,-13.5);
\draw(12,-16.5) ..controls +(0,-1) and +(0,1).. (11,-18.5);
\draw(7.5,-17) ..controls +(0,-1) and +(0,1).. (9,-18.5);
\droite{11}{-18.5}
\draw[below](10.5,-20) node{$M$};
\draw[below](14,-13) node{\footnotesize{$=$}};
\draw(16,-10)--(22,-10);
\draw(16,-20)--(22,-20);
\draw[above](18, -10) node{$A$};
\draw[above](21.5,-10) node{$M$};
\draw[above](20,-10) node{$A$};
\coproduit{18.5}{-11}{0.5}
\coproduit{20.5}{-11}{0.5}
\smallbraid{19.5}{-11.7}
\draw(17.5, -11)--(17.5,-12);
\draw(20.5, -11)--(20.5,-12);
\smallbraid{21.5}{-11.7}
\smallbraid{20.5}{-12.7}
\draw [fill=white] (18.5,-11.4) circle (0.35);
\draw (18.5,-11.4) node {\footnotesize{$S$}};
\draw [fill=white] (20.5,-11.4) circle (0.35);
\draw (20.5,-11.4) node {\footnotesize{$S$}};
\draw(21.5, -10)--(21.5,-11.7);
\draw(21.5, -12.7)--(21.5,-13.7);
\produit{20.5}{-15.7}{0.5}
\smallbraid{21.5}{-13.7}
\draw(21.5, -14.7)--(21.5,-15.7);
\draw(19.5, -13.7)--(19.5,-14.7);
\draw(18.5, -12.7)--(18.5,-13.7);
\draw(19.5, -11)--(19.5,-11.7);
\draw(20.5, -14.7)--(20.5,-15.7);
\gauche{19.5}{-14.7}
\draw(18.5,-13.7) ..controls +(0,-0.7) and +(0,0.7).. (17.5,-14.7);
\gauche{18}{-15.7}
\draw(17.5, -12)--(17.5,-13);
\draw(17.5,-13) ..controls +(0,-1) and +(0,1).. (16,-15);
\draw(16, -15)--(16,-16);
\draw(21,-16.7) ..controls +(0,-1) and +(0,1).. (20,-18.5);
\draw(16.5,-17.2) ..controls +(0,-1) and +(0,1).. (18,-18.5);
\droite{20}{-18.5}
\draw[below](19.5,-20) node{$M$};
\end{tikzpicture}
\]
\[
\begin{tikzpicture}[scale=0.4]
\draw (6,-34.5) node {$= $};
\draw[above] (6,-34.5) node {$\eqref{l}$};
\draw(9, -29)--(16,-29);
\draw (9,-39.5)--(16,-39.5);
                   \draw[above] (10,-29) node {$A$};
                    \draw[above] (12.5,-29) node {$A$};
                    \draw[above] (14.5,-29) node {$M$};
                    \draw[below] (12.5,-39.5) node {$M$};
        \draw(14.5,-29)--(14.5, -32);
        \coproduit{10.5}{-30}{0.5}
        \coproduit{13}{-30}{0.5}
        \draw(10.5,-30)--(10.5, -33.5);
        \draw(12,-30)--(12, -31);
        \draw(13,-30)--(13, -31);
\draw(12, -31) ..controls +(0,-1.5) and +(0,1.5).. (11.5, -33);
        \smallbraid{14.5}{-32}
        \draw(13, -31) ..controls +(0,-0.7) and +(0,0.6).. (13.5, -32);
        \gauche{13.5}{-33}
        \smallbraid{12}{-34.5}
         \draw(10.5, -33.5) ..controls +(0,-0.7) and +(0,0.6).. (11, -34.5);
        \gauche{11}{-35.5}
        \draw(9.5, -30) ..controls +(0,-1.5) and +(0,1.5).. (9, -32);
        \draw(9,-32)--(9,-35.5);
        
        \draw [fill=white] (13,-30.8) circle (0.4);
         \draw [fill=white] (10.5,-31.3) circle (0.4);
                      \draw (13,-30.8) node {\footnotesize{$S$}};
                     \draw (10.5,-31.3) node {\footnotesize{$S$}};
            \draw (17,-34.5) node {$= $};
            \draw(14.5,-33)--(14.5,-34.5);
       \draw(14.5, -34.5) ..controls +(0,-0.7) and +(0,0.6).. (13, -35.5);
       \draw(12, -36.5) ..controls +(0,-0.4) and +(0,0.4).. (11.5, -37);
       \smallbraid{13}{-35.5}
      \droite{11.5}{-37}
      \droite{13}{-38}
      \draw(13, -36.5)--(13, -38);
      
\draw(19, -29)--(26,-29);
\draw (19,-40.5)--(26,-40.5);
                   \draw[above] (20,-29) node {$A$};
                    \draw[above] (22.5,-29) node {$A$};
                    \draw[above] (24.5,-29) node {$M$};
                    \draw[below] (21.5,-40.5) node {$M$};
        \draw(20,-29)--(20, -34);
        \coproduit{20.5}{-35}{0.5}
        \coproduit{23}{-30}{0.5}
        \draw(20.5,-35)--(20.5, -36);
        \draw(22,-30)--(22, -31);
        \draw(23,-30)--(23, -31);
\draw(22, -31) ..controls +(0,-1) and +(0,1).. (21.5, -32.5);
        \smallbraid{24.5}{-31.5}
\draw(23, -31) ..controls +(0,-0.6) and +(0,0.3).. (23.5, -31.5);
        \gauche{23.5}{-32.5}
        \droite{24}{-33.5}
        \smallbraid{22.5}{-36.5}
        \draw(24.5, -32.5) ..controls +(0,-0.7) and +(0,0.6).. (24, -33.5);
        \gauche{21.5}{-37.5}
       \draw(23.5, -35) ..controls +(0,-1) and +(0,1).. (22.5, -36.5);
       \draw(24.5,-29)--(24.5,-31.5);
        \draw(19.5,-35)--(19.5,-37.5);
        \draw [fill=white] (23,-30.8) circle (0.4);
         \draw [fill=white] (20.5,-35.5) circle (0.4);
                      \draw (23,-30.8) node {\footnotesize{$S$}};
                     \draw (20.5,-35.5) node {\footnotesize{$S$}};
\droite{22}{-39}
\draw(22.5, -37.5) ..controls +(0,-1) and +(0,1).. (22, -39);
\draw(20.5, -36) ..controls +(0,-0.6) and +(0,0.3).. (21.5, -36.5);
\draw (27,-39) node {.};
\end{tikzpicture}
\]
Then, \quad
\begin{tikzpicture}[scale=0.45, base align tikzpicture]
\draw(0,1.2)--(3.5,1.2);
\draw(0,-6.2)--(3.5,-6.2);
\draw(0.5, 0.2)--(0.5, 0);
\coproduit{1}{-1}{0.5}
\draw(0.5, 0.2) node{$\bullet$};
\draw[above](3,1.2) node{$M$};
\draw[below](2,-6.2) node{$M$};
\draw(1, -1)--(1, -2);
\draw(3, 1.2)--(3, -2.7);
\draw(0, -1)--(0, -3.7);
 \draw [fill=white] (1,-1.5) circle (0.4);
 \draw (1,-1.5) node {\footnotesize{$S$}};
 \smallbraid{3}{-2.7}
 \gauche{2}{-3.7}
\draw(1,-2) ..controls +(0,-0.4) and +(0,0.4).. (2,-2.7);
\droite{2.5}{-4.7}
\draw(3,-3.7) ..controls +(0,-0.7) and +(0,0.7).. (2.5,-4.7);
\draw(4.5, -3) node{$=$};
\draw[above](4.5, -3) node{$\eqref{Hopf}$};
\draw(5.5,1.2)--(9.5,1.2);
\draw(5.5,-6.2)--(9.5,-6.2);
\draw(6, 0.2)--(6, -1);
\draw(7, 0.2)--(7, -1);
\draw(6, 0.2) node{$\bullet$};
\draw(7, 0.2) node{$\bullet$};
\draw[above](9,1.2) node{$M$};
\draw[below](8,-6.2) node{$M$};
\draw(7, -1)--(7, -2);
\draw(9, 1.2)--(9, -2.7);
\draw(6, -1)--(6, -3.7);
 \draw [fill=white] (7,-1.3) circle (0.4);
 \draw (7,-1.3) node {\footnotesize{$S$}};
 \smallbraid{9}{-2.7}
 \gauche{8}{-3.7}
\draw(7,-2) ..controls +(0,-0.4) and +(0,0.4).. (8,-2.7);
\droite{8.5}{-4.7}
\draw(9,-3.7) ..controls +(0,-0.7) and +(0,0.7).. (8.5,-4.7);
\draw(10.5, -3) node{$=$};
\draw[above](10.5, -3) node{$\eqref{i}$};
\draw(11.5,0)--(15.5,0);
\draw(11.5,-5)--(15.5,-5);
\draw(12, -0.7) node{$\bullet$};
\draw(15, -0.7) node{$\bullet$};
\draw[above](13.5,0) node{$M$};
\draw[below](12.5,-5) node{$M$};
\draw(13.5, 0)--(13.5, -2);
\draw(15, -0.7)--(15, -2);
\gauche{13.5}{-2}
\draw(12,-0.7) ..controls +(0,-1) and +(0,1).. (11.5,-2);
\gauche{14}{-3.5}
\draw(15,-2) ..controls +(0,-1) and +(0,1).. (14,-3.5);
\draw(16.5, -3) node{$=$};
\draw(17.5,-1.7)--(19,-1.7);
\draw(17.5,-3.7)--(19,-3.7);
\draw(18.25,-1.7)--(18.25,-3.7);
\draw[above](18.25, -1.7) node{$M$};
\draw[below](18.25,-3.7) node{$M$};
\end{tikzpicture}, that is, $\mu^{l}_{\widetilde{M}} \circ (\eta_{A} \otimes \id_{M}) = \id_{M}$. This finishes our proof.
\epf

\bp\label{prop:adjointmodulebimodule}
Let $\CC$ be a braided category and $A$ be a Hopf algebra in $\CC$. Then the  functor $R:{}_{A}\CC_{A} \longrightarrow {}_{A}\CC$ is right adjoint to the functor $L = - \boxtimes A : \, _{A}\CC \longrightarrow \,_{A}\CC_{A}$.
\ep

\bpf Let $V \in {}_{A}\CC$ and $M \in {}_{A}\CC_{A}$. Consider
\begin{align*}
\Phi_{V,M} : \text{Hom}_{_{A}\CC_{A}}(V \boxtimes A, M) &\longrightarrow \text{Hom}_{_{A}\CC}(V, \widetilde{M})\\
f &\longmapsto \tilde{f} = f\circ (\id_{V} \otimes \eta_{A}).
\end{align*}
We verify that $\tilde{f}$ is well defined as a morphism in $_{A}\CC$, which means $\mu^{l}_{\widetilde{M}} \circ (\id_{A} \otimes \tilde{f}) = \tilde{f} \circ \mu^{l}_{V}:$
\[
\begin{tikzpicture}[scale=0.45]
\draw(0,0)--(4.5,0);
\draw(0,-6)--(4.5,-6);
\draw(3.9,-0.7)--(3.9,-1);
\draw(1, 0)--(1,-1);
\draw(3.9, -0.7) node{$\bullet$};
\draw[above](1, 0) node{$A$};
\draw[above](2.5,0) node{$V$};
\draw(2.5,0)--(2.5,-1);
\produit{2.5}{-1}{0.7}
\draw [fill=white] (3.2,-1.5) circle (0.4);
\draw(3.2,-1.5) node {\footnotesize{$f$}};
\coproduit{1.5}{-2}{.5}

\smallbraid{2.5}{-2.6}
\draw [fill=white] (1.5,-2.3) circle (0.35);
\draw (1.5,-2.3) node {\footnotesize{$S$}};
\draw(0.5,-2) ..controls +(0,-1) and +(0,1).. (-0.5,-3.5);
\draw(2.5,-3.5) ..controls +(0,-0.7) and +(0,0.7).. (2,-4.5);
\draw(3.2,-2.1) ..controls +(0,-0.5) and +(0,0.5).. (2.5,-2.7);
\gauche{1.5}{-3.5}
\droite{2}{-4.5}
\draw[below](1.5,-6) node{$M$};
\draw[below](5,-3) node{\footnotesize{$=$}};
%
\draw(7,0)--(11.5,0);
\draw(7,-7)--(11.5,-7);
\draw(8, 0)--(8,-1);
\draw(10.9, -0.7) node{$\bullet$};
\draw[above](8, 0) node{$A$};
\draw[above](9.5,0) node{$V$};
\draw(9.5,0)--(9.5,-1.5);
\draw(10.9,-.7)--(10.9,-1.7);
\produit{8.1}{-3.5}{0.7}
\draw [fill=white] (8.8,-4) circle (0.4);
\draw(8.8,-4) node {\footnotesize{$f$}};
\coproduit{8.5}{-1.5}{.5}
\draw(8.5,-2.5) ..controls +(0,-0.7) and +(0,0.7).. (8.1,-3.5);

\smallbraid{9.5}{-1.5}
\smallbraid{10.5}{-2.5}
\draw(10.5, -3.5)--(10.5,-4.5);
\draw [fill=white] (10.5,-4.2) circle (0.35);
\draw (10.5,-4.2) node {\footnotesize{$S$}};
\draw(7.5,-3) ..controls +(0,-1) and +(0,1).. (6.8,-4.5);
\draw(7.5,-1.5)--(7.5,-3);
\draw(10.5,-4.5) ..controls +(0,-0.7) and +(0,0.7).. (9.3,-5.5);
\draw(10.9,-1.7) ..controls +(0,-0.5) and +(0,0.5).. (10.5, -2.5);
\gauche{8.8}{-4.5}
\droite{9.3}{-5.5}
\draw[below](8.5,-7) node{$M$};
\draw[below](13,-3) node{\footnotesize{$=$}};
\draw[above](13,-3) node{\footnotesize{$(*)$}};
\draw(15,0)--(19.5,0);
\draw(15,-7.2)--(19.5,-7.2);
\draw(17.5, -3)--(17.5,-4);
\draw(18.9, -0.7) node{$\bullet$};
\draw[above](16, 0) node{$A$};
\draw[above](17.5,0) node{$V$};
\draw(17.5,0)--(17.5,-1);
\draw(18.9,-.7)--(18.9,-1.2);
\produit{17.5}{-4}{0.5}
\coproduit{16.5}{-1}{.5}
\draw(15.5,-1) ..controls +(0,-0.5) and +(0,0.5).. (15,-1.7);
\coproduit{15.5}{-2.6}{0.5}
\smallbraid{16.5}{-2.6}
\smallbraid{17.5}{-1}
\smallbraid{18.5}{-2}
\draw(18.5, -3)--(18.5,-4);
\draw [fill=white] (18.5,-3.4) circle (0.35);
\draw (18.5,-3.4) node {\footnotesize{$S$}};
\draw(14.5,-2.5) ..controls +(0,-1) and +(0,1).. (14,-4);
\draw(16.5,-3.6) ..controls +(0,-1) and +(0,1).. (17,-5);
\draw(15.5,-3.6) ..controls +(0,-0.3) and +(0,0.3).. (16,-4);
\produit{17}{-5}{0.5}
\draw(18.9,-1.2) ..controls +(0,-0.5) and +(0,0.5).. (18.5, -2);
\draw(16.5,-2)--(16.5,-2.6);
\gauche{16}{-4}
\produit{16.1}{-6}{0.7}
\draw [fill=white] (16.8,-6.5) circle (0.4);
\draw(16.8,-6.5) node {\footnotesize{$f$}};
\draw[below](17,-7.1) node{$M$};
\draw(14.5,-5.5) ..controls +(0,-0.5) and +(0,0.5).. (16.1, -6);
\draw[below](21,-3) node{\footnotesize{$=$}};
\draw(23,0)--(27,0);
\draw(23,-7.2)--(27,-7.2);
\draw[above](25, 0) node{$A$};
\draw[above](26.5,0) node{$V$};

\coproduit{25.5}{-1}{0.5}
\smallbraid{26.5}{-1}
\draw(26.5,-0)--(26.5,-1);
\smallbraid{25.5}{-2.6}
\draw(26.5,-2)--(26.5,-3.5);

\draw(24.5,-1) ..controls +(0,-0.5) and +(0,0.5).. (24,-1.7);
\coproduit{24.5}{-2.6}{0.5}
\smallbraid{17.5}{-1}
\smallbraid{18.5}{-2}
\draw [fill=white] (26.5,-3) circle (0.35);
\draw (26.5,-3) node {\footnotesize{$S$}};
\produit{25.5}{-3.5}{0.5}
\gauche{25}{-4}
\produit{24.2}{-6}{0.7}

\draw(23.5,-2.5) ..controls +(0,-1) and +(0,1).. (23,-4);
\draw(26,-4.5) ..controls +(0,-1) and +(0,1).. (25.6,-6);
\draw(24.5,-3.6) ..controls +(0,-0.3) and +(0,0.3).. (25,-4);
\draw(23.5,-5.5) ..controls +(0,-0.3) and +(0,0.3).. (24.2,-6);
\draw(25.5,-2)--(25.5,-2.6);
\draw [fill=white] (24.9,-6.5) circle (0.4);
\draw(24.9,-6.5) node {\footnotesize{$f$}};
\draw[below](25,-7.1) node{$M$};
\draw(0,-9.5)--(4,-9.5);
\draw(0,-16.8)--(4,-16.8);
\draw[above](2, -9.5) node{$A$};
\draw[above](3.5,-9.5) node{$V$};

\coproduit{2.5}{-10.5}{0.5}
\draw(3.5,-9.5)--(3.5,-11.5);
\draw(2.5,-10.5)--(2.5,-12.2);

\draw(1.5,-10.5) ..controls +(0,-0.5) and +(0,0.5).. (1,-11.2);
\coproduit{1.5}{-12.2}{0.5}
\produit{1.5}{-12.2}{0.5}

\draw [fill=white] (2.5,-11) circle (0.35);
\draw (2.5,-11) node {\footnotesize{$S$}};
\smallbraid{3}{-13}

\gauche{2}{-14}
\produit{0.5}{-15.5}{0.8}

\draw(0.5,-12.2) ..controls +(0,-1) and +(0,1).. (0,-14);
\draw(3.5,-11.5) ..controls +(0,-1) and +(0,1).. (3,-13);
\draw(3,-14) ..controls +(0,-1) and +(0,1).. (2.1,-15.5);
\draw [fill=white] (1.3,-16.2) circle (0.4);
\draw(1.3,-16.2) node {\footnotesize{$f$}};
\draw[below](1.5,-16.8) node{$M$};
\draw[below](-1,-13) node{\footnotesize{$=$}};
\draw[below](5.5,-13) node{\footnotesize{$=$}};
\draw[above](5.5,-13) node{$\eqref{cogebra}$};
\draw(8,-9.5)--(12,-9.5);
\draw(8,-17.3)--(12,-17.3);
\draw[above](9, -9.5) node{$A$};
\draw[above](11.5,-9.5) node{$V$};
\coproduit{9.5}{-10.5}{0.5}
\coproduit{10}{-11.5}{0.5}
\draw(11.5, -9.5)--(11.5,-12.5);

\draw(10,-11.5)--(10,-12.5);
\draw(9,-11.5)--(9,-12.5);
\produit{9}{-12.5}{0.5}
\draw [fill=white] (10,-12) circle (0.35);
\draw (10,-12) node {\footnotesize{$S$}};
\smallbraid{10.5}{-13.5}
\gauche{9.5}{-14.5}
\draw(8.5,-10.5) ..controls +(0,-1) and +(0,1).. (7.5,-12.5);
\draw(11.5,-12.5) ..controls +(0,-1) and +(0,1).. (10.5,-13.5);
\draw(7.5,-12.5)--(7.5,-14.5);
\produit{8}{-16}{0.8}
\draw(10.5,-14.5) ..controls +(0,-1) and +(0,1).. (9.6,-16);
\draw[below](13.5,-13) node{\footnotesize{$=$}};
\draw[above](13.5,-13) node{\footnotesize{$\eqref{S}$}};
\draw[below](9,-17.2) node{$M$};
\draw [fill=white] (8.8,-16.7) circle (0.4);
\draw(8.8,-16.7) node {\footnotesize{$f$}};
\draw(16,-9.5)--(20,-9.5);
\draw(16,-16.8)--(20,-16.8);
\draw[above](17, -9.5) node{$A$};
\draw[above](19,-9.5) node{$V$};
\draw(19,-9.5)--(19,-13);
\coproduit{17.5}{-10.5}{0.5}
\draw(17.5,-11) node{$\bullet$};
\draw(17.5,-10.5)--(17.5,-11);
\draw(17.5,-12) node{$\bullet$};
\smallbraid{19}{-13}
\gauche{18}{-14}
\produit{16.5}{-15.5}{0.8}
\draw(16,-12)--(16,-14);
\draw(17.5,-12) ..controls +(0,-0.7) and +(0,0.7).. (18,-13);
\draw(19,-14) ..controls +(0,-1) and +(0,1).. (18.1,-15.5);
\draw(16.5,-10.5) ..controls +(0,-1) and +(0,1).. (16,-12);
\draw [fill=white] (17.3,-16.2) circle (0.4);
\draw(17.3,-16.2) node {\footnotesize{$f$}};
\draw[below](17.2,-16.8) node{$M$};
\draw[below](21.5,-13) node{\footnotesize{$=$}};
\draw(24,-11)--(28,-11);
\draw[above](24.5, -11) node{$A$};
\draw[above](26.5,-11) node{$V$};
\gauche{26.5}{-11}
\produit{25.5}{-13.2}{0.8}
\draw(27.1,-12.5) node{$\bullet$};
\draw(27.1,-12.5)--(27.1,-13.2);
\draw [fill=white] (26.3,-14) circle (0.4);
\draw(26.3,-14) node {\footnotesize{$f$}};
\draw(26.3,-14.5)--(26.3,-15);
\draw(24,-15)--(28,-15);
\draw(25,-12.5) ..controls +(0,-0.4) and +(0,0.5).. (25.5,-13.2);
\draw[below](26.3,-15) node{$M$};
\draw[below](29,-13) node{\footnotesize{$.$}};
\end{tikzpicture}
\]
In the above computation, the equality $(*)$ arises from the fact that $f$ is a morphism in $_{A}\CC_{A}$. We also have
the map
\begin{align*}
\Psi_{V,M} : \text{Hom}_{_{A}\CC}(V, \widetilde{M}) &\longrightarrow \text{Hom}_{_{A}\CC_{A}}(V \boxtimes A, M)\\
g &\longmapsto \tilde{g} = \mu^{r}_{M}\circ (g \otimes \id_{A}).
\end{align*}
where we check similarly that $\tilde{g}$ is a morphism in $_{A}\CC_{A}$. It is then straightforward to check that $\Phi_{V,M}$ and $\Psi_{V,M}$ are inverse natural isomorphims, and we conclude that the functor $R$ is right adjoint to $L = - \boxtimes A$.
\epf

We obtain the following generalization of \cite[Proposition 5.6]{GK93}.

\bc \label{cor:equalpd}
Let $\CC$ be an abelian $k$-linear braided category with enough projectives and let $A$ be a Hopf algebra in $\CC$.  There exists a natural  isomorphism
\[\normalfont \text{Ext}^{*}_{{}_{A}\CC_{A}}(A, M) \simeq \text{Ext}^{*}_{{}_{A}\CC}(_\varepsilon I, \widetilde{M}).\]
and we have $\normalfont \text{pd}_{_{A}\CC_{A}}(A) = \text{pd}_{_{A}\CC}(_\varepsilon I) = \pd_{\CC_A}(I_\varepsilon).$
\ec

\bpf
The categories $_A\CC_A$ and $_A \CC$  are abelian $k$-linear by Proposition \ref{prop:modabelian},  the adjoint functors in Proposition \ref{prop:adjointmodulebimodule} are exact, and we have clearly $_\varepsilon I \boxtimes A=A$. Hence, the announced natural isomorphisms are obtained from Proposition \ref{prop:abelian}, and we get $\pd_{_{A}\CC_{A}}(A)\leq \pd_{_A\CC}(_\varepsilon I)$ as well.

For a left $A$-module $M$, consider the $A$-bimodule $M_\varepsilon$ as in the end of Subsection \ref{sub:braidedHopf}. Then the morphism $M \to \widetilde{M_{\varepsilon}}$ which is the identity in $\CC$ is an isomorphism in $_A\CC$, because
\begin{center}
\begin{tikzpicture}[scale=0.4]
\draw(-1, 0)--(3.5, 0);
\draw(1, 0)[above]node{$A$};
\draw(2.5, 0)[above]node{$M$};
\draw(2.5, 0)--(2.5, -2.5);
\draw(1, 0)--(1, -1);
\coproduit{1.5}{-1.5}{0.5}
\draw(1.5, -1.5)--(1.5, -2.5);
\draw[fill=white] (1.5,-2) circle (0.4);
\draw (1.5,-2) node {\footnotesize{$S$}};
\smallbraid{2.5}{-2.5}
\gauche{1.5}{-3.5}
\draw(0.5,-2) ..controls +(0,-1) and +(0,1).. (-0.5,-3.5);
\draw(0.5, -1.5)--(0.5, -2);
\draw(2.5, -3.5)--(2.5, -4.5);
\draw(2.5, -4.5) node{$\bullet$};
\draw(-1, -5)--(3.5, -5);
\draw(4.5, -2.5) node{$=$};
\draw(7, 0)--(10.5, 0);
\draw(8, 0)[above]node{$A$};
\draw(9.5, 0)[above]node{$M$};
\draw(9.5, 0)--(9.5, -3.5);
\draw(8, 0)--(8, -1);
\coproduit{8.5}{-1.5}{0.5}
\draw(8.5, -1.5)--(8.5, -2.5);
\draw[fill=white] (8.5,-2) circle (0.4);
\draw (8.5,-2) node {\footnotesize{$S$}};
\gauche{9.5}{-3.5}
\draw(2.5, -4.5) node{$\bullet$};
\draw(7, -5)--(10.5, -5);
\draw(7.5, -1.5)--(7.5, -3.5);
\draw(8.5, -2.5)--(8.5, -3.5);
\draw (8.5, -3.5) node{$\bullet$};
\draw(11.5, -2.5) node{$=$};
\draw(13, -1.5)--(17, -1.5);
\draw(14, -1.5)[above]node{$A$};
\draw(16, -1.5)[above]node{$M$};
\gauche{16}{-1.5}
\draw(13, -3)--(17, -3);
\draw(18, -2.5) node{$.$};
\end{tikzpicture}
\end{center}
Thus $\widetilde{M_{\varepsilon}} \simeq M$ in $_A\CC$ and we obtain
$$\text{Ext}^{*}_{{}_{A}\CC_{A}}(A, M_{\varepsilon}) \simeq \text{Ext}^{*}_{{}_{A}\CC}(_{\varepsilon}I, \widetilde{M_{\varepsilon}}) \simeq \text{Ext}^{*}_{{}_{A}\CC}(_{\varepsilon}I, M).$$
Hence $\text{pd}_{_{A}\CC}(I) \leq \text{pd}_{_{A}\CC_{A}}(A)$, and it follows that
 $\text{pd}_{_{A}\CC_{A}}(A) = \text{pd}_{_{A}\CC}(I)$. The equality $\pd_{\CC_A}(I_\varepsilon)= \text{pd}_{_{A}\CC_{A}}(A)$ is obtained by applying the left case to the reverse category $\CC^{\rm rev}$.
\epf

\bt \label{thm:equaldimen}
Let $A$ be a Hopf algebra in the braided category $\MM^{H}$ of comodules over a coquasitriangular cosemisimple Hopf algebra $H$. Then we have
\[\cd(A)= \lgd(A)=\rgd(A)=\pd_A(_\varepsilon k) = \pd_{A^{\rm op}}(k_\varepsilon).\]
\et

\bpf We have
\begin{enumerate}
 \item $\pd_{_A\mathcal M_A^H}(A) = \pd_{_A\mathcal M^H}(_\varepsilon k)$ by Corollary \ref{cor:equalpd};
\item $\pd_{_A\mathcal M^H}(_\varepsilon k) = \pd_{_A\mathcal M}(_\varepsilon k)$ by Proposition \ref{prop:pdHopfmod};
\item   $\pd_{_A\mathcal M_A}(A) \leq \pd_{_A\mathcal M_A^H}(A)$ by \cite[Corollary 11]{bichon2022monoidal} (which follows from Proposition \ref{prop:abelian}, because the (exact) forgetful functor $_A\mathcal M_A^H \to \,_A\mathcal M_A$ has an exact right adjoint).
\end{enumerate}
Hence we obtain
\[\lgd(A) \leq \cd(A) = \pd_{_A\mathcal M_A}(A) \leq \pd_{_A\mathcal M_A^H}(A)  =\pd_{_A\mathcal M^H}(_\varepsilon k) =  \pd_{_A\mathcal M}(_\varepsilon k)\leq \lgd(A)
\]
which gives the announced equality for left global dimension, and the one for right global dimension is obtained similarly.
\epf

\begin{rmq}\label{rem:ext}
 Let $A$ be a Hopf algebra in the braided category $\MM^{H}$ of comodules over a coquasitriangular  Hopf algebra $H$. Then for an object $M \in {_A\MM_A^H}$, the isomorphisms 
\[\e^{*}_{_{A}\MM_{A}^H}(A, M) \simeq \text{Ext}^{*}_{{}_{A}\MM^H}(_\varepsilon k, \widetilde{M})\]
in Corollary \ref{cor:equalpd} are valid without assuming that $\MM^H$ has enough projectives. This follows from Proposition \ref{prop:adjointmodulebimodule}, Proposition \ref{prop:pdHopfmod} (the categories ${_A \MM_A^H}$ and ${_A\MM^H}$ have enough injectives) and Proposition \ref{prop:abelian}. In particular, combining this with Proposition \ref{prop:pdHopfmod} gives, for any $A$-bimodule $M$, the following description for Hochshild cohomology:
\begin{align*}
H^{*}(A, M) &= \ext^{*}_{\,_{A}\MM_{A}}(A, M) \simeq \ext^{*}_{\,_{A}\MM_{A}^{H}}(A, M \odot H)  \\&\simeq \ext^{*}_{\,_{A}\MM^{H}}(_{\varepsilon}k, \stackon[-8pt]{$M \odot H$}{\vstretch{1.5}{\hstretch{1.8}{\widetilde{\phantom{\;\;\;\;\;\;\;\;}}}}}).  \\
\end{align*}
\end{rmq}

\section{Finiteness conditions and smoothness}\label{sec:finiteness}

In this section we use the previous constructions to obtain a convenient smoothness criterion for a Hopf algebra in the braided category of comodules over a coquasitriangular Hopf algebra.

Let us first recall the following standard finiteness condition \cite{brown} on a module over an ordinary $k$-algebra $A$: a left $A$-module $M$ is said to be of \textsl{type} ${\rm FP}_\infty$ if it admits a projective resolution by  finitely generated and projective $A$-modules, and is said to be of \textsl{type} FP if it admits a finite projective resolution by finitely generated $A$-modules, which means that there is an exact sequence of $A$-modules
\[0 \to P_n \to P_{n-1} \to \cdots P_2\to P_1 \to P_0 \to M\to 0\]
{where each $P_i$ is a finitely generated and projective $A$-module. A similar definition holds for right modules and for bimodules, and an algebra $A$ is said to be \textsl{smooth} if $A$ is of type FP as an $A$-bimodule.

To adapt the definition of a module of type FP to a more general monoidal category $\CC$, recall first that an object $V$ in $\CC$ is said to have a left dual if there exists an object $V^*$ together with morphisms $e : V^*\otimes V \to I$ and $\delta : I \to V\otimes V^*$ such that
\[ (\id_V\otimes \, e ) \circ (\delta \otimes \id_V ) = \id_V, \quad (e\otimes \id_{V^*}) \circ (\id_{V^*}\otimes \, \delta ) = \id_{V^*}\]
When $\CC = {_k\mathcal M}$, a vector space $V$ has a left dual if and only if it is finite dimensional.

\bd\label{def:fp}
Let $\CC$ be an abelian $k$-linear monoidal category and let $A$ be an algebra in $\CC$. 
\begin{enumerate}
 \item A left $A$-module $M$ is said to be relative projective if $M$ is isomorphic, as an $A$-module, to a direct summand of a free $A$-module $A \otimes V$.
 \item A left $A$-module $M$ is said to be finite relative projective if $M$ is isomorphic, as an $A$-module, to a direct summand of a free $A$-module $A \otimes V$, with $V$ an object of $\CC$ having a left dual.
\item A left $A$-module $M$ is said to be of type ${\rm FP}_\infty$ if it has a resolution by finite relative projective $A$-modules. 
\item A left $A$-module $M$ is said to be of type FP if it has a finite resolution by finite relative projectives, in the sense that there exists an exact sequence of $A$-modules
\[0 \to P_n \to P_{n-1} \to \cdots P_2\to P_1 \to P_0 \to M\to 0\]
where for each $i$, the $A$-module $P_i$ is finite relative projective. 
\end{enumerate}
\ed

Of course similar definitions hold for right $A$-modules and for $A$-bimodules.

\bp\label{prop:free2free}
Let $\CC$ be a braided category and let $A$ be an algebra in $\CC$.
The functor $L = - \boxtimes A: \, _{A}\CC \longrightarrow \,_{A}\CC_{A}$ transforms free $A$-modules into free $A$-bimodules. If moreover $\CC$ is an abelian $k$-linear braided category, then the functor $L$ transforms objects that are of type ${\rm FP}$ (resp of type ${\rm FP}_\infty$)  in $_A\CC$ into objects that are of type ${\rm FP}$ (resp. of type ${\rm FP}_\infty$) in $_{A}\CC_{A}$.
\ep

\bpf
One has to prove that for an object $V$ of $\CC$, the $A$-bimodule $(A\otimes V)\boxtimes A$ is isomorphic to the free $A$-bimodule $A\otimes V \otimes A$. Consider the linear map $f : A\otimes V \otimes A \longrightarrow (A\otimes V)\boxtimes A$ defined by
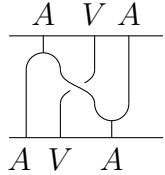
\[\begin{tikzpicture}[scale=0.45]
\draw(-0.5,0)--(4,0);
\draw(-0.5,-3)--(4,-3);
\draw[above](0.5,0) node{$A$};
\draw[above](2,0) node{$V$};
\draw (0,-1)--(0,-3);
\draw (1,-2)--(1,-3);
\draw (2,0)--(2,-1);
\draw (3,0)--(3,-2);
\draw[above](3,0) node{$A$};
\coproduit{1}{-1}{0.5}
\smallbraid{2}{-1}
\produit{2}{-2}{0.5}
\draw[below](-0.2,-3) node{$A$};
\draw[below](1,-3) node{$V$};
\draw[below](2.5,-3) node{$A$};
\end{tikzpicture}\]

We first check that $f \circ (m_{A} \otimes \id_{V \otimes A}) = \mu^{l}_{(A\otimes V)\otimes A} \circ (\id_{A} \otimes f)$, 
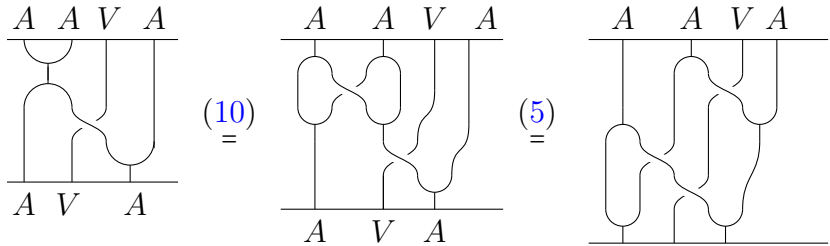
\[\begin{tikzpicture}[scale=0.45]
\draw(0,0)--(5,0);
\draw(0,-4.2)--(5,-4.2);
\draw[above](0.5,0) node{$A$};
\draw[above](1.8,0) node{$A$};
\draw[above](3,0) node{$V$};
\draw[above](4.3,0) node{$A$};
\produit{0.5}{0}{0.7}
\coproduit{1.9}{-2}{0.7}
\smallbraid{2.9}{-2}
\draw(2.9,0)--(2.9,-2);
\draw(4.3,0)--(4.3,-3.2);
\produit{2.9}{-3}{0.7}
\draw(0.5,-2)--(0.5,-4.2);
\draw(1.9,-3)--(1.9,-4.2);
\draw[below](0.5,-4.2) node{$A$};
\draw[below](1.8,-4.2) node{$V$};
\draw[below](3.75,-4.2) node{$A$};
\draw(6.5,-3) node{$=$};
\draw[above](6.5,-3) node{$\eqref{Hopf}$};
\draw(8,0)--(14.5,0);
\draw(8,-5)--(14.5,-5);
\draw[above](9,0) node{$A$};
\draw[above](11,0) node{$A$};
\draw[above](12.5,0) node{$V$};
\draw[above](14,0) node{$A$};
\coproduit{9.5}{-1}{0.5}
\coproduit{11.5}{-1}{0.5}
\smallbraid{10.5}{-1}
\produit{8.5}{-2}{0.5}
\produit{10.5}{-2}{0.5}
\smallbraid{12}{-3}
\produit{12}{-4}{0.5}
\draw(8.5,-1)--(8.5,-2);
\draw(11.5,-1)--(11.5,-2);
\draw(12.5,0)--(12.5,-1.5);
\draw(13.5,0)--(13.5,-2.5);
\draw(12.5,-1.5) ..controls +(0,-1) and +(0,1).. (12,-3);
\draw(13.5,-2.5) ..controls +(0,-1) and +(0,1).. (13,-4);
\draw(11,-4)--(11,-5);
\draw(9,-3)--(9,-5);
\draw[below](11,-5) node{$V$};
\draw[below](9,-5) node{$A$};
\draw[below](12.5,-5) node{$A$};
\draw(15.5,-3) node{$=$};
\draw[above](15.5,-3) node{$\eqref{associatif}$};
\draw(17,0)--(24,0);
\draw(17,-6)--(24,-6);
\draw[above](18,0) node{$A$};
\draw[above](20,0) node{$A$};
\draw[above](21.5,0) node{$V$};
\draw[above](22.5,0) node{$A$};
\coproduit{20.5}{-1}{0.5}
\smallbraid{21.5}{-1}
\produit{21.5}{-2}{0.5}
\coproduit{18.5}{-3}{0.5}
\smallbraid{19.5}{-3}
\smallbraid{20.5}{-4}
\draw(18,0)--(18,-2);
\draw(19.5,-1)--(19.5,-3);
\draw(20.5,-2)--(20.5,-4);
\produit{20.5}{-5}{0.5}
\produit{17.5}{-5}{0.5}
\draw(17.5,-3)--(17.5,-5);
\draw(18.5,-4)--(18.5,-5);
\draw(19.5,-5)--(19.5,-6);
\draw(21.5,0)--(21.5,-1);
\draw(22.5,0)--(22.5,-2);
\draw(22,-3) ..controls +(0,-1) and +(0,1).. (21.5,-5);
\draw(11,-4)--(11,-5);
\draw(25,-3) node{$.$};
\end{tikzpicture}\]
So $f$ is a morphism in $_{A}\CC$ and it is not difficult to verify that $f$ is a morphism in $\CC_{A}$. Thus $f$ is indeed a morphism in $_A\CC_{A}$. Moreover $f$ is an isomorphism with  inverse
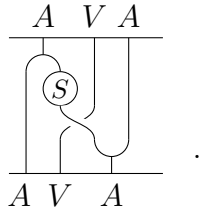
\[\begin{tikzpicture}[scale=0.45]
\draw(-0.5,0)--(4,0);
\draw(-0.5,-4)--(4,-4);
\draw[above](0.5,0) node{$A$};
\draw[above](2,0) node{$V$};
\draw (0,-1)--(0,-4);
\draw (1,-3)--(1,-4);
\draw (2,0)--(2,-2);
\draw (3,0)--(3,-3);
\draw[above](3,0) node{$A$};
\coproduit{1}{-1}{0.5}
\draw[fill=white] (1,-1.5) circle (0.5);
\draw (1,-1.5) node {\footnotesize{$S$}};
\smallbraid{2}{-2}
\produit{2}{-3}{0.5}
\draw[below](-0.2,-4) node{$A$};
\draw[below](1,-4) node{$V$};
\draw[below](2.5,-4) node{$A$};
\draw(5,-3.5) node{$.$};
\end{tikzpicture}\]
This means the functor $L$ transforms free $A$-modules into free $A$-bimodules, and the remaining statements follow from the exactness of $L$.
\epf

We get the announced smoothness criterion.

\bt\label{thm:smoothness}
Let $A$ be a Hopf algebra in the braided category $\MM^{H}$ of comodules over a coquasitriangular Hopf algebra $H$. If $_\varepsilon k$ is of type FP in $_A\mathcal M^H$, then $A$ is a smooth algebra.
\et

\bpf
We have $L(_\varepsilon k) \simeq A$, hence $A$ is of type FP in  $_A\mathcal M^H_A$ if $_\varepsilon k$ is of type FP in $_A\mathcal M^H$, by  Proposition \ref{prop:free2free}. It follows that $A$ is of type FP in $_A\mathcal M_A$, since finite relative projective objects in $_A\mathcal M^H_A$ clearly are finitely generated projective $A$-bimodules.
\epf 

\begin{rmq}
 Let $H$ be a Hopf algebra and let $A$ be a right $H$-comodule algebra. If $A$ is (left) Noetherian, then every object in $_A\mathcal M^H$ that is finitely generated as an $A$-module is of type ${\rm FP}_\infty$ in $_A\mathcal M^H$.
\end{rmq}

\begin{proof}
 This is similar to the usual argument with ordinary modules: let $M$ be an object in $_A\mathcal M^H$ and let $V\subset M$ be a finite-dimensional subspace that generates $M$ as an $A$-module. Then there is a finite-dimensional $H$-subcomodule $W$ of $M$ that contains $V$, and a surjective $A$-linear and $H$-colinear map $A\otimes W \to M$. The kernel is an object of $_A\mathcal M^H$ and is finitely generated by Noetherianity of $A$, so we can repeat the process to get the desired resolution of $M$.
\end{proof}

We now use the ${\rm FP}_\infty$ condition to construct a comodule structure on certain $\ext$ spaces. This will be used in the next section and relies on the following observation.

\bl\label{Lem:Cohom}
Let $H$ be a  Hopf algebra with bijective antipode and let $A$ be an $H$-comodule algebra.
 Let $P$ be a finite relative projective object in $\,_{A}\MM^{H}$. Then there is a map
\begin{align*}
\delta : \Hom_{A}(P, A) &\longrightarrow \Hom_{A}(P, A) \otimes H\\
f &\longmapsto f_{(0)} \otimes f_{(1)}
\end{align*}
such that for all $x \in P$, 
\begin{align}
 f_{(0)}(x) \otimes f_{(1)} = f(x_{(0)})_{(0)} \otimes S_{H}^{-1}(x_{(1)})f(x_{(0)})_{(1)}\label{CoAc:Hom}
\end{align}
that endows $\Hom_{A}(P, A)$ with an $H$-comodule structure, and makes it into an object in $\mathcal M_A^H$ ($\Hom_{A}(P, A)$ being endowed with its natural right $A$-module structure).
\el 

\bpf
Start with the the map 
\begin{align*}
\delta_0 : \Hom_{A}(P, A) &\longrightarrow \Hom_{A}(P, {_AA}\otimes H)\\
f &\longmapsto \delta_0(f), \ \delta_0(f)(x)=  f(x_{(0)})_{(0)} \otimes S_{H}^{-1}(x_{(1)})f(x_{(0)})_{(1)}
\end{align*}
Let us check that $\delta_0(f)$ is indeed $A$-linear. 
For $a \in A$ and $x\in P$, we have
\begin{align*}
 \delta_0(f)(a.x) &=  f((a.x)_{(0)})_{(0)} \otimes S_{H}^{-1}((a.x)_{(1)})f((a.x)_{(0)})_{(1)} \\
&= f(a_{(0)}.x_{(0)})_{(0)} \otimes S_{H}^{-1}(a_{(1)}x_{(1)})f(a_{(0)}.x_{(0)})_{(1)} \\
&= \big(a_{(0)}f(x_{(0)})\big)_{(0)} \otimes S_{H}^{-1}(a_{(1)}x_{(1)})\big(a_{(0)}f(x_{(0)})\big)_{(1)} \\
 &= a_{(0)}f(x_{(0)})_{(0)} \otimes S_{H}^{-1}(a_{(2)}x_{(1)})a_{(1)}f(x_{(0)})_{(1)} \\
 &= a_{(0)}f(x_{(0)})_{(0)} \otimes S_H^{-1}(x_{(1)})S^{-1}_{H}(a_{(2)})a_{(1)}f(x_{(0)})_{(1)} \\
& =  af(x_{(0)})_{(0)} \otimes S_H^{-1}(x_{(1)})f(x_{(0)})_{(1)}.
\end{align*}
Since $P$ is finite relative projective in  $\,_{A}\MM^{H}$, it is in particular finitely generated projective as an $A$-module and the map 
\begin{align*}
 \Hom_{A}(P, A) \otimes H &\longrightarrow \Hom_{A}(P, {_AA}\otimes H) \\
\sum_i \varphi_i\otimes h_i &\longmapsto (x \mapsto \sum_i \varphi_i(x)\otimes h_i) 
\end{align*}
is an isomorphim. The map $\delta$ is then obtained from the composition of the inverse of this map and $\delta_0$.
Now we check that $\delta$ is coassociative.
Applying $(\id \otimes \Delta_{H})$ to  \eqref{CoAc:Hom}, we have, for all $x \in P$,
\begin{align*}
f_{(0)}(x) \otimes \Delta_{H}(f_{(1)})&= f(x_{(0)})_{(0)} \otimes S_{H}^{-1}(x_{(1)})_{(1)}f(x_{(0)})_{(1)} \otimes S_{H}^{-1}(x_{(1)})_{(2)}f(x_{(0)})_{(2)}\\
&=f(x_{(0)})_{(0)} \otimes S_{H}^{-1}(x_{(2)})f(x_{(0)})_{(1)} \otimes S_{H}^{-1}(x_{(1)})f(x_{(0)})_{(2)}.
\end{align*}
On the other hand we have
\begin{align*}
 f_{(0)(0)}(x)\otimes f_{(0)(1)} \otimes f_{(1)} & = f_{(0)}(x_{(0)})_{(0)} \otimes S_{H}^{-1}(x_{(1)})f_{(0)}(x_{(0)})_{(1)} \otimes f_{(1)}\\
&= f(x_{(0)})_{(0)} \otimes S_{H}^{-1}(x_{(2)})f(x_{(0)})_{(1)} \otimes S_H^{-1}(x_{(1)})f(x_{(0)})_{(2)}
\end{align*}
and this proves the coassociativity. The counit property is an easy verification, and we have indeed defined the anounced comodule structure on $\Hom_{A}(P, A)$.

For $a \in A$, we have
\begin{align*}
(f\cdot a)_{(0)}(x) \otimes (f\cdot a)_{(1)} & = (f\cdot a)(x_{(0)})_{(0)} \otimes S_{H}^{-1}(x_{(1)})(f\cdot a)(x_{(0)})_{(1)}\\
& = (f(x_{(0)})a)_{(0)} \otimes S_{H}^{-1}(x_{(1)})(f(x_{(0)})a)_{(1)}\\
& = f(x_{(0)})_{(0)}a_{(0)} \otimes S_{H}^{-1}(x_{(1)})f(x_{(0)})_{(1)}a_{(1)}\\
&= f_{(0)}\cdot a_{(0)}(x)\otimes f_{(1)}a_{(1)} 
\end{align*}
and this shows that $\Hom_{A}(P, A)$ is indeed an object in $\mathcal M_A^H$.
\epf 

\bl \label{Lem:CoExt}
Let $H$ be a  Hopf algebra with bijective antipode, and let $A$ be an $H$-comodule algebra.
 Let $M$ be an $A$-module of type ${\rm FP}_\infty$  in $\,_{A}\MM^{H}$. For $n \in \N$, the comodule structure of Lemma \ref{Lem:Cohom} induces a map
\begin{align*}
\bar{\delta} : \ext^{n}_{A}(M, A) &\longrightarrow \ext^{n}_{A}(M, A) \otimes H\\
[f] &\longmapsto [f]_{(0)} \otimes [f]_{(1)} = [f_{(0)}] \otimes f_{(1)}
\end{align*}
making $\ext^{n}_{A}(M, A)$ into an $H$-comodule, and an object in $\mathcal M_A^H$.
\el

\bpf
Let $(P_{*}, d_{*}) \longrightarrow M$ be a projective resolution of $M$ in $\,_{A}\MM^H$ such that for each $n$, $P_{n}$ is finite relative projective. Each  $\Hom_{A}(P_{n},A)$ inherits the $H$-comodule structure of Lemma \ref{Lem:Cohom}, and 
to prove our assertion, it is enough to check that the following diagram commutes:
\[
\begin{tikzcd}
\Hom_{A}(P_{n},A)\arrow[rr, "-\circ d_{n+1}"] \arrow[d, "\delta"]&& 
\Hom_{A}(P_{n+1},A) \arrow[d, "\delta"] \\
\Hom_{A}(P_{n},A)\otimes H \arrow[rr, "(-\circ d_{n+1})\otimes {\rm id}_H"]&&  \Hom_{A}(P_{n+1},A)\otimes H
\end{tikzcd}\]
For $f \in \Hom_{A}(P_{n}, A)$ and $x \in P_{n+1}$ we have, using the colinearity of $d_{n+1}$,
\begin{align*}
f_{(0)}\circ d_{n+1}(x) \otimes f_{(1)} &= f(d_{n+1}(x)_{(0)})_{(0)} \otimes S_{H}^{-1}(d_{n+1}(x)_{(1)}) f(d_{n+1}(x)_{(0)})_{(1)}\\
&= f(d_{n+1}(x_{(0)}))_{(0)} \otimes S_{H}^{-1}(x_{(1)}) f(d_{n+1}(x_{(0)}))_{(1)}\\
& =(f\circ d_{n+1})_{(0)}(x) \otimes (f\circ d_{n+1})_{(1)}
\end{align*}
and this concludes the proof of the lemma.
\epf

We conclude this section by a last lemma, that we will use in Section \ref{sec:twop}.

\begin{lemma}\label{lem:gp-trivial}
 Let $H$ be a  Hopf algebra with bijective antipode, and let $A$ be an $H$-comodule algebra.
 Let $M$ be a $A$-module of type ${\rm FP}$  in $\,_{A}\MM^{H}$ having a  resolution 
\[0 \to P_n \to P_{n-1} \to \cdots P_2\to P_1 \to P_0 \to M\to 0\]
with each $A$-module $P_i$ finite relative projective and $P_n=A$. Assume moreover that  $\ext^{n}_{A}(M, A)$ is one dimensional. Then the group-like corresponding to the $H$-comodule  structure on  $\ext^{n}_{A}(M, A)$ in Lemma \ref{Lem:CoExt} is trivial.
\end{lemma}

\sloppy
\bpf
It is a straightforward verification that under the identification $A\simeq {\rm Hom}_A(A, A)$, the $H$-coaction of Lemma \ref{Lem:Cohom} on ${\rm Hom}_A(A, A)$ is given by $a \mapsto a_{(0)} \otimes a_{(1)}$. The right $A$-module $\ext^{n}_{A}(M, A)$ is the cokernel of the right $A$-linear map ${\rm Hom}_{A}(P_{n-1},A) \to  {\rm Hom}_A(A, A) \simeq A$, which is surjective if and only if $1_A$ belongs to its image. The assumption $\dim_k(\ext^{n}_{A}(M, A))=1$ then ensures that the class of $1_A$ generates the vector space $\ext^{n}_{A}(M, A)$ and our first observation in the proof thus ensures that the $H$-coaction is trivial. 
\epf

\section{Homological duality} \label{sec:homodual}

In this section we provide a convenient criterion that ensures that a braided Hopf algebra in a comodule category is twisted Calabi-Yau.

\subsection{Twisted Calabi-Yau algebras} We begin the section by recalling the concept of twisted Calabi-Yau algebra.
Recall that if $R$ is an algebra and $M, N$ are left $R$-modules with $N$ is a right $S$-module for another algebra $S$ such that $N$ is a $R$-$S$-bimodule, then the space of right $R$-linear maps $\Hom_{R}(M, N)$ carries a natural right $S$-module structure defined by \[(f \cdot s)(x) = f(x)\cdot s.\]
This induces   a natural right $S$-module structure on $\ext^{*}_{R}(M, N)$. In particular the Hochschild cohomology spaces 
$H^*(A, {_AA\otimes A_A})= \ext_{A^e}(A,A^e)$ are naturally right $A^e$-modules, hence $A$-bimodules. The following condition appears in  \cite{BZ08} under the name \textsl{rigid Gorenstein}, see for example \cite{LWW14} for the present terminology.

\bd
An algebra $A$ is said to be twisted Calabi-Yau of dimension $n\geq 0$ if A is smooth and 
 \[H^{i}(A, \,_{A}A \otimes A_{A}) \simeq 
\begin{cases}
	\{0\} \quad \text{if} \quad i \neq n\\
	A_{\mu} \quad \text{if} \quad i = n
\end{cases}
\]
as $A$-bimodules, for an algebra automorphism $\mu \in {\rm Aut}(A)$, called the Nakayama automorphism of $A$. 
\ed

The interest in the twisted Calabi-Yau condition comes from the fact that it induces a duality between the Hochschild homologies and cohomologies \cite{VDB}: if $A$ is a twisted Calabi-Yau algebra of dimension $n$ with Nakayama automorphism $\mu$, then necessarily $n=\cd(A)$, and if $M$ is an $A$-bimodule, then we have for any $i \geq 0$
\[H^i(A,M) \simeq H_{n-i}(A, {\,_{\mu^{-1}}  M})\]

We will prove the following result, which generalizes the usual result \cite[Corollary 5.2]{BZ08} for ordinary Hopf algebras.

\bt \label{thm:tcy} Let $A$ be a Hopf algebra with bijective antipode in the braided category $\MM^{H}$ of comodules over a coquasitriangular Hopf algebra $H$. 
Assume that the $A$-module $_\varepsilon k$ is of type FP  in $_A\mathcal M^H$ and that there is an integer $n\geq 0$ such that $\ext_A^i(_\varepsilon k, A)=\{0\}$ for $i\not =n$ and $\ext_A^n(_\varepsilon k, A)$ is one-dimensional. Then $A$ is twisted Calabi-Yau of dimension $n$, with Nakayama automorphism defined by 
 \[\mu(a) = \psi(a_{[1]})\, \normalfont \textbf{r}\big(a_{[2](1)}, S_{H}(a_{[2](2)})g^{-1}\big)  S_{A}^{2}(a_{[2](0)})\] 
where $\psi : A\to k$ is the algebra map corresponding to the $A$-module structure on $\ext_A^n(_\varepsilon k, A)$ and  satisfies $\psi(a_{(0)}) a_{(1)} = \psi(a)1$ for any $a\in A$,  
and $g \in H$ is the group-like element corresponding to the $H$-comodule structure on $\ext_A^n(_\varepsilon k, A)$ from Lemma \ref{Lem:CoExt}.
\et

The rest of the section is devoted to the proof of Theorem \ref{thm:tcy}.

\subsection{The structure of \texorpdfstring {$H^*(A, \,_AA\otimes A_A)$}{H(A, AA x AA)}} Theorem \ref{thm:tcy} will be a consequence of the following result.

\bt\label{Thm:Hoch}
Let $A$ be a Hopf algebra in the braided categry $\MM^H$ of comodules over a coquasitriangular  Hopf algebra $H$. If $\,_{\varepsilon} k$ is of type $\text{FP}_\infty$ in $\,_A\MM^H$, then there is an isomorphism of right $A^{e}$-modules
\[H^{*}(A, \,_{A}A \otimes A_{A}) \simeq \ext^{*}_{A}(_{\varepsilon}k, \,_{A}A) \otimes A\]
where the right $A^{e}$-action on $\ext^{*}_{A}(_{\varepsilon}k, \,_{A}A) \otimes A$ is defined by
\[([f] \otimes a') \cdot (a \otimes b) = ([f ]\cdot a_{[1]})_{(0)} \otimes ba'S^{2}_{A}(a_{[2](0)}) \normalfont\textbf{r}\big[a_{[2](1)}, S_{H}\big(a_{[2](2)}\big)S_{H}\big(([f] \cdot a_{[1]})_{(1)} \big)\big] \]
with the right $A$-structure on $\ext^{*}_{A}(_{\varepsilon}k, \,_{A}A)$  induced by right multiplication in $A$ and the right $H$-comodule structure being the one of Lemma \ref{Lem:CoExt}.
\et

Taking Theorem \ref{Thm:Hoch} for granted, the proof of Theorem \ref{thm:tcy} follows easily:

\begin{proof}[Proof of Theorem \ref{thm:tcy}]
Let $A$ be a Hopf algebra in the braided categry $\MM^H$ of comodules over a coquasitriangular  Hopf algebra $H$.   If $\,_{\varepsilon} k$ is of type FP in $\,_A\MM^H$, we know from Theorem \ref{thm:smoothness} that $A$ is smooth. Assuming moreover that $\ext_A^i(_\varepsilon k, A)=\{0\}$ for $i\neq n$, we obtain from Theorem \ref{Thm:Hoch} that  $H^{i}(A, \,_{A}A \otimes A_{A}) = \{0\}$ for $i\neq n$. Assume finally that   $\ext_A^n(_\varepsilon k, A)$ is one dimensional and let $\psi : A\to k$ be the algebra map corresponding to the $A$-module structure on $\ext_A^n(_\varepsilon k, A)$ 
and $g \in H$ be the group-like element corresponding to the $H$-comodule structure on $\ext_A^n(_\varepsilon k, A)$ given by Lemma \ref{Lem:CoExt}. It is easily seen from the fact that $\ext_A^n(_\varepsilon k, A)$ is an object in $\mathcal M^H_A$ (Lemma \ref{Lem:CoExt}) that $\psi$ satisfies $\psi(a_{(0)}) a_{(1)} = \psi(a)1$ for any $a\in A$ (while there is no such condition on $g$). Then the right $A^e$-module on the right term of Theorem \ref{Thm:Hoch} is
\begin{align*}([f] \otimes a') \cdot (a \otimes b) &= ([f ]\cdot a_{[1]})_{(0)} \otimes ba'S^{2}_{A}(a_{[2](0)}) \textbf{r}\big[a_{[2](1)}, S_{H}\big(a_{[2](2)}\big)S_{H}\big(([f] \cdot a_{[1]})_{(1)} \big)\big] \\
&  = \psi(a_{[1]}) [f]_{(0)} \otimes ba'S^{2}_{A}(a_{[2](0)}) \textbf{r}\big[a_{[2](1)}, S_{H}\big(a_{[2](2)}\big)S_{H}\big([f]_{(1)} \big)\big] \\
&  =  [f] \otimes b a'\psi(a_{[1]})S^{2}_{A}(a_{[2](0)}) \textbf{r}\big[a_{[2](1)}, S_{H}\big(a_{[2](2)}\big)g^{-1}\big]
\end{align*}
and therefore Theorem \ref{thm:tcy} follows from Theorem \ref{Thm:Hoch}.
\end{proof}

\subsection{Proof of Theorem \ref{Thm:Hoch}}

We now fix a coquasitriangular Hopf algebra $H$ and a Hopf algebra $A$ in the braided category $\mathcal M^H$, 
for which we use the following Sweedler notation for the respective comultiplications and coaction of  $H$ on $A$: 
\[\Delta_A(a) = a_{[1]} \otimes a_{[2]}, \quad \Delta_H(x) = x_{(1)}\otimes x_{(2)}, \quad  \alpha(a) = a_{(0)} \otimes a_{(1)}.\]
The $H$-colinearity of $\Delta_A$ reads, for $a \in A$,
\begin{align}
a_{[1](0)} \otimes a_{[2](0)} \otimes a_{[1](1)}a_{[2](0)} &= a_{(0)[1]} \otimes a_{(0)[2]} \otimes a_{(1)} \label{morcoproduct}
\end{align}
and, if we apply $\id_{A} \otimes \, \alpha \otimes \id_{H}$ to both sides of this equality, we obtain:
\begin{align}
a_{[1](0)} \otimes a_{[2](0)} \otimes a_{[2](1)} \otimes a_{[1](1)}a_{[2](2)} &= a_{(0)[1]} \otimes a_{(0)[2](0)} \otimes a_{(0)[2](1)} \otimes a_{(1)} \label{morcoproduct1}
\end{align}
Since $S_{A}$ is also $H$-colinear, we have as well
\begin{align}
S_{A}(a)_{(0)} \otimes S_{A}(a)_{(1)} &= S_{A}(a_{(0)}) \otimes a_{(1)} \label{morantipode}.
\end{align}

The following result, which provides an isomorphism between two natural objects in $_A\mathcal M^H$, generalizes a known result for ordinary Hopf algebras \cite[Lemma 2.2]{BZ08}.

\bp \label{Thm:Iso}
Let $A$ be a Hopf algebra in the braided category $\MM^{H}$ of comodules over a coquasitriangular Hopf algebra $H$. We have an isomorphism in $\,_{A}\MM^{H}$: 
\[F : (\,_{A}A \otimes A) \boxdot H \longrightarrow \stackon[-8pt]{$(\,_{A}A \otimes A_{A}) \odot H$}{\vstretch{1.5}{\hstretch{3.1}{\widetilde{\phantom{\;\;\;\;\;\;\;\;}}}}} \]
defined by
\begin{align*}
F(a \otimes b \otimes h) &= a_{(0)}\cdot \big( 1 \otimes b \otimes S_{H}(a_{(1)})h\big)\\
&= a_{[1](0)} \otimes bS_{A}(a_{[2](0)}) \otimes h_{(2)} \, \normalfont\textbf{r}\left(a_{[2](1)}, S_{H}(a_{[1](1)}a_{[2](2)})h_{(1)}\right),
\end{align*}
with inverse 
\[G : \stackon[-8pt]{$\,_{A}A \otimes A_{A} \odot H$}{\vstretch{1.5}{\hstretch{3.1}{\widetilde{\phantom{\;\;\;\;\;\;\;\;}}}}} \longrightarrow \,_{A}A \otimes A \boxdot H\]
given by
\[G(a \otimes b \otimes h) =  a_{[1](0)} \otimes bS^{2}_{A}(a_{[2](0)}) \otimes h_{(2)} \, \normalfont \textbf{r}\left(a_{[2](1)}, S_{H}(a_{[1](1)}a_{[2](2)})h_{(1)}\right).\]
\ep

\bpf
Using the left $A$-module structure from Proposition \ref{prop:bimtoleft}, and  $\eqref{r1}$ and $\eqref{morcoproduct1}$, we have  
\begin{align*}
&F(a \otimes b \otimes h) = a_{(0)}\cdot \big( 1 \otimes b \otimes S_{H}(a_{(1)})h\big)\\
&= a_{(0)[1](0)} \otimes b S_{A}(a_{(0)[2](0)}) \otimes a_{(0)[1](1)}\left(S_{H}(a_{(1)})h\right)_{(1)}a_{(0)[2](1)} \textbf{r} \left[a_{(0)[2](2)}, \left(S_{H}(a_{(1)})h\right)_{(2)}\right]\\
&=  a_{(0)[1](0)} \otimes b S_{A}(a_{(0)[2](0)}) \otimes a_{(0)[1](1)}a_{(0)[2](2)}\left(S_{H}(a_{(1)})h\right)_{(2)} \textbf{r} \left[a_{(0)[2](1)}, \left(S_{H}(a_{(1)})h\right)_{(1)}\right]\\
&=  a_{(0)[1]} \otimes b S_{A}(a_{(0)[2](0)}) \otimes a_{(1)}S_{H}(a_{(2)})h_{(2)} \textbf{r} \left[a_{(0)[2](1)}, S_{H}(a_{(3)})h_{(1)}\right]\\
&= a_{(0)[1]} \otimes b S_{A}(a_{(0)[2](0)}) \otimes h_{(2)} \textbf{r} \left[a_{(0)[2](1)}, S_{H}(a_{(1)})h_{(1)}\right]\\
&= a_{[1](0)} \otimes b S_{A}(a_{[2](0)}) \otimes h_{(2)} \textbf{r} \left[a_{[2](1)}, S_{H}(a_{[1](1)}a_{[2](2)})h_{(1)}\right].
\end{align*}
We obtain the expression of $F$ as stated.
For $x \in A$, we have
\begin{align*}
F\left(x \cdot (a \otimes b \otimes h)\right) & = F \left(x_{(0)}a \otimes b \otimes x_{(1)}h\right)\\
&= (x_{(0)}a)_{(0)} \cdot \left(1 \otimes b \otimes S_{H}\big((x_{(0)}a)_{(1)}))x_{(1)}h\right)\\
&= (x_{(0)}a_{(0)}) \cdot \left(1 \otimes b \otimes S_{H}(x_{(1)}a_{(1)})x_{(2)}\right)\\
&= x \cdot \left(a_{(0)} \cdot \big(1 \otimes b \otimes S_{H}(a_{(1)})\big)\right)\\
&= x \cdot F(a \otimes b \otimes h).
\end{align*}
Hence $F$ is $A$-linear, and it is immediate to check that $F$ is also $H$-colinear.
It remains to prove that $F$ is indeed an isomorphism with inverse $G$. We have 
\begin{align*}
&F \circ G (a \otimes b \otimes h) = F\left(a_{[1](0)} \otimes bS^{2}_{A}(a_{[2](0)}) \otimes h_{(2)}\right) \, \textbf{r}\left(a_{[2](1)}, S_{H}(a_{[1](1)}a_{[2](2)})h_{(1)}\right)\\
&= a_{[1](0)[1](0)} \otimes bS^{2}_{A}(a_{[2](0)})S_{A}(a_{[1](0)[2](0)}) \otimes h_{(3)}\\ &\hspace*{1cm} \textbf{r}\left(a_{[1](0)[2](1)}, S_{H}(a_{[1](0)[1](1)}a_{[1](0)[2](2)})h_{(2)}\right)
\textbf{r}\left(a_{[2](1)}, S_{H}(a_{[1](1)}a_{[2](2)})h_{(1)}\right)\\ 
&= a_{[1](0)[1]} \otimes bS^{2}_{A}(a_{[2](0)})S_{A}(a_{[1](0)[2](0)}) \otimes h_{(3)} \, \textbf{r}\left(a_{[1](0)[2](1)}, S_{H}(a_{[1](1)})h_{(2)}\right)\\
&\hspace*{1cm}\textbf{r}\left(a_{[2](1)}, S_{H}(a_{[1](2)}a_{[2](2)})h_{(1)}\right) \hspace*{3cm}(\text{by using $\eqref{morcoproduct1}$ for $a_{[1](0)}$})\\
&= a_{[1](0)} \otimes bS^{2}_{A}(a_{[3](0)})S_{A}(a_{[2](0)}) \otimes h_{(3)} \, \textbf{r}\left(a_{[2](1)}, S_{H}(a_{[1](1)}a_{[2](2)})h_{(2)}\right)\\
&\hspace*{1cm} \textbf{r}\left(a_{[3](1)}, S_{H}(a_{[1](2)}a_{[2](3)}a_{[3](2)})h_{(1)}\right) \hspace*{1.5cm}(\text{applying again $\eqref{morcoproduct1}$ to $a_{[1]}$}) \\
&=a_{[1](0)} \otimes bS_{A}\left(a_{[2](0)} S_{A}(a_{[3](0)})\right)\otimes h_{(3)} \, \textbf{r}^{-1}\left(a_{[2](1)}, a_{[3](1)}\right) \, \textbf{r}\left(a_{[2](2)},h_{(2)}\right)\\
&\hspace*{1cm}\textbf{r}\left(a_{[2](3)},S_{H}(a_{[1](1)}a_{[2](4)})\right)\, \textbf{r}\left(a_{[3](2)},h_{(1)}\right) \textbf{r}\left(a_{[3](3)}, S_{H}(a_{[1](2)}a_{[2](5)}a_{[3](4)})\right)
\end{align*}
where we have used the fact that $S_A$ is an algebra anti-morphism in $\MM^H$ for the last equality. Using successively the properties of $\textbf{r}$, we get
\begin{align*}
&F \circ G (a \otimes b \otimes h) \\
&= a_{[1](0)} \otimes bS_{A}\left(a_{[2](0)} S_{A}(a_{[3](0)})\right)\otimes h_{(2)} \, \textbf{r}^{-1}\left(a_{[2](1)}, a_{[3](1)}\right) \, \textbf{r}\left(a_{[3](2)}a_{[2](2)},h_{(1)}\right)\\
&\hspace*{1cm}\textbf{r}\left(a_{[2](3)},S_{H}(a_{[1](1)}a_{[2](4)})\right)\, \textbf{r}\left(a_{[3](3)}, S_{H}(a_{[1](2)}a_{[2](5)}a_{[3](4)})\right)\\
&=a_{[1](0)} \otimes bS_{A}\left(a_{[2](0)} S_{A}(a_{[3](0)})\right)\otimes h_{(2)} \textbf{r}^{-1}\left(a_{[2](2)}, a_{[3](2)}\right)\textbf{r}\left(a_{[2](1)}a_{[3](1)},h_{(1)}\right)\\
&\hspace*{1cm}\textbf{r}\left(a_{[2](3)},S_{H}(a_{[1](1)}a_{[2](4)})\right)\, \textbf{r}\left(a_{[3](3)}, S_{H}(a_{[1](2)}a_{[2](5)})\right) \textbf{r}\left(a_{[3](4)}, S_{H}(a_{[3](5)})\right)\\
&= a_{[1](0)} \otimes bS_{A}\left(a_{[2](0)} S_{A}(a_{[3](0)})\right)\otimes h_{(2)} \, \textbf{r}^{-1}\left(a_{[2](2)}, a_{[3](2)}\right) \, \textbf{r}\left(a_{[2](1)}a_{[3](1)},h_{(1)}\right)\\
&\hspace*{1cm}\textbf{r}\left(a_{[3](3)}a_{[2](3)}, S_{H}(a_{[1](1)}a_{[2](4)})\right) \textbf{r}\left(a_{[3](4)}, S_{H}(a_{[3](5)})\right)\\
&= a_{[1](0)} \otimes bS_{A}\left(a_{[2](0)} S_{A}(a_{[3](0)})\right)\otimes h_{(2)} \, \textbf{r}^{-1}\left(a_{[2](3)}, a_{[3](3)}\right) \, \textbf{r}\left(a_{[2](1)}a_{[3](1)},h_{(1)}\right)\\
&\hspace*{1cm}\textbf{r}\left(a_{[2](2)}a_{[3](2)}, S_{H}(a_{[1](1)}a_{[2](4)})\right) \textbf{r}\left(a_{[3](4)}, S_{H}(a_{[3](5)})\right).
\end{align*}
Now let $f : H\to k$ be the linear map defined by $f(x) = \textbf{r}\big(x_{(1)}, S_{H}(x_{(2)})\big)$ and recall from \cite[Proposition 2.v, p.334]{KSBook97} that, for $h, k\in H$,
$\textbf{r}^{-1}(h, k) = \textbf{r}(S_{H}(h), k)$ and  $\textbf{r}(h, k) = \textbf{r}(S_{H}(h), S_{H}(k))$. We have the partial expression
\begin{align*}
\textbf{r}^{-1}\left(a_{[2](3)}, a_{[3](3)}\right) &\textbf{r}\left(a_{[3](4)}, S_{H}(a_{[3](5)})\right) = \textbf{r}\left(S_{H}(a_{[2](3)}), a_{[3](3)}\right) f\big(a_{[3](4)}\big)\\
&=\textbf{r}\left(S_{H}(a_{[2](3)}), a_{[3](3)} f\big(a_{[3](4)}\big)\right)\\
&= \textbf{r}\left(S_{H}(a_{[2](3)}), S^{2}_{H}(a_{[3](4)} )f\big(a_{[3](3)}\big)\right) \, (\text{by \cite[Proposition 3, p.334]{KSBook97}})\\
&= \textbf{r}\left(a_{[2](3)}, S_{H}(a_{[3](4)})\right) f\big(a_{[3](3)}\big)\\
&= \textbf{r}\left(a_{[2](3)}, S_{H}(a_{[3](5)})\right) \textbf{r}\big(a_{[3](3)}, S_{H}(a_{[3](4)})\big)\\
&= \textbf{r}\left(a_{[2](3)}a_{[3](3)} , S_{H}(a_{[3](4)})\right).
\end{align*}
We finally obtain
\begin{align*}
&F \circ G (a \otimes b \otimes h) = a_{[1](0)} \otimes bS_{A}\left(a_{[2](0)} S_{A}(a_{[3](0)})\right)\otimes h_{(2)} \, \textbf{r}\left(a_{[2](1)}a_{[3](1)},h_{(1)}\right) \\
&\hspace*{1cm}\textbf{r}\left(a_{[2](2)}a_{[3](2)}, S_{H}(a_{[1](1)}a_{[2](4)})\right) \textbf{r}\left(a_{[2](3)}a_{[3](3)} , S_{H}(a_{[3](4)})\right)\\
&= a_{[1](0)} \otimes bS_{A}\left(a_{[2](0)} S_{A}(a_{[3](0)})\right)\otimes h_{(2)} \, \textbf{r}\left(a_{[2](1)}a_{[3](1)},S_{H}(a_{[1](1)}a_{[2](2)}a_{[3](2)})h_{(1)}\right) \\
&= a_{[1](0)} \otimes bS_{A}\left[\left(a_{[2]} S_{A}(a_{[3]})\right)_{(0)}\right]\otimes h_{(2)}\\
& \hspace*{1cm}  \textbf{r}\left[\left(a_{[2]}S_{A}(a_{[3]})\right)_{(1)}, S_{H}\left((a_{[2]}S_{A}(a_{[3]})_{(2)}\right) S_{H}(a_{[1](1)})h_{(1)}\right] \hspace*{1cm} (\text{by \eqref{morantipode}})\\
&= a_{[1](0)} \otimes bS_{A}\left(\varepsilon_{A}(a_{[2]})\right) \otimes h_{(2)} \, \textbf{r}\left[1, S_{H}(a_{[1](1)})h_{(1)}\right]\\
&= a \otimes b \otimes h.
\end{align*}
Similarly, we also have 
\begin{align*}
&G \circ F (a \otimes b \otimes h) = G\left(a_{[1](0)} \otimes bS_{A}(a_{[2](0)}) \otimes h_{(2)}\right) \, \textbf{r}\left(a_{[2](1)}, S_{H}(a_{[1](1)}a_{[2](2)})h_{(1)}\right)\\
&= a_{[1](0)[1](0)} \otimes bS_{A}(a_{[2](0)})S^{2}_{A}(a_{[1](0)[2](0)}) \otimes h_{(3)} \, \textbf{r}\left(a_{[2](1)}, S_{H}(a_{[1](1)}a_{[2](2)})h_{(1)}\right)\\
&\hspace*{1cm} \textbf{r}\left(a_{[1](0)[2](1)}, S_{H}(a_{[1](0)[1](1)}a_{[1](0)[2](2)})h_{(2)}\right)\\
&= a_{[1](0)} \otimes bS_{A}(a_{[3](0)})S^{2}_{A}(a_{[2](0)}) \otimes h_{(3)} \, \textbf{r}\left(a_{[2](1)}, S_{H}(a_{[1](1)}a_{[2](2)})h_{(2)}\right)\\
&\hspace*{1cm} \textbf{r}\left(a_{[3](1)}, S_{H}(a_{[1](2)}a_{[2](3)}a_{[3](2)})h_{(1)}\right) \hspace*{2cm}(\text{by using $\eqref{morcoproduct1}$ successively})\\
&= a_{[1](0)} \otimes bS_{A}\left[S_{A}(a_{[2](0)})a_{[3](0)}\right] \otimes h_{(3)} \, \textbf{r}^{-1}(a_{[2](1)}, a_{[3](1)})\, \textbf{r}\left(a_{[2](2)}, S_{H}(a_{[1](1)}a_{[2](3)})h_{(2)}\right)\\
&\hspace*{1cm} \textbf{r}\left(a_{[3](2)}, S_{H}(a_{[1](2)}a_{[2](4)}a_{[3](3)})h_{(1)}\right) \hspace*{2cm}(\text{ $S_{A}$ is an antimorphism in $\MM^{H}$})\\
&= a \otimes b \otimes h
\end{align*}
by applying the same reasoning as in the calculations of $F \circ G$, which completes the proof. 
\epf

The proof of Theorem \ref{Thm:Hoch} will consist in transporting the right $A^e$-module structure on $H^{*}(A, \,_{A}A \otimes A_{A})$ to $\ext^{*}_{A}(_{\varepsilon}k, \,_{A}A) \otimes A$ using several successive isomorphisms:
\begin{align*}
H^{*}(A, \,_{A}A \otimes A_{A}) &= \ext^{*}_{\,_{A}\MM_{A}}(A, \,_{A}A \otimes A_{A})&\\ &\simeq \ext^{*}_{\,_{A}\MM_{A}^{H}}(A, \,_{A}A \otimes A_{A} \odot H)  &\text{(by Proposition \ref{prop:pdHopfmod})}\\&\simeq \ext^{*}_{\,_{A}\MM^{H}}(_{\varepsilon}k, \stackon[-8pt]{$\,_{A}A \otimes A_{A} \odot H$}{\vstretch{1.5}{\hstretch{3.1}{\widetilde{\phantom{\;\;\;\;\;\;\;\;}}}}})  &\text{(by Remark \ref{rem:ext})}\\
&\simeq \ext^{*}_{\,_{A}\MM^{H}}(_{\varepsilon}k, \,_{A}A \otimes A \boxdot H) \quad &\text{(by Theorem \ref{Thm:Iso})}\\
&\simeq \ext^{*}_{A}(_{\varepsilon}k, \,_{A}A \otimes A)  &\text{(by Proposition \ref{prop:pdHopfmod})}\\
&\simeq \ext^{*}_{A}(_{\varepsilon}k, \,_{A}A) \otimes A &(\text{$_{\varepsilon}k$ is of type FP}). 
\end{align*}

We now proceed with the transportation of the various $A^e$-module structures, step by step. We fix an object  $P$ be in $_A\mathcal M^H$.

\bl \label{Lem: HomAMAH}
There is a right $A^e$-module structure on $\Hom_{\,_A\MM^{H}_A}(P \boxtimes A, \,_AA \otimes A_A \odot H)$ defined by
\begin{align}
f \cdot (a' \otimes b') (x \otimes a) = \sum_{i}a^{i}a' \otimes b'b^{i} \otimes h^{i}
\end{align}
where $f(x \otimes a) = \sum_{i}a^{i} \otimes b^{i} \otimes h^{i}$, and such that the natural isomorphism
\[\Hom_{\,_A\MM^{H}_A}(P \boxtimes A, \,_AA \otimes A_A \odot H) \simeq \Hom_{\,_A\MM_A}(P \boxtimes A, \,_AA \otimes A_A)\]
is $A^e$-linear.
\el
\bpf
The isomorphism is 
\begin{align*}
\Hom_{\,_A\MM^{H}_A}(P \boxtimes A, \,_AA \otimes A_A \odot H) &\longrightarrow \Hom_{\,_A\MM_A}(P \boxtimes A, \,_AA \otimes A_A)\\
f \,&\longmapsto \quad \tilde{f} = (\id \otimes \varepsilon_H) \circ f
\end{align*}
and the verification is immediate.
\epf

\bl\label{Lem: HomAMH1}
There is a right $A^e$-module structure on $\Hom_{\,_A\MM^{H}}(P, \stackon[-8pt]{$\,_{A}A \otimes A_{A} \odot H$}{\vstretch{1.5}{\hstretch{3.1}{\widetilde{\phantom{\;\;\;\;\;\;\;\;}}}}})$, defined by 
\begin{align*}
f \cdot (a' \otimes b')(x) = \sum_i a^{i}a' \otimes b'b^{i} \otimes h^{i} 
\end{align*}
where $f (x) = \sum_i a^{i} \otimes b^{i} \otimes h^{i}$, and such that the natural isomorphism
\[\Hom_{\,_A\MM^{H}}(P, \stackon[-8pt]{$\,_{A}A \otimes A_{A} \odot H$}{\vstretch{1.5}{\hstretch{3.1}{\widetilde{\phantom{\;\;\;\;\;\;\;\;}}}}}) \simeq \Hom_{\,_A\MM_A^H}(P \boxtimes A, \,_AA \otimes A_A \odot H)\]
is $A^e$-linear.
\el

\bpf
The isomorphism (from Proposition \ref{prop:adjointmodulebimodule})  is given by
\begin{align*}
\Hom_{\,_A\MM^{H}}(P, \stackon[-8pt]{$\,_{A}A \otimes A_{A} \odot H$}{\vstretch{1.5}{\hstretch{3.1}{\widetilde{\phantom{\;\;\;\;\;\;\;\;}}}}}) &\longrightarrow \Hom_{\,_A\MM_A^H}(P \boxtimes A, \,_AA \otimes A_A \odot H)\\
f \quad &\longmapsto \quad \tilde{f},  \ x \otimes a \mapsto f(x).a
\end{align*}
and again the verification is immediate.
\epf

\bl\label{Lem:HomAMH2} There is a right $A^e$-module structure on
$\Hom_{\,_A\MM^{H}}(P, \,_AA \otimes A \boxdot H)$ defined by
\begin{align*}
f \cdot (a' \otimes b')(x) = \sum_{i}a^{i}_{(0)}a'_{(0)[1]} \otimes b'b^{i}S^{2}_{A}\big(a'_{(0)[2](0)}\big) \otimes h^{i}_{(2)}\, \normalfont \textbf{r}\left[a'_{(0)[2](1)}, S_{H}\big(a^{i}_{(1)}a'_{(1)}\big)h^{i}_{(1)}\right]
\end{align*}
where $f(x) = \sum_i a^{i} \otimes b^{i} \otimes h^{i}$, and such that the isomorphism induced by the isomorphism of Proposition \ref{Thm:Iso}
\[\Hom_{\,_A\MM^{H}}(P, \,_AA \otimes A \boxdot H) \simeq \Hom_{\,_A\MM^H}(P, \, \stackon[-8pt]{$\,_{A}A \otimes A_{A} \odot H$}{\vstretch{1.5}{\hstretch{3.1}{\widetilde{\phantom{\;\;\;\;\;\;\;\;}}}}})\]
is $A^e$-linear.
\el
\bpf
The above isomorphism coming from Proposition  \ref{Thm:Iso} is
\begin{align*}
\Hom_{\,_A\MM^{H}}(P, \,_AA \otimes A \boxdot H) &\longrightarrow \Hom_{\,_A\MM^H}(P, \, \stackon[-8pt]{$\,_{A}A \otimes A_{A} \odot H$}{\vstretch{1.5}{\hstretch{3.1}{\widetilde{\phantom{\;\;\;\;\;\;\;\;}}}}})\\
f \hspace*{.5cm}&\longmapsto \quad F \circ f\\
G \circ g \hspace*{.5cm}&\longmapsfrom \quad g
\end{align*}
The transported $A^e$-module structure on $\Hom_{\,_A\MM^{H}}(P, \,_AA \otimes A \boxdot H)$ is defined by 
\[f \cdot (a' \otimes b')(x) = G\big((F\circ f) \cdot (a' \otimes b') (x)\big)\]
For $f(x) = \sum_i a^{i} \otimes b^{i} \otimes h^{i}$, we have
\[F(f(x)) = \sum_{i}a^{i}_{[1](0)} \otimes b^iS_{A}(a^{i}_{[2](0)}) \otimes h^{i}_{(2)} \, \normalfont\textbf{r}\left(a^{i}_{[2](1)}, S_{H}(a^{i}_{[1](1)}a^{i}_{[2](2)})h^{i}_{(1)}\right) \]
and hence
\[(F\circ f) \cdot (a' \otimes b')(x) = \sum_{i}a^{i}_{[1](0)}a' \otimes b'b^iS_{A}(a^{i}_{[2](0)}) \otimes h^{i}_{(2)} \, \normalfont\textbf{r}\left(a^{i}_{[2](1)}, S_{H}(a^{i}_{[1](1)}a^{i}_{[2](2)})h^{i}_{(1)}\right)\]
We thus have
\begin{align*}
&f \cdot (a' \otimes b')(x) =  G\big((F\circ f) \cdot (a' \otimes b')(x)\big)\\
&= \sum_{i}(a^{i}_{[1](0)}a')_{[1](0)} \otimes b'b^{i}S_{A}(a^{i}_{[2](0)})S^{2}_{A}\big[(a^{i}_{[1](0)}a')_{[2](0)}\big] \otimes h^{i}_{(3)}\\
&\textbf{r}\left[a^{i}_{[2](1)}, S_{H}(a^{i}_{[1](1)}a^{i}_{[2](2)})h^{i}_{(1)}\right]\textbf{r}\left[(a^{i}_{[1](0)}a')_{[2](1)}, S_{H}((a^{i}_{[1](0)}a')_{[1](1)}(a^{i}_{[1](0)}a')_{[2](2)})h^{i}_{(2)}\right]
\end{align*}
Since $\Delta_{A}$ is an algebra morphism in $\MM^{H}$, we have the partial expression
\[(a^{i}_{[1](0)}a')_{[1]} \otimes (a^{i}_{[1](0)}a')_{[2]} = a^{i}_{[1](0){[1]}}a'_{[1](0)} \otimes a^{i}_{[1](0)[2](0)}a'_{[2]}\, \textbf{r}\left[a^{i}_{[1](0)[2](1)}, a'_{[1](1)}\right]\]
and hence
\begin{align*}
&f \cdot (a' \otimes b')(x) = \sum_{i}(a^{i}_{[1](0){[1]}}a'_{[1](0)})_{(0)} \otimes b'b^{i}S_{A}(a^{i}_{[2](0)})S^{2}_{A}\big[(a^{i}_{[1](0)[2](0)}a'_{[2]})_{(0)}\big] \otimes h^{i}_{(3)}\\ &\hspace*{1cm} \textbf{r}\left[a^{i}_{[1](0)[2](1)}, a'_{[1](1)}\right]\textbf{r}\left[a^{i}_{[2](1)}, S_{H}\big(a^{i}_{[1](1)}a^{i}_{[2](2)}\big)h^{i}_{(1)}\right]\\
&\hspace*{1cm} \textbf{r}\left[(a^{i}_{[1](0)[2](0)}a'_{[2]})_{(1)}, S_{H}\big((a^{i}_{[1](0){[1]}}a'_{[1](0)})_{(1)}(a^{i}_{[1](0)[2](0)}a'_{[2]})_{(2)}\big)h^{i}_{(2)}\right]\\
&= \sum_{i}a^{i}_{[1](0)[1](0)}a'_{[1](0)} \otimes b'b^{i}S_{A}(a^{i}_{[2](0)})S^{2}_{A}\big(a^{i}_{[1](0)[2](0)}a'_{[2](0)}\big) \otimes h^{i}_{(3)}\\ &\hspace*{1cm} \textbf{r}\left[a^{i}_{[1](0)[2](3)}, a'_{[1](2)}\right]\textbf{r}\left[a^{i}_{[2](1)}, S_{H}(a^{i}_{[1](1)}a^{i}_{[2](2)})h^{i}_{(1)}\right]\\
&\hspace*{1cm} \textbf{r}\left[a^{i}_{[1](0)[2](1)}a'_{[2](1)}, S_{H}\big(a^{i}_{[1](0)[1](1)}a'_{[1](1)}a^{i}_{[1](0)[2](2)}a'_{[2](2)}\big)h^{i}_{(2)}\right].
\end{align*}
Using now the properties of $\textbf{r}$, we get
\begin{align*}
f \cdot (a' \otimes b')(x) &= \sum_{i}a^{i}_{[1](0)[1](0)}a'_{[1](0)} \otimes b'b^{i}S_{A}(a^{i}_{[2](0)})S^{2}_{A}\big(a^{i}_{[1](0)[2](0)}a'_{[2](0)}\big )\otimes h^{i}_{(3)}\\ &\hspace*{1cm} \textbf{r}\left[a^{i}_{[1](0)[2](2)}, a'_{[1](1)}\right]\textbf{r}\left[a^{i}_{[2](1)}, S_{H}(a^{i}_{[1](1)}a^{i}_{[2](2)})h^{i}_{(1)}\right]\\
&\hspace*{1cm} \textbf{r}\left[a^{i}_{[1](0)[2](1)}a'_{[2](1)}, S_{H}\big(a^{i}_{[1](0)[1](1)}a^{i}_{[1](0)[2](3)}a'_{[1](2)}a'_{[2](2)}\big)h^{i}_{(2)}\right].
\end{align*}
Applying the $H$-colinearity of $\Delta_A$ $\eqref{morcoproduct1}$ successively for $a_{[1](0)}$ and  $a_{[1]}$ then gives
\begin{align*}
f \cdot (a' \otimes b')(x)&= \sum_{i}a^{i}_{[1](0)}a'_{[1](0)} \otimes b'b^{i}S_{A}(a^{i}_{[3](0)})S^{2}_{A}\big(a^{i}_{[2](0)}a'_{[2](0)}\big) \otimes h^{i}_{(3)}\\ &\hspace*{1cm} \textbf{r}\left[a^{i}_{[2](2)}, a'_{[1](1)}\right]\textbf{r}\left[a^{i}_{[3](1)}, S_{H}\big(a^{i}_{[1](2)}a^{i}_{[2](4)}a^{i}_{[3](2)}\big)h^{i}_{(1)}\right]\\
&\hspace*{1cm} \textbf{r}\left[a^{i}_{[2](1)}a'_{[2](1)}, S_{H}\big(a^{i}_{[1](1)}a^{i}_{[2](3)}a'_{[1](2)}a'_{[2](2)}\big)h^{i}_{(2)}\right].
\end{align*}
Using the fact that $S_A$ is an anti-morphism and then successive applications of the properties of $\textbf{r}$, we get
\begin{align*}
f \cdot &(a' \otimes b')(x)= \sum_{i}a^{i}_{[1](0)}a'_{[1](0)} \otimes b'b^{i}S_{A}\left[S_{A}\big(a^{i}_{[2](0)}a'_{[2](0)}\big)a^{i}_{[3](0)}\right] \otimes h^{i}_{(3)} \\ &\textbf{r}^{-1}\left[a^{i}_{[2](1)}a'_{[2](1)}, a^{i}_{[3](1)}\right] \textbf{r}\left[a^{i}_{[2](3)}, a'_{[1](1)}\right]\textbf{r}\left[a^{i}_{[3](2)}, S_{H}\big(a^{i}_{[1](2)}a^{i}_{[2](5)}a^{i}_{[3](3)}\big)h^{i}_{(1)}\right]\\
&\textbf{r}\left[a^{i}_{[2](2)}a'_{[2](2)}, S_{H}\big(a^{i}_{[1](1)}a^{i}_{[2](4)}a'_{[1](2)}a'_{[2](3)}\big)h^{i}_{(2)}\right]\\
&= \sum_{i}a^{i}_{[1](0)}a'_{[1](0)} \otimes b'b^{i}S_{A}\left[S_{A}\big(a^{i}_{[2](0)}a'_{[2](0)}\big)a^{i}_{[3](0)}\right] \otimes h^{i}_{(3)} \, \textbf{r}^{-1}\left[a^{i}_{[2](1)}a'_{[2](1)}, a^{i}_{[3](1)}\right]\\ &\hspace*{1cm} \textbf{r}\left[a^{i}_{[2](4)}, a'_{[1](1)}\right]\textbf{r}\left[a^{i}_{[3](2)}, h^{i}_{(1)}\right]\textbf{r}\left[a^{i}_{[3](3)}, S_{H}\big(a^{i}_{[1](2)}a^{i}_{[2](6)}a^{i}_{[3](4)}\big)\right]\\
&\hspace*{1cm} \textbf{r}\left[a^{i}_{[2](2)}a'_{[2](2)}, h^{i}_{(2)}\right]\textbf{r}\left[a^{i}_{[2](3)}a'_{[2](3)}, S_{H}\big(a^{i}_{[1](1)}a^{i}_{[2](5)}a'_{[1](2)}a'_{[2](4)}\big)\right]\\
&= \sum_{i}a^{i}_{[1](0)}a'_{[1](0)} \otimes b'b^{i}S_{A}\left[S_{A}\big(a^{i}_{[2](0)}a'_{[2](0)}\big)a^{i}_{[3](0)}\right] \otimes h^{i}_{(2)} \, \textbf{r}^{-1}\left[a^{i}_{[2](3)}a'_{[2](3)}, a^{i}_{[3](3)}\right]\\ 
&\hspace*{1cm} \textbf{r}\left[a^{i}_{[2](5)}, a'_{[1](1)}\right]\textbf{r}\left[a^{i}_{[2](2)}a'_{[2](2)}a^{i}_{[3](2)}, S_{H}\big(a^{i}_{[1](1)}a^{i}_{[2](6)}\big)\right] \textbf{r}\left[ a^{i}_{[3](4)}, S_{H}(a^{i}_{[3](5)})\right]\\
&\hspace*{1cm}\textbf{r} \big[ a^{i}_{[2](4)}a'_{[2](4)}, S_{H}\big(a'_{[1](2)}a'_{[2](5)}\big)\big] \textbf{r}\left[a^{i}_{[2](1)}a'_{[2](1)}a^{i}_{[3](1)}, h^{i}_{(1)}\right].\hspace*{1cm}  
\end{align*}
Considering again the linear form $f : H\to k$ defined by $f(x)= \textbf{r}(x_{(1)},S(x_{(2)})$, we obtain, 
by \cite[Proposition 2.v and Proposition 3, p.334]{KSBook97}, the partial expression
\begin{align*}
\textbf{r}^{-1}\left[a^{i}_{[2](3)}a'_{[2](3)}, a^{i}_{[3](3)}\right]&\textbf{r}\left[a^{i}_{[3](4)}, S_{H}(a^{i}_{[3](5)})\right] = \textbf{r}\left[S_{H}\big(a^{i}_{[2](3)}a'_{[2](3)}\big), a^{i}_{[3](3)}f(a^{i}_{[3](4)})\right]\nonumber\\
&= \textbf{r}\left[S_{H}\big(a^{i}_{[2](3)}a'_{[2](3)}\big), S^{2}_{H}(a^{i}_{[3](4)})f(a^{i}_{[3](3)})\right]\nonumber\\
&= \textbf{r}\left[a^{i}_{[2](3)}a'_{[2](3)}, S_{H}(a^{i}_{[3](4)})f(a^{i}_{[3](3)})\right] \nonumber\\
&= \textbf{r}\left[a^{i}_{[2](3)}a'_{[2](3)}, S_{H}(a^{i}_{[3](5)})\right]\textbf{r}\left[a^{i}_{[3](3)}, S_{H}(a^{i}_{[3](4)})\right] \nonumber\\
&= \textbf{r}\left[a^{i}_{[2](3)}a'_{[2](3)}a^{i}_{[3](3)}, S_{H}(a^{i}_{[3](4)})\right].
\end{align*}
Hence we have
\begin{align*}
&f \cdot (a' \otimes b')(x) = \sum_{i}a^{i}_{[1](0)}a'_{[1](0)} \otimes b'b^{i}S_{A}\left[S_{A}\big(a^{i}_{[2](0)}a'_{[2](0)}\big)a^{i}_{[3](0)}\right] \otimes h^{i}_{(2)}\\ 
&\hspace*{1cm} \textbf{r}\left[a^{i}_{[2](2)}a'_{[2](2)}a^{i}_{[3](2)}, S_{H}\big(a^{i}_{[1](1)}a^{i}_{[2](6)}\big)\right]\textbf{r}\left[a^{i}_{[2](3)}a'_{[2](3)}a^{i}_{[3](3)}, S_{H}(a^{i}_{[3](4)})\right]\\
&\hspace*{1cm}\textbf{r} \big[ a^{i}_{[2](4)}a'_{[2](4)}, S_{H}\big(a'_{[1](2)}a'_{[2](5)}\big)\big] \textbf{r}\left[a^{i}_{[2](1)}a'_{[2](1)}a^{i}_{[3](1)}, h^{i}_{(1)}\right]\textbf{r}\left[a^{i}_{[2](5)}, a'_{[1](1)}\right]\\
&= \sum_{i}a^{i}_{[1](0)}a'_{[1](0)} \otimes b'b^{i}S_{A}\left[S_{A}\big(a^{i}_{[2](0)}a'_{[2](0)}\big)a^{i}_{[3](0)}\right] \otimes h^{i}_{(2)}\\ 
&\hspace*{1cm} \textbf{r}\left[a^{i}_{[2](1)}a'_{[2](1)}a^{i}_{[3](1)}, S_{H}\big(a^{i}_{[1](1)}a^{i}_{[2](4)}a^{i}_{[3](2)}\big)h^{i}_{(1)}\right]\\
&\hspace*{1cm}\textbf{r} \big[ a^{i}_{[2](2)}a'_{[2](2)}, S_{H}\big(a'_{[1](2)}a'_{[2](3)}\big)\big] \textbf{r}\left[a^{i}_{[2](3)}, a'_{[1](1)}\right]
\end{align*}
\begin{align*}
&= \sum_{i}a^{i}_{[1](0)}a'_{[1](0)} \otimes b'b^{i}S_{A}\left[S_{A}\big(a'_{[2](0)}\big)S_{A}\big(a^{i}_{[2](0)}\big)a^{i}_{[3](0)}\right] \otimes h^{i}_{(2)} \textbf{r}\big[a^{i}_{[2](1)}, a'_{[2](1)}\big]\\ 
&\hspace*{1cm} \textbf{r}\left[a^{i}_{[2](2)}a'_{[2](2)}a^{i}_{[3](1)}, S_{H}\big(a^{i}_{[1](1)}a^{i}_{[2](5)}a^{i}_{[3](2)}\big)h^{i}_{(1)}\right]\\
&\hspace*{1cm}\textbf{r} \big[ a^{i}_{[2](3)}a'_{[2](3)}, S_{H}\big(a'_{[1](2)}a'_{[2](4)}\big)\big] \textbf{r}\left[a^{i}_{[2](4)}, a'_{[1](1)}\right].
\end{align*}
Using once again the properties of $\textbf{r}$, we then have
\begin{align*}
f \cdot &(a' \otimes b')(x) = \sum_{i}a^{i}_{[1](0)}a'_{[1](0)} \otimes b'b^{i}S_{A}\left[S_{A}\big(a'_{[2](0)}\big)S_{A}\big(a^{i}_{[2](0)}\big)a^{i}_{[3](0)}\right] \otimes h^{i}_{(2)}\\ 
&\textbf{r}\big[a^{i}_{[2](2)}, a'_{[2](2)}\big] \textbf{r}\left[a'_{[2](1)}a^{i}_{[2](1)}a^{i}_{[3](1)}, S_{H}\big(a^{i}_{[1](1)}a^{i}_{[2](5)}a^{i}_{[3](2)}\big)h^{i}_{(1)}\right] \, (\text{permute $a^{i}$ and $a'$})\\
&\textbf{r} \big[ a^{i}_{[2](3)}a'_{[2](3)}, S_{H}\big(a'_{[1](2)}a'_{[2](4)}\big)\big] \textbf{r}\left[a^{i}_{[2](4)}, a'_{[1](1)}\right].
\end{align*}
The partial expression
\begin{align*}
&\textbf{r}\big[a^{i}_{[2](2)}, a'_{[2](2)}\big]\textbf{r} \big[ a^{i}_{[2](3)}a'_{[2](3)}, S_{H}\big(a'_{[1](2)}a'_{[2](4)}\big)\big] \textbf{r}\left[a^{i}_{[2](4)}, a'_{[1](1)}\right]\\
&= \textbf{r}\big[a^{i}_{[2](3)}, a'_{[2](3)}\big]\textbf{r} \big[a'_{[2](2)}a^{i}_{[2](2)}, S_{H}\big(a'_{[1](2)}a'_{[2](4)}\big)\big] \textbf{r}\left[a^{i}_{[2](4)}, a'_{[1](1)}\right]\\
&= \textbf{r}\big[a^{i}_{[2](3)}, a'_{[1](1)}a'_{[2](3)}\big]\textbf{r} \big[a^{i}_{[2](2)}, S_{H}\big(a'_{[1](2)}a'_{[2](4)}\big)\big]\textbf{r} \big[a'_{[2](2)}, S_{H}\big(a'_{[1](3)}a'_{[2](5)}\big)\big]\\
&= \varepsilon_{H}\big(a^{i}_{[2](2)}\big) \textbf{r} \big[a'_{[2](2)}, S_{H}\big(a'_{[1](1)}a'_{[2](3)}\big)\big]
\end{align*}
enables us to obtain, using again the colinearity of $S_A$ and \eqref{morcoproduct1},
\begin{align*}
&f \cdot (a' \otimes b')(x) = \sum_{i}a^{i}_{[1](0)}a'_{[1](0)} \otimes b'b^{i}S_{A}\left[S_{A}\big(a'_{[2](0)}\big)\big(S_{A}(a^{i}_{[2]})a^{i}_{[3]}\big)_{(0)}\right] \otimes h^{i}_{(2)}\\ &\hspace*{0.5cm}\textbf{r}\left[a'_{[2](1)}\big(S_{A}(a^{i}_{[2]})a^{i}_{[3]}\big)_{(1)}, \big(S_{A}(a^{i}_{[2]})a^{i}_{[3]}\big)_{(2)}S_{H}\big(a^{i}_{[1](1)}\big)h^{i}_{(1)}\right]\, \textbf{r}\left[a'_{[2](2)}, S_{H}(a'_{[1](1)}a'_{[2](3)})\right]\\
&\hspace*{1cm}(\text{Since $S_A$ is $H$-colinear})\\
&=  \sum_{i}a^{i}_{[1](0)}a'_{[1](0)} \otimes b'b^{i}S^{2}_{A}\big(a'_{[2](0)}\big) \varepsilon_A(a^{i}_{[2]}) \otimes h^{i}_{(2)}\,\textbf{r}\left[a'_{[2](1)}, S_{H}\big(a^{i}_{[1](1)}a'_{[1](1)}a'_{[2](2)}\big)h^{i}_{(1)}\right]\\
&= \sum_{i}a^{i}_{(0)}a'_{(0)[1]} \otimes b'b^{i}S^{2}_{A}\big(a'_{(0)[2](0)}\big) \otimes h^{i}_{(2)}\,\textbf{r}\left[a'_{(0)[2](1)}, S_{H}\big(a^{i}_{(1)}a'_{(1)}\big)h^{i}_{(1)}\right]\quad 
\end{align*}
and this completes the proof.
\epf

\bl \label{Lem:AM} Assume that $P$ is finitely generated and projective as an $A$-module. Then there is a right $A^e$-module structure on $\Hom_{\,_A\MM}(P, A) \otimes A$ such that 
\begin{align*}
(f \otimes a') \cdot (a \otimes b) = (f \cdot a_{[1]})_{(0)} \otimes ba'S^{2}_{A}(a_{[2](0)}) \, \normalfont\textbf{r}\big[a_{[2](1)}, S_{H}\big(a_{[2](2)}\big)S_{H}\big((f \cdot a_{[1]})_{(1)} \big)\big] 
\end{align*}
where the right $A$-action and $H$-coaction on $\Hom_{\,_A\MM}(P, A)$ are induced respectively by right multiplication in $A$ and by Lemma \ref{Lem:Cohom}. This $A^e$-module structure is such that
the composition of the canonical isomorphisms
\[
 \Hom_{\,_A\MM}(P, A) \otimes A  \to \Hom_{\,_A\MM}(P, \,_AA \otimes A) \to \Hom_{\,_A\MM^{H}}(P, \,_AA \otimes A \boxdot H)
\]
is $A^e$-linear, where the right term has the $A^e$-module structure given in Lemma \ref{Lem:HomAMH2}.
\el

\bpf The map on the left above is an isomorphism because $P$ is finitely generated projective, the map on the right is the adjunction isomorphism, and the composition, that we denote $T$, is then defined by
\[T(f\otimes a)(x) = f(x_{(0)}) \otimes a \otimes x_{(1)}.\]
It is enough to check that $T((f \otimes a') \cdot (a \otimes b)) = T(f \otimes a') \cdot (a \otimes b)$, with $(f \otimes a') \cdot (a \otimes b)$ defined as above. We have
\begin{align*}
&T\left((f \otimes a') \cdot (a \otimes b)\right)(x)\\
&= T \left[(f \cdot a_{[1]})_{(0)} \otimes ba'S^{2}_{A}(a_{[2](0)}) \right] (x) \, \textbf{r}\big[a_{[2](1)}, S_{H}\big(a_{[2](2)}\big)S_{H}\big((f \cdot a_{[1]})_{(1)} \big)\big]                                                                                                                                                                                                                                                                                                                                                                                                                                                                                                                                                                                                                                                                                                                                                                                                                                                                                                                                                                                                                                                                                                                                                                                                                                                                                                                                                                                                                                                                                                                                                                                                                                                                                                                                                                              \\
&= (f \cdot a_{[1]})_{(0)} (x_{(0)}) \otimes ba'S^{2}_{A}(a_{[2](0)})\otimes x_{(1)} \, \textbf{r}\big[a_{[2](1)}, S_{H}\big(a_{[2](2)}\big)S_{H}\big((f \cdot a_{[1]})_{(1)} \big)\big]\\
&= (f \cdot a_{[1]})(x_{(0)})_{(0)} \otimes ba'S^{2}_{A}(a_{[2](0)})\otimes x_{(2)}\\
&\quad \textbf{r}\big[a_{[2](1)}, S_{H}\big(a_{[2](2)}\big)S_{H}\big(S^{-1}_{H}(x_{(1)})(f \cdot a_{[1]})(x_{(0)})_{(1)}\big)\big] \quad (\text{here, we use Lemma \ref{Lem:Cohom}})\\
&= \big(f(x_{(0)})a_{[1]}\big)_{(0)} \otimes ba'S^{2}_{A}(a_{[2](0)}) \otimes x_{(2)}\textbf{r}\big[a_{[2](1)}, S_{H}\big(a_{[2](2)}\big)S_{H}\big((f(x_{(0)}) a_{[1]})_{(1)}\big)x_{(1)}\big]\\
&= f(x_{(0)})_{(0)}a_{[1](0)}\otimes ba'S^{2}_{A}(a_{[2](0)}) \otimes x_{(2)} \textbf{r}\big[a_{[2](1)}, S_{H}\big(a_{[1](1)}a_{[2](2)}\big)S_{H}\big(f(x_{(0)})_{(1)} \big)x_{(1)}\big]. 
\end{align*}
On the other hand we have
\begin{align*}
&T(f \otimes a') \cdot (a \otimes b) (x) \\
&=  f(x_{(0)})_{(0)}a_{(0)[1]} \otimes ba'S_{A}^{2}(a_{(0)[2](0)}) \otimes x_{(2)} \, \textbf{r}
\big[a_{(0)[2](1)}, S_{H}(f(x_{(0)})_{(1)}a_{(1)})x_{(1)}\big]\\
&= f(x_{(0)})_{(0)}a_{[1](0)} \otimes ba'S_{A}^{2}(a_{[2](0)}) \otimes x_{(2)} \textbf{r}
\big[a_{[2](1)}, S_{H}(f(x_{(0)})_{(1)}a_{[1](1)}a_{[2](2)})x_{(1)}\big] \, \text{(by $\eqref{morcoproduct1}$)}
\end{align*}
and the two expressions coincide, which proves our lemma. 
\epf

We can now prove Theorem \ref{Thm:Hoch}.
Let $P_{*} \longrightarrow \,_\varepsilon k$ be a resolution of $_\varepsilon k$ by finite relative projectives  in $\,_{A}\MM^H$. Then the $P_i$ are finitely generated projective $A$-modules,
and $P_* \boxtimes A \longrightarrow A$ is also a resolution of $A$ in $\,_{A}\MM^H_A$ by relative projective objects (by Proposition \ref{prop:free2free}) and hence by finitely generated projective $A$-bimodules. It follows that $\ext^{*}_{\,_{A}\MM_{A}}(A, \,_{A}A \otimes A_{A})$ is the homology of the complex $\Hom_{\,_A\MM_A}(P_{*} \boxtimes A, \,_AA \otimes A_A)$ (with the natural right $A^e$-module structure), while $\ext^{*}_{A}(_{\varepsilon}k, \,_{A}A) \otimes A$ is the homology of the complex 
${\rm Hom}_{_A\mathcal M}(P_*, A)\otimes A$, with the $A^e$-module structure obtained by applications of Lemma \ref{Lem: HomAMAH}, \ref{Lem: HomAMH1}, \ref{Lem:HomAMH2} and \ref{Lem:AM}, and that these homologies are isomorphic as $A^e$-modules.

\subsection{Bosonization and another approach to  Theorem \ref{thm:tcy}} \label{subsec:bozo} 
In this subsection we propose another approach to the proof of Theorem \ref{thm:tcy}, making use of results of Kr\"ahmer \cite{Kra12}. The only small but important nuance with this approach is that we have not been able to get relevant information on the group-like $g\in H$ in Theorem \ref{thm:tcy}, which is important in view of applications. 

To connect Theorem \ref{thm:tcy} and the results in \cite{Kra12}, we need the bosonization construction \cite{Rad85,Maj94}, that we recall now. Let $H$ be a coquasitriangular Hopf algebra and let $A$ be a Hopf algebra in $\mathcal M^H$. We retain the previous Sweedler notation:
\[\Delta_A(a) = a_{[1]} \otimes a_{[2]}, \quad \Delta_H(x) = x_{(1)}\otimes x_{(2)}, \quad  \alpha(a) = a_{(0)} \otimes a_{(1)}.\]
The \textsl{bosonization} $H\#A$ is then the ordinary Hopf algebra that has $H\otimes A$ as underlying vector space, has the unit and counit of the ordinary tensor product of algebras and coalgebras, and comultiplication, product and antipode given by 
\[ x\#a \cdot y\# b = \textbf{r}(a_{(1)},y_{(2)}) xy_{(1)} \# a_{(0)}b, \quad \Delta(x\#a)= (x_{(1)} \# a_{[1](0)}) \otimes (x_{(2)}a_{[1](1)} \# a_{[2]}) \]
\begin{align*} S(x\#a) &= (1\#S_A(a_{(0)}))\cdot (S_H(xa_{(1)})\#1) \\
& = \textbf{r}\left(a_{(1)}, S_H(x_{(1)}a_{(2)})\right) S_H(x_{(2)}a_{(3)})\# S_A(a_{(0)})
\end{align*}
In particular, this gives
\[  S(1\#a) = \textbf{r}\left(a_{(1)}, S_H(a_{(2)})\right) S_H(a_{(3)})\# S_A(a_{(0)})\]
and 
\[ S^2(x\#a) = \textbf{r}\left(a_{(1)}, S_H(a_{(2)})\right) S_H^2(x)\# S_A^2(a_{(0)})\]
The algebra embedding $A \to H\#A$, $a\to 1\#a$, realizes $A$ as a right coideal subalgebra $A\subset H\#A$, since 
\[\Delta(1\#a)= (1 \# a_{[1](0)}) \otimes (a_{[1](1)} \# a_{[2]}).\]
The Hopf algebra $H\#A$ has bijective antipode since $S_A$ is bijective, $H\#A$ is free as a right $A$-module hence is faithfully flat, while the assumption that  $_\varepsilon k$ is of type FP  in $_A\mathcal M^H$ ensures by Theorem \ref{thm:smoothness} that $A$ is smooth, therefore we can use the results in \cite{Kra12}. In particular \cite[Theorem 6]{Kra12} and the reasoning inside the proof of \cite[Theorem 1]{Kra12} (p. 249) give  $H^{i}(A, \,_{A}A \otimes A_{A}) = \{0\}$ for $i \not = n$. It then  remains to study the $A$-bimodule structure of $H^{n}(A, \,_{A}A \otimes A_{A})$.

Let $\psi : A\to k$ be the algebra map corresponding to the $A$-module structure on the one-dimensional space $\ext_A^n(_\varepsilon k, A)$. By \cite[Corollary 2]{Kra12}, the $A$-module $\ext_A^n(_\varepsilon k, A)$ is an object in a certain category of "twisted'' Hopf modules $_A\mathcal M_{A,S^2}^{H \# A}$. Inspecting the condition that defines the objects of  $_A\mathcal M_{A,S^2}^{H \# A}$ \cite[p. 246]{Kra12}, we get  $\psi(a_{(0)}) a_{(1)} = \psi(a)1$ for any $a\in A$ (while there is no condition on the group-like $g \in H$ corresponding to the $H$-comodule structure). Following \cite[Lemma 7]{Kra12}, we define the algebra map
\begin{align*}
 \sigma : A & \longrightarrow H\#A \\
a & \longmapsto \psi(a_{[1](0)}) S^2(a_{[1](1)} \# a_{[2]})
\end{align*}
We have, using the equality $\psi(a_{(0)}) a_{(1)} = \psi(a)1$, 
\begin{align*}\sigma(a) & = \psi(a_{[1](0)})\textbf{r}\left(a_{[2](1)}, S_H(a_{[2](2)})\right) S^2_H(a_{[1](1)})\# S^2_A(a_{[2](0)})\\
& = \psi(a_{[1]})\textbf{r}\left(a_{[2](1)}, S_H(a_{[2](2)})\right) 1\# S^2_A(a_{[2](0)}) 
\end{align*}
so that $\sigma$ defines an automorphism of $A$ (as predicted by \cite[Corollary 4]{Kra12}). 
Consider now the $A$-bimodule $(H\#A)_\sigma$. By \cite[Theorem 7]{Kra12}, the $A$-bimodule $H^{n}(A, \,_{A}A \otimes A_{A})$ is isomorphic to the sub-A-bimodule
\[X = \{ g\# a, a \in A\} \subset (H\#A)_\sigma.\]
It is an immediate verification that $X$ is isomorphic to $_\alpha A _\sigma$, where $\alpha$ is the automorphism of $A$ defined by $\alpha(a) = \textbf{r}(a_{(1)},g)a_{(0)}$, and hence to $A_{\alpha^{-1}\sigma}$. Since $\alpha^{-1}(a) =\textbf{r}(a_{(1)},g^{-1})a_{(0)}$, it is a direct verification to check that $\mu = \alpha^{-1} \sigma$ has the announced form.

\section{Example : two-parameter braided quantum \texorpdfstring{${\rm SL}_2$}{SL2}}\label{sec:twop}

In this section we apply our various results to the coordinate algebra on  the two-parameter braided quantum group SL$_{2}$.

\subsection{Two-parameter braided quantum \texorpdfstring{${\rm SL}_2$}{SL2}}

\bd \label{def:OO}
	Let $p, q \in k^{*}$. The algebra $\OO$ is the algebra presented by generators $a, b, c, d$ subject to the relations
\[ba = qab, ca = pac, db = qbd, dc = pcd, bc=cb\]
\[ad - p^{-1}bc = 1 = da - qbc.\]
\ed

For $p=q$ the algebra $\OO$ is the classical coordinate algebra on the quantum group ${\rm SL}_2$, and has a well-known ordinary Hopf algebra structure. Generalizing this for $p\not=q$, we construct a braided Hopf algebra structure on $\OO$.
For this, the following first piece of structure on $\OO$ is an immediate verification.

\bp\label{alge}
The algebra $\OO$ has a $k\Z$-comodule algebra structure whose coaction is defined by the algebra map
\begin{align*}
\delta: \OO &\longrightarrow \,\OO \otimes k\Z\\
a, \ b, \ c, \ d &\longmapsto a \otimes 1, \ b\otimes z^{-1}, \ c\otimes z, \ d\otimes 1
\end{align*}
where $z$ is a fixed generator of the infinite cyclic group $\mathbb Z$.
\ep

From now on, we denote $A = \OO$, and we work  in the abelian $k$-linear braided category  $\mathcal M^{k\mathbb Z,\xi}$, with $\xi=p^{-1}q$, with its braiding denoted by $c$, see Example \ref{ex:gradedbraided}.

\bp
There exist algebra morphisms
\begin{align*}
\Delta : A &\longrightarrow A \otimes_{c} A\\
 \begin{pmatrix}
a &b\\
c &d
\end{pmatrix} & \longmapsto
 \begin{pmatrix}
a \otimes a + b \otimes c &
a \otimes b + b \otimes d \\
c \otimes a + d \otimes c &
c \otimes b + d \otimes d
\end{pmatrix}
\end{align*}
and
\begin{align*}
\varepsilon: \hspace*{0.3cm}A &\longrightarrow \hspace*{0.3cm}k, &S: \hspace*{0.5cm}A &\longrightarrow \hspace*{0.5cm}A^{{\rm op},c}\\
\begin{pmatrix}
a &b\\
c &d
\end{pmatrix}
&\mapsto
\begin{pmatrix}
1 & 0 \\
0 & 1
\end{pmatrix}
&\begin{pmatrix}
a &b\\
c &d
\end{pmatrix}
&\longmapsto \begin{pmatrix}
d & -qb \\
-p^{-1}c & a
\end{pmatrix}
\end{align*}
that endow $A$ with a Hopf algebra structure in the braided category $\m^{k\Z,\xi}$.
\ep

\bpf
It is immediate to check that $\varepsilon$ is a well-defined algebra map,  and is a morphism in  $\m^{k\Z, \xi}$. Consider now the algebra map
\begin{align*}
\Delta_{0} : k\langle a, b, c, d \rangle &\longrightarrow A \otimes_{c} A\\
\begin{pmatrix}
a &b\\
c &d
\end{pmatrix} & \longmapsto
 \begin{pmatrix}
a \otimes a + b \otimes c &
a \otimes b + b \otimes d \\
c \otimes a + d \otimes c &
c \otimes b + d \otimes d
\end{pmatrix}
\end{align*}
where $A \otimes_{c} A$ is the braided tensor product of the algebra $A$ with itself, see Subsection \ref{sub:braidedHopf}.
In view of the structure of the braiding of $\m^{k\Z,\xi}$, we have, in $A \otimes_{c} A$, for arbitrary elements $x, y \in A$ and  $z, t \in \{a, b, c, d\}$ with $(z, t) \notin \{(b, b), (b, c), (c, b), (c, c)\}$,
\begin{equation*}
    \begin{cases}
      (x \otimes b).(b \otimes y) = p^{-1}qxb \otimes by\\
      (x \otimes b).(c \otimes y) = pq^{-1}xc \otimes by\\
      (x \otimes c).(b \otimes y) = pq^{-1}xb \otimes cy\\
      (x \otimes c).(c \otimes y) = p^{-1}qxc \otimes cy\\
      (x \otimes z).(t \otimes y) = xt \otimes zy
    \end{cases}.
\end{equation*}
We now have, in $A \otimes_{c} A$,
\begin{align*}
\Delta_{0}(ba) &= (a \otimes b + b \otimes d).(a \otimes a + b \otimes c)\\
&=(a \otimes b).(a \otimes a) + (a \otimes b).(b \otimes c)+ (b \otimes d). (a \otimes a) + (b \otimes d). (b \otimes c)\\
 &= a^{2} \otimes ba + p^{-1}qab \otimes bc + ba \otimes da + b^{2} \otimes dc\\
&= qa^{2} \otimes ab + p^{-1}qab \otimes bc + ba \otimes da + b^{2} \otimes dc.
\end{align*}
On the other hand, 
\begin{align*}
\Delta_{0}(ab) &= (a \otimes a + b \otimes c).(a \otimes b + b \otimes d)\\
&=(a \otimes a).(a \otimes b) + (b \otimes c).(a \otimes b)+ (a \otimes a).(b \otimes d)+ (b \otimes c).(b \otimes d)\\
&= a^{2} \otimes ab + ba \otimes cb + ab \otimes ad + pq^{-1}b^{2} \otimes cd\\
&= a^{2} \otimes ab + ba \otimes bc + ab \otimes ad + q^{-1}b^{2} \otimes dc\\
&= a^{2} \otimes ab + ba \otimes bc + q^{-1}ba \otimes (da - qbc + p^{-1}bc) + q^{-1}b^{2} \otimes dc\\
&= a^{2} \otimes ab + q^{-1}ba \otimes da + p^{-1}ab \otimes bc + q^{-1}b^{2} \otimes dc
\end{align*}
and hence $\Delta_{0}(ba) = q\Delta_{0}(ab)$. Next, we also have
\begin{align*}
&\Delta_{0}(ad -p^{-1}bc) = (a \otimes a + b\otimes c).(c \otimes b + d\otimes d) - p^{-1}(a \otimes b + b \otimes d).(c \otimes a + d \otimes c)\\
&= (a \otimes a). (c \otimes b) + (a \otimes a).(d \otimes d) + (b \otimes c).(c \otimes b) + (b \otimes c).(d \otimes d)\\
& \hspace*{0.5cm}- p^{-1}(a \otimes b).(c \otimes a) - p^{-1}(a \otimes b). (d \otimes c) - p^{-1}(b \otimes d). (c \otimes a) - p^{-1}(b \otimes d). (d \otimes c)\\
&= ad \otimes ad + p^{-1}qbc \otimes cb + bd \otimes cd - ad \otimes (ad - 1) - p^{-1} bc \otimes(1 + qbc) - bd \otimes cd\\
&= ad \otimes 1 - p^{-1}bc \otimes 1 = 1 \otimes 1 = \Delta_{0}(1).
\end{align*}
Similar calculations, which we leave to the reader, show that \[\Delta_{0}(ca) = p\Delta_{0}(ac), \,\Delta_{0}(db) = q\Delta_{0}(bd), \,\Delta_{0}(dc) = p\Delta_{0}(cd), \,\Delta_{0}(bc) = \Delta_{0}(cb)\] and $$\Delta_{0}(da - qbc) = \Delta_{0}(1).$$ Therefore we obtain  the announced morphism of algebras $\Delta: A \to A\otimes_{c} A$, which is easily seen to be a morphism in  $\m^{k\Z, \xi}$. 

In $A^{{\rm op}, c}$, we have for $x, y \in \{a, b, c, d\}$ with $(x, y) \notin \{(b, b), (b, c), (c, b), (c, c)\}$,
\[x \cdot y =yx\]
while
\[b\cdot b = p^{-1}q b^2, \ b\cdot c = pq^{-1} bc, \ c\cdot b = pq^{-1}bc, \ c\cdot c = p^{-1}qc^2.\]
From this, it is a direct verification to define the algebra map $S$ as in the statement, and to check that $S$ is  a morphism in  $\m^{k\Z, \xi}$. It is then immediate that $\Delta, \varepsilon, S$ satisfy the Hopf algebra axioms on the generators of $A$, and hence on the whole of $A$ since these are algebra maps.  
\epf

We call the braided Hopf algebra $\OO$ \textsl{the coordinate algebra on the two-parameter braided quantum group ${\rm SL}_2$}.

The following are straightforward generalizations of classical resuls in the case $p=q$.

\bp\label{prop:diamond2}
The set $\{a^{i}b^{j}c^{k} \mid i, j, k \in \N\} \cup \{b^{j}c^{k}d^{l} \mid j, k \in \N, l \in \N^{*}\}$ is a vector space basis of $\OO$.
\ep

\bpf The result is obtained using the diamond lemma,  as in \cite[4.1.5]{KSBook97}.
\epf

\bp \label{domain}
The  algebra $A = \OO$ and its quotients $A /(b)$, $A /(c)$ and $A /(b, c)$ are domains.
\ep

\bpf
It is well-known that $A /(b)$, $A /(c)$ and $A /(b, c)\simeq k \mathbb Z$ are domains. For $A$, we can proceed exactly as in \cite[I.1]{BG02}. We consider first the algebra $A_{p,q}$ presented by generators 
 $a, b, c, d$ and relations
\[ba = qab, \,ca = pac, \,bc = cb, \,
db = qbd, \,dc = pcd, \,ad - da = (p^{-1} - q)bc\]
and remark that $A_{p,q}$ is an iterated Ore extension, hence is a domain.
Put  $D_{p, q} = ad - p^{-1}bc = da - qbc$. We have $D_{p, q} \in Z(A_{p,q})$ and  
$\OO = A_{p,q} /(D_{p, q} - 1)$. Consider then the localization $A_{p,q}[D^{-1}_{p,q}]$ with respect to the central regular element $D_{p, q}$. We then have an algebra isomorphism
\begin{align*}
f: A_{p,q}[D_{p, q}^{-1  }] &\longrightarrow \OO \otimes k[z, z^{-1}]\\
a,b,c,d &\longmapsto a\otimes z, b\otimes z, c\otimes 1, d\otimes 1
\end{align*}
and since $A_{p,q}[D^{-1}_{p,q}]$ is a domain, so is $A=\OO$.
\epf

\subsection{Relation with previous literature}\label{sub:related}
To the best of our knowledge,  Definition \ref{def:OO} seems to be the first formal occurrence of the braided Hopf algebra $\OO$ under this form in the literature. There are however some related known objects, that we briefly mention and discuss in this subsection.  

 First, assume that $k= \mathbb C$ and, for $q\in \mathbb C^*$, consider $C({\rm SU}_q(2))$, the algebra of continuous functions on the braided compact quantum group ${\rm SU}_q(2)$ defined by Kasprzak, Meyer, Roy and Woronowicz in \cite{KMRW16}: $C({\rm SU}_q(2))$ is the universal $C^*$-algebra generated by elements $\alpha$, $\gamma$ satisfying the relations
\[\alpha^*\alpha + \gamma^*\gamma = 1, \ \alpha \alpha^*+ |q|^2 \gamma \gamma^* =1, \
\gamma\gamma^*=\gamma^*\gamma, \ \alpha\gamma = \overline{q}\gamma \alpha, \ \alpha \gamma^* = q \gamma^*\alpha.
\]
For $p=\overline{q}$, it is an immediate verification to check that 
$\mathscr{O}_{\overline{q} , q}({\rm SL}_2(\mathbb C))$ 
has a $*$-algebra structure given by
\[a^*=d, \ b^*=-\overline{q}c, \ c^*=-q^{-1}b, \ d^*=a \]
and that there exists a $*$-algebra map with dense image
\begin{align*}
\mathscr{O}_{\overline{q}^{-1} , q^{-1}}({\rm SL}_2(\mathbb C))  &\longrightarrow C({\rm SU}_q(2)) \\
\begin{pmatrix}
 a & b\\
c & d
\end{pmatrix}
&\longmapsto
\begin{pmatrix}
 \alpha &-q\gamma^*\\
\gamma & \alpha^*
\end{pmatrix}.
\end{align*}
Hence $\mathscr{O}_{\overline{q}^{-1} , q^{-1}}({\rm SL}_2(\mathbb C)) $ might be called the coordinate algebra on the braided compact quantum group ${\rm SU}_q(2)$, and denoted by $\mathscr{O}({\rm SU}_q(2))$.

It is shown in \cite{HN24} that $\mathscr{O}({\rm SU}_q(2))$ can be constructed from the ordinary Hopf $*$-algebra $\mathscr{O}({\rm SU}_{q'}(2))$ for some $q'\in \mathbb R$ via Majid's transmutation operation \cite{Maj93}. Similarly it is also possible to construct $\OO$ from the ordinary Hopf algebra $\mathscr{O}_{q'}({\rm SL}_2(k))$ via transmutation. We will not give the details here, and instead briefly explain how $\OO$ naturally occurs in the setting of 
the more familiar Takeuchi's two-parameter quantum ${\rm GL}_2$  \cite{Tak90}.

Recall \cite{Tak90} that for $p,q\in k^*$, the algebra $\mathscr{O}_{q , p}({\rm GL}_2(k)) $ is presented by generators $a,b,c,d, \delta^{-1}$ and relations
\[ba = qab, \ ca = pac, \ db = pbd, \ dc = qcd, \ pbc=qcb\]
\[da-ad = (p-q^{-1})bc, \ (ad - q^{-1}bc)\delta^{-1} = 1 = \delta^{-1}(ad - q^{-1}bc).\]
To connect $\OO$ and $\mathscr{O}_{q , p}({\rm GL}_2(k)) $,  we use the bosonization construction recalled in Subsection \ref{subsec:bozo}.

\bp
The bosonization $k\mathbb Z \# \OO$ is isomorphic to $\mathscr{O}_{p , q}({\rm GL}_2(k))$.
\ep 

\bpf
It is a straightforward verification to check that there exists an algebra map
\begin{align*}
 f : \mathscr{O}_{p , q}({\rm GL}_2(k)) & \longrightarrow k\mathbb Z \# \OO \\
\begin{pmatrix}
 a & b \\ c & d
\end{pmatrix}, \ \delta^{-1}
& \longmapsto 
\begin{pmatrix}
 z\# a & z \# b \\ 1 \# c & 1\# d
\end{pmatrix}, \ z^{-1}\#1
\end{align*}
which is an isomorphism, and a coalgebra map as well.
\epf

It follows from this result that $\OO$ can be recovered as the subalgebra of coinvariants associated with a Hopf algebra projection on  $\mathscr{O}_{p , q}({\rm GL}_2(k))$, see \cite{Rad85}.

\subsection{Homological properties of \texorpdfstring{$\OO$}{OO}}

In this subsection we compute the cohomological dimension of $A=\OO$, prove that $A$ is smooth and twisted Calabi-Yau. We begin by constructing an suitable resolution for the trivial module, generalizing that found by  Hadfield and Kr\"ahmer \cite{HK05} in the $p=q$ case.

\bp \label{prop:reso}
The following is a resolution of $_\varepsilon k $ by free left $A$-modules:
\begin{center}
\begin{tikzcd}
(P_{\ast}): \hspace*{0.5cm}0 \ar[r] &A \ar[r, "\phi_{3}"] &A^{3} \ar[r, "\phi_{2}"] &A^{3} \ar[r, "\phi_{1}"] &A \ar[r, "\varepsilon"] &k \ar[r] &0
\end{tikzcd}
\end{center}
where $\phi_{1}(x, y, z) = x(a - 1) + yb + zc$, $\phi_{3}(x) = x(c, -b,  pqa -1)$ and
\begin{equation*}
\phi_{2}(x, y, z) = (x, y, z)
\begin{pmatrix}
b & 1 - qa & 0 \\
c & 0 & 1 - pa \\
0 & c & -b
\end{pmatrix}.
\end{equation*}
\ep

\bpf
The maps $\phi_{1}, \phi_{2}, \phi_{3}$ are clearly $A$-linear, and that
$(P_{*})$ is a complex follows from the matrix computations
\[
\begin{pmatrix}
b & 1 - qa & 0 \\
c & 0 & 1 - pa \\
0 & c & -b
\end{pmatrix} 
\begin{pmatrix}
 a-1\\
b\\
c
\end{pmatrix}
= \begin{pmatrix}
   0 \\ 0 \\ 0
  \end{pmatrix},
\quad
 (c , -b , pqa-1)
\begin{pmatrix}
b & 1 - qa & 0 \\
c & 0 & 1 - pa \\
0 & c & -b
\end{pmatrix} =
(0, 0, 0).
\]

The injectivity of $\phi_{3}$ follows from the fact that $A$ is a domain. Moreover $\varepsilon$ is surjective and 
it is a standard verification that ${\rm Ker}(\varepsilon)$ is generated as a left $A$-module by $a-1$, $b$ and $c$, so ${\rm Ker}(\varepsilon) = \text{Im}(\phi_{1})$.

Let $X = (x, y, z) \in \text{Ker}(\phi_{1})$, we have $x(a - 1) + yb + zc = 0$
and hence $x(a - 1) = 0$ in the domain $A/(b, c)$, so that $x = 0$ in $A/(b, c)$. By using the relations that define $A$, we see that $bA = A b$ and $cA = A c$, hence we can write $x = \alpha b + \beta c$ for some $\alpha, \beta \in A$. Then we have 
\begin{align*} 
(x, y, z) - \phi_{2}(\alpha, \beta, 0) &= (x, y, z) - (\alpha b + \beta c, \alpha(1 - qa), \beta(1 - pa))\\
&= (0, y - \alpha(1 - qa), z - \beta(1 - pa)).
\end{align*}
Thus, to show that $X \in \text{Im}(\phi_{2})$, we can assume that $X = (0, y, z)$ for some $y, z \in A$. Then we have $yb + zc = 0$ which gives $yb = 0$ in the domain $A/(c)$. Hence, $y = \gamma c$ and $z = - \gamma b$ for some $\gamma \in A$. It follows that
$X = (0, \gamma c, -\gamma b) = \phi_{2}(0, 0, \gamma)$, and therefore ${\rm Ker}(\phi_{1}) = {\rm Im}(\phi_{2})$.

Let $X = (x, y, z) \in \text{Ker}(\phi_{2})$. Then 
\begin{equation}
    \begin{cases}
      xb + yc = 0\\
      x(1 - qa) + zc = 0\\
      y(1 - pa) -zb = 0
    \end{cases}\,
\end{equation}
and $x(1 - qa) = 0$ in the domain $A /(c)$, hence $x = 0$ in $A /(c)$,  which implies $x = x'c$ for $x' \in A$. Since $xb + yc = 0$ and $A$ is a domain, we obtain $y = -x'b$. We have now
  \[  \begin{cases}
      x'c(1 - qa) + zc = 0\\
      -x'b(1 - pa) -zb = 0
    \end{cases}\,\]
from which it follows that $z = x'(pqa - 1)$ since $A$ is a domain.  We get $X = x'(c, - b, pqa -1) = \phi_{3}(x')$.  Therefore, Ker($\phi_{2}$) $=$ Im($\phi_{3}$), and we conclude that the sequence $(P_{*})$ is exact.
\epf

We now use the previous resolution to compute some $\ext$ spaces. 
For $t \in k^*$, it is straightforward to check that there exists an algebra map
	\begin{align*}
		\varepsilon_{t} : A &\longrightarrow k\\
		\begin{pmatrix}
			a & b  \\
			c & d
		\end{pmatrix}
		&\longmapsto 
		\begin{pmatrix}
			t & 0  \\
			0 & t^{-1}
		\end{pmatrix}
	\end{align*}
with $\varepsilon_1=\varepsilon$.

\bp\label{prop:extcomp}
For $p,q\in k^*$, put $t=(pq)^{-1}$. The vector spaces $\ext^*_A(_\varepsilon k, {_{\varepsilon_t}k})$ are described by the following table.
\begin{center}
\begin{tabular}{|c|c|c|c|c|}
\hline
& $pq\not=1$, $p \not=1$, $q\not=1$ & $pq\not=1$, $1\in \{p,q\}$ & $pq=1$, $p\not=1$, $q\not=1$ & $p=q=1$ \\ 
\hline
$\ext^0_A(_\varepsilon k, {_{\varepsilon_t}k})$ & $\{0\}$ & $\{0\}$ & $k$ & $k$ \\
\hline
$\ext^1_A(_\varepsilon k, {_{\varepsilon_t}k})$ & $\{0\}$ & $k$ & $k$ & $k^3$\\
\hline
$\ext^2_A(_\varepsilon k, {_{\varepsilon_t}k})$ & $k$     & $k^2$ & $k$ & $k^3$ \\
\hline
$\ext^3_A(_\varepsilon k, {_{\varepsilon_t}k})$ & $k$ & $k$ & $k$ & $k$ \\
\hline
\end{tabular}
 \end{center}
Moreover, for any algebra map $\psi : A \to k$ with $\psi\not=\varepsilon_t$, we have $\ext^3_A(_\varepsilon k, {_\psi  k})=\{0\}$. 
\ep 

\bpf
Applying $\Hom_A(-, {_\psi k})$ to the resolution in Proposition \ref{prop:reso}, we see, after some standard identifications, that  $\ext^*_A(_\varepsilon k, {_\psi  k})$  is the cohomology of the complex
\[
 0 \to k \overset{f_1}\longrightarrow k^3 \overset{f_2}\longrightarrow k^3 \overset{f_3}\longrightarrow k \to 0
\]
where $f_1(x)=((\psi(a)-1)x, \psi(b)x,\psi(c)x)$, $f_3(x,y,z) = x\psi(c)-y\psi(b) + (pq\psi(a)-1)z$, and 
\begin{equation*}
f_{2}(x, y, z) = (x, y, z)
\begin{pmatrix}
\psi(b) & \psi(c) & 0 \\
1-q\psi(a) & 0 & \psi(c) \\
0 & 1-p\psi(a) & -\psi(b)
\end{pmatrix}.
\end{equation*}
We thus see that if $\psi(b)\not=0$ or $\psi(c)\not=0$ or $\psi(a)\not=(pq)^{-1}$, we have $\ext^3_A(_\varepsilon k, {_\psi  k})=\{0\}$. Otherwise $\psi=\varepsilon_t$, and the announced results are immediate.
\epf

\bc\label{cor:cd}
We have $\cd(\OO)=3$ for any $p,q\in k^*$.
\ec

\bpf
We have $\pd_A(_\varepsilon k)\leq 3$ by Proposition \ref{prop:reso} and $\pd_A(_\varepsilon k)\geq 3$ by Proposition \ref{prop:extcomp}, hence $\pd_A(_\varepsilon k)= 3$. We conclude by Theorem \ref{thm:equaldimen}.
\epf

We now want to prove that $A=\OO$ is smooth using Theorem \ref{thm:smoothness}. For this we interpret the resolution in Proposition \ref{prop:reso} as a resolution in $_A\mathcal M^{k\mathbb Z}$.

\bp\label{prop:resoequi}
Let $V$, $W$ be the $3$-dimensional $k\mathbb Z$-comodules with  respective bases  $(e_{1}, e_{2}, e_{3})$ and $(e'_{1}, e'_{2}, e'_{3})$, and coactions defined by 
\begin{align*}
\delta_{V}: V &\longrightarrow V \otimes k\Z &\delta_{W}: W &\longrightarrow W \otimes k\Z\\
e_{1},e_2, e_3 &\longmapsto e_{1} \otimes 1, e_2\otimes z^{-1}, e_3\otimes z &e'_{1},e'_2, e'_3 &\longmapsto e'_{1}\otimes z^{-1}, e'_2\otimes z, e'_3\otimes 1.
\end{align*}
Then we have a resolution of $_\varepsilon k$ by free $A$-modules in $\MM^{k\Z}$
\[0 \to A \overset{\phi_3}\longrightarrow A\otimes W \overset{\phi_2}\longrightarrow A \otimes V \overset{\phi_1}\longrightarrow A \overset{\varepsilon}\to k \to0\]
where 
\begin{align*}
\phi_{1}&(x \otimes e_{1} + y \otimes e_{2} + z \otimes e_{3}) = x(a - 1) + yb + zc \\
\phi_{3}&(x) = xc \otimes e'_{1} - xb \otimes e'_{2} + x(pqa - 1) \otimes e'_{3}\\
\phi_{2}&(x \otimes e'_{1}) = xb \otimes e_{1} + x(1 - qa) \otimes e_{2}, \quad
\phi_{2}(x \otimes e'_{2}) = xc \otimes e_{1} + x(1 - pa) \otimes e_{3}\\
\phi_{2}&(x \otimes e'_{3}) = xc \otimes e_{2} - xb \otimes e_{3}.
\end{align*}
In particular $_\varepsilon k$ is of type FP in $_A\MM^{k\Z}$.
\ep

\bpf
After the obvious identifications between $A\otimes V$ and $A^3$ and between $A\otimes W$ and $A^3$, the above sequence is the same as the one in Proposition \ref{prop:reso}, and hence is exact. It is also an immediate verification that the above maps are $k\mathbb Z$-colinear.
\epf

\sloppy
Combining Corollary \ref{cor:cd}, Theorem \ref{thm:smoothness} and Proposition  \ref{prop:resoequi}, we  obtain that $\OO$ is smooth, with $\cd(\OO)=3$.

We can also use Proposition \ref{prop:resoequi} to compute some $\ext$ spaces in $_A\mathcal M^{k\mathbb Z}$.

\bp \label{prop:extequi}
For $p,q\in k^*$, put $t=(pq)^{-1}$.  We have
\[\ext^i_{_A\mathcal M^{k\mathbb Z}}(_\varepsilon k, {_{\varepsilon_t}k}) \simeq
\begin{cases}
	k \ \text{if} \ i \in \{2,3\}, \ \text{or} \ i \in \{0,1\} \ \text{and} \ pq=1\\
	0 \ \text{otherwise}.
\end{cases},
\]
\ep

\bpf
Let $\psi : A \to k$ be an algebra map that is also a map of $k\mathbb Z$-comodules.
Applying $\Hom_{_A\mathcal M^{k\mathbb Z}}(-, {_\psi k})$ to the resolution in Proposition \ref{prop:resoequi} we see, after some standard identifications, that  $\ext^*_{_A\mathcal M^{k\mathbb Z}}(_\varepsilon k, {_\psi  k})$  is the cohomology of the complex
\[
 0 \to k \overset{f_1}\longrightarrow k \overset{f_2}\longrightarrow k \overset{f_3}\longrightarrow k \to 0
\]
where $f_1(x)=(\psi(a)-1)x$, $f_2(x)=0$ and $f_3(x) =  (pq\psi(a)-1)x$. The announced result is then a direct verification.
\epf

Our next aim is to prove that $\OO$ is a twisted Calabi-Yau algebra.

\bp \label{prop:ext(k,A)}
We have $\e^{n}_{A}(_{\varepsilon}k, A) = 0$ if $n\neq 3$, and  $\e_{A}^{3}(_{\varepsilon}k, A) \simeq k_{\varepsilon_{(pq)^{-1}}}$ as right $A$-modules.
\ep

\bpf
Applying the functor $\text{Hom}_{A}(-, A)$ to the resolution $(P_{*})$ in Proposition \ref{prop:reso} and using some standard isomorphisms, we obtain the complex of right $A$-modules $\text{Hom}_{A}(P_{*}, A)$: 
\begin{center}
\begin{tikzcd}
0 \ar[r] &A \ar[r, "\phi^{*}_{1}"] &A^{3} \ar[r, "\phi^{*}_{2}"] &A^{3} \ar[r, "\phi^{*}_{3}"] &A \ar[r] &0
\end{tikzcd}
\end{center}
where 
$\phi^{*}_{1}(x) = \big(a - 1, b, c\big)x; \, \phi^{*}_{3}(x, y, z) = cx - by + (pqa - 1)z$ and \begin{equation*}
\phi^{*}_{2}(x, y, z) =
\left[\begin{pmatrix}
b & 1-qa & 0 \\
c & 0 & 1 -pa \\
0 & c & -b
\end{pmatrix}\begin{pmatrix}
x\\
y\\
z
\end{pmatrix}\right]^{t}.
\end{equation*}
We have $\e_{A}^{n}(_{\varepsilon}k, A)\simeq \text{Ker}(\phi^{*}_{n + 1}) / \text{Im}(\phi^{*}_{n})$ and
verifications are straightforward using the same arguments as in Proposition \ref{prop:reso}, especially the fact that $A, A /(b), A/(c)$ and $A /(b, c)$ are domains. We leave the verifications to the reader, except for the degree $3$ where we want to obtain explicitely the $A$-module structure. 

Let $x \in A$, then $x$ is a linear combination of elements of the forms $a^{i}b^{j}c^{k}$ and $b^{j}c^{k}d^{\ell}$. In $A / \text{Im}(\phi^{*}_{3})$, we have $b = c = 0$, $a = (pq)^{-1}1$ and $d = (pq)1$. Hence (the class of) $1$ generates $A / \text{Im}(\phi^{*}_{3})$, and we have to check that
 $1 \notin \text{Im}(\phi^{*}_{3})$. If $1 \in \text{Im}(\phi^{*}_{3})$, there exists $(x, y, z) \in A^{3}$ such that $1 = cx - by +(pqa - 1)z$, and we have $1 = (pqa - 1)z$ in the Laurent polynomial ring $A /(b,c) = k[a, a^{-1}]$, which is impossible. Thus $1 \notin \text{Im}(\phi^{*}_{3})$ and hence, $\e_{A}^{3}(_{\varepsilon}k, A) \simeq k$ as vector spaces, and it is clear from the previous relations that
$\e_{A}^{3}(_{\varepsilon}k, A)\simeq k_{\varepsilon_{(pq)^{-1}}}$ as right $A$-modules.
\epf

\begin{theorem}\label{thm:OOtcy}
 The algebra $\OO$ is twisted Calabi-Yau of dimension $3$, with Nakayama automorphism defined by
\begin{align*}
\mu : \quad \OO \quad &\longrightarrow \OO \\
\begin{pmatrix}
a & b\\
c & d
\end{pmatrix} &\longmapsto
\begin{pmatrix}
(pq)^{-1}a &  b\\
c & (pq) d
\end{pmatrix}.
\end{align*}
\end{theorem}

\begin{proof}
Proposition \ref{prop:resoequi} and Proposition \ref{prop:ext(k,A)} enable us to use Theorem \ref{thm:tcy} to conclude that $\OO$ is twisted Calabi-Yau of dimension $3$. Moreover the form of the resolution in Proposition \ref{prop:resoequi} gives, by Lemma \ref{lem:gp-trivial}, that the group-like $g$ occurring in Theorem \ref{thm:tcy} is trivial. Hence the Nakayama automorphism is given by
 \[\mu(x) = S_{A}^{2}(x_{[2](0)}) \varepsilon_{(pq)^{-1}}(x_{[1]})\, \textbf{r}\big[x_{[2](1)}, S_{k\Z}(x_{[2](2)})\big]\]
and the computation of its values on the generators is immediate. 
\end{proof}

We finish the paper by recording some Hochschild cohomology computations for $\OO$ when the bimodule of coefficients is one dimensional. This has some interest in connection with the probabilistic questions studied in \cite{FGT15}. We begin with a general observation.

\bp
Let $A$ be a Hopf algebra in the braided category $\MM^{H}$ of comodules over a coquasitriangular  Hopf algebra $H$. 
Let $M$ be a left $A$-module, and endow $M$ with the trivial right $A$-module structure.
Then we have
\[H^*(A,M_\varepsilon) \simeq \e_A^*({_\varepsilon k}, M).\]
\ep

\bpf
Start with an $A$-bimodule $M$ and recall the isomorphisms
\[H^*(A, M) \simeq \ext^{*}_{\,_{A}\MM^{H}}(_{\varepsilon}k, \stackon[-8pt]{$M \odot H$}{\vstretch{1.5}{\hstretch{1.8}{\widetilde{\phantom{\;\;\;\;\;\;\;\;}}}}}) \]
from Remark \ref{rem:ext}. The left $A$-module structure on $\widetilde{M\odot H}$ is given by
\[
 a\cdot (x\otimes h) = \normalfont \textbf{r}\left(a_{[2](2)}, h_{(2)}\right) a_{[1](0)}.x.S_A(a_{[2](0)}) \otimes a_{[1](1)}h_{(1)}a_{[2](1)}
\]
and if we assume that the right $A$-module structure on $M$ is trivial this gives
\[
 a\cdot(x\otimes h)= a_{[1](0)}.x \otimes a_{[1](1)}h.
\]
Hence we have $\widetilde{M_\varepsilon \odot H} = M\boxdot H$. We conclude by Proposition \ref{prop:pdHopfmod}. 
\epf

It follows that the Hochschild cohomology spaces $H^*(\OO, {_\alpha k_\varepsilon})$, for any algebra map $\alpha : \OO \to k$, can now be computed using Proposition \ref{prop:reso}. In particular, still in connection with \cite{FGT15}, we notice that 
\[H^2(\OO, {_\varepsilon k_\varepsilon})\simeq 
\begin{cases}
 \{0\} \ &\text{if $pq\not=1$, $p\not=1$, $q\not=1$,} \\
k \ & \text{if $pq=1$, $p\not=1$, $q\not=1$, or $1\in\{p,q\}$ and $pq\not=1$,} \\
k^3 \ & \text{if $p=1=q$.}
\end{cases} 
\]

\bibliographystyle{plain}
\bibliography{biblio.bib}

\begin{thebibliography}{10}

\bibitem{adams1967adjoint}
W.~Adams and M.~A. Rieffel.
\newblock Adjoint functors and derived functors with an application to the
  cohomology of semigroups.
\newblock {\em J. Algebra}, 7(1):25--34, 1967.

\bibitem{ardizzoni2004category}
A.~Ardizzoni.
\newblock The category of modules over a monoidal category: Abelian or not?
\newblock {\em Annali dell’Universit{\`a} di Ferrara}, 50:167--185, 2004.

\bibitem{AMS07}
A.~Ardizzoni, C.~Menini, and D.~\c{S}tefan.
\newblock Hochschild cohomology and ``smoothness'' in monoidal categories.
\newblock {\em J. Pure Appl. Algebra}, 208(1):297--330, 2007.

\bibitem{Baez94}
J.~C. Baez.
\newblock Hochschild homology in a braided tensor category.
\newblock {\em Trans. Amer. Math. Soc.}, 344(2):885--906, 1994.

\bibitem{Ber05}
R.~Berger.
\newblock Dimension de {H}ochschild des alg\`ebres gradu\'{e}es.
\newblock {\em C. R. Math. Acad. Sci. Paris}, 341(10):597--600, 2005.

\bibitem{bichon2022monoidal}
J.~Bichon.
\newblock On the monoidal invariance of the cohomological dimension of {H}opf
  algebras.
\newblock {\em C. R. Math. Acad. Sci. Paris}, 360:561--582, 2022.

\bibitem{BG02}
K.~A. Brown and K.~R. Goodearl.
\newblock {\em Lectures on algebraic quantum groups}.
\newblock Advanced Courses in Mathematics. CRM Barcelona. Birkh\"{a}user
  Verlag, Basel, 2002.

\bibitem{BZ08}
K.~A. Brown and J.~J. Zhang.
\newblock Dualising complexes and twisted {H}ochschild (co)homology for
  {N}oetherian {H}opf algebras.
\newblock {\em J. Algebra}, 320(5):1814--1850, 2008.

\bibitem{brown}
K.~S. Brown.
\newblock {\em Cohomology of groups}.
\newblock Graduate Texts in Mathematics 87. Springer-Verlag, New York-Berlin,
  1982.

\bibitem{bulacu2019quasi}
D.~Bulacu, S.~Caenepeel, F.~Panaite, and F.~Van~Oystaeyen.
\newblock {\em Quasi-{H}opf algebras. A categorical approach}, volume 171 of
  {\em Encyclopedia of Mathematics and its Applications}.
\newblock Cambridge University Press, Cambridge, 2019.

\bibitem{CMIZ99}
S.~Caenepeel, G.~Militaru, B.~Ion, and S.~Zhu.
\newblock Separable functors for the category of {D}oi-{H}opf modules,
  applications.
\newblock {\em Adv. Math.}, 145(2):239--290, 1999.

\bibitem{caenepeel2004frobenius}
S.~Caenepeel, G.~Militaru, and S.~Zhu.
\newblock {\em Frobenius and separable functors for generalized module
  categories and nonlinear equations}, volume 1787 of {\em Lecture Notes in
  Mathematics}.
\newblock Springer-Verlag, Berlin, 2002.

\bibitem{CE56}
H.~Cartan and S.~Eilenberg.
\newblock {\em Homological algebra}.
\newblock Princeton University Press, Princeton, NJ, 1956.

\bibitem{EGNO15}
P.~Etingof, S.~Gelaki, D.~Nikshych, and V.~Ostrik.
\newblock {\em Tensor categories}, volume 205 of {\em Mathematical Surveys and
  Monographs}.
\newblock American Mathematical Society, Providence, RI, 2015.

\bibitem{FGT15}
U.~Franz, M.~Gerhold, and A.~Thom.
\newblock On the {L}\'evy-{K}hinchin decomposition of generating functionals.
\newblock {\em Commun. Stoch. Anal.}, 9(4):529--544, 2015.

\bibitem{GK93}
V.~Ginzburg and S.~Kumar.
\newblock Cohomology of quantum groups at roots of unity.
\newblock {\em Duke Math. J.}, 69(1):179--198, 1993.

\bibitem{HN24}
E.~Habbestad and S.~Neshveyev.
\newblock Cocycle twisting of semidirect products and transmutation.
\newblock {\em Int. Math. Res. Not. IMRN}, (11):9142--9164, 2024.

\bibitem{HK05}
T.~Hadfield and U.~Kr\"{a}hmer.
\newblock Twisted homology of quantum {${\rm SL}(2)$}.
\newblock {\em $K$-Theory}, 34(4):327--360, 2005.

\bibitem{HK09}
T.~Hadfield and U.~Kr\"{a}hmer.
\newblock Braided homology of quantum groups.
\newblock {\em J. K-Theory}, 4(2):299--332, 2009.

\bibitem{heckenberger2020hopf}
I.~Heckenberger and H.-J. Schneider.
\newblock {\em Hopf algebras and root systems}, volume 247 of {\em Mathematical
  Surveys and Monographs}.
\newblock American Mathematical Society, Providence, RI, 2020.

\bibitem{KMRW16}
P.~Kasprzak, R.~Meyer, S.~Roy, and S.L. Woronowicz.
\newblock Braided quantum {$\rm SU(2)$} groups.
\newblock {\em J. Noncommut. Geom.}, 10(4):1611--1625, 2016.

\bibitem{KSBook97}
A.~Klimyk and K.~Schm\"{u}dgen.
\newblock {\em Quantum groups and their representations}.
\newblock Texts and Monographs in Physics. Springer-Verlag, 1997.

\bibitem{Kra12}
U.~Kr\"ahmer.
\newblock On the {H}ochschild (co)homology of quantum homogeneous spaces.
\newblock {\em Israel J. Math.}, 189:237--266, 2012.

\bibitem{LWW14}
L.~Liu, S.~Wang, and Q.~Wu.
\newblock Twisted {C}alabi-{Y}au property of {O}re extensions.
\newblock {\em J. Noncommut. Geom.}, 8(2):587--609, 2014.

\bibitem{Maj93}
S.~Majid.
\newblock Braided groups.
\newblock {\em J. Pure Appl. Algebra}, 86(2):187--221, 1993.

\bibitem{Maj94}
S.~Majid.
\newblock Cross products by braided groups and bosonization.
\newblock {\em J. Algebra}, 163(1):165--190, 1994.

\bibitem{na1989separable}
C.~Nastasescu, M.~Van~den Bergh, and F.~Van~Oystaeyen.
\newblock Separable functors applied to graded rings.
\newblock {\em J. Algebra}, 123(2):397--413, 1989.

\bibitem{Rad85}
D.~E. Radford.
\newblock The structure of {H}opf algebras with a projection.
\newblock {\em J. Algebra}, 92(2):322--347, 1985.

\bibitem{rin62}
G.~S. Rinehart.
\newblock Note on the global dimension of a certain ring.
\newblock {\em Proc. Amer. Math. Soc.}, 13:341--346, 1962.

\bibitem{Sri}
R.~Sridharan.
\newblock Filtered algebras and representations of {L}ie algebras.
\newblock {\em Trans. Amer. Math. Soc.}, 100:530--550, 1961.

\bibitem{Tak90}
M.~Takeuchi.
\newblock A two-parameter quantization of {${\rm GL}(n)$} (summary).
\newblock {\em Proc. Japan Acad. Ser. A Math. Sci.}, 66(5):112--114, 1990.

\bibitem{VDB}
M.~Van~den Bergh.
\newblock Erratum to: ``{A} relation between {H}ochschild homology and
  cohomology for {G}orenstein rings'' [{P}roc. {A}mer. {M}ath. {S}oc. {\bf 126}
  (1998), no. 5, 1345--1348; ].
\newblock {\em Proc. Amer. Math. Soc.}, 130(9):2809--2810, 2002.

\bibitem{WYZ17}
X.~Wang, X.~Yu, and Y.~Zhang.
\newblock Calabi-{Y}au property under monoidal {M}orita-{T}akeuchi equivalence.
\newblock {\em Pacific J. Math.}, 290(2):481--510, 2017.

\bibitem{weibel1994introduction}
C.~A. Weibel.
\newblock {\em An introduction to homological algebra}, volume~38 of {\em
  Cambridge Studies in Advanced Mathematics}.
\newblock Cambridge University Press, Cambridge, 1994.

\end{thebibliography}

\end{document}